\documentclass[12pt,reqno]{amsart}
\usepackage[margin=1.1in]{geometry}
\usepackage[notref,notcite]{showkeys}
\usepackage{hyperref}

\makeatletter
\def\@tocline#1#2#3#4#5#6#7{\relax
  \ifnum #1>\c@tocdepth 
  \else
    \par \addpenalty\@secpenalty\addvspace{#2}%
    \begingroup \hyphenpenalty\@M
    \@ifempty{#4}{%
      \@tempdima\csname r@tocindent\number#1\endcsname\relax
    }{%
      \@tempdima#4\relax
    }%
    \parindent\z@ \leftskip#3\relax \advance\leftskip\@tempdima\relax
    \rightskip\@pnumwidth plus4em \parfillskip-\@pnumwidth
    #5\leavevmode\hskip-\@tempdima
      \ifcase #1
       \or\or \hskip 1em \or \hskip 2em \else \hskip 3em \fi%
      #6\nobreak\relax
    \dotfill\hbox to\@pnumwidth{\@tocpagenum{#7}}\par
    \nobreak
    \endgroup
  \fi}
\makeatother

\usepackage{tikz,amsthm,amsmath,amstext,amssymb,epsfig,euscript, mathrsfs, dsfont,pspicture,multicol,graphpap,graphics,graphicx,times,enumerate,subfig,sidecap,wrapfig,color}

\usepackage{enumerate}

 \numberwithin{equation}{section}

\def\bB{{\mathbb{B}}}
\def\bC{{\mathbb{C}}}

\def\bR{{\mathbb{R}}}
\def\bS{{\mathbb{S}}}

\def\bH{{\mathbb{H}}}

\def\bZ{{\mathbb{Z}}}

\def\cB{{\mathscr{B}}}

\def\cF{{\mathscr{F}}}
\def\cG{{\mathscr{G}}}
\def\cH{{\mathscr{H}}}
\def\cI{{\mathscr{I}}}

\def\cL{{\mathscr{L}}}
\def\cM{{\mathscr{M}}}

\def\cU{{\mathscr{U}}}

\def\one{\mathds{1}}

\def\ve{\varepsilon}

\renewcommand{\d}{{\partial}}
\def\lec{\lesssim}
\def\gec{\gtrsim}

\def\diam{\mathop\mathrm{diam}}

\def\dist{\mathop\mathrm{dist}}

\def\supp{\mathop\mathrm{supp}}

\def\esssup{\mathop\mathrm{esssup}\;}

\newcommand{\ps}[1]{\left( #1 \right)}

\newcommand{\ck}[1]{\left\{#1 \right\}}
\newcommand{\av}[1]{\left| #1 \right|}
\newcommand{\ip}[1]{\left\langle #1 \right\rangle}
\newcommand{\floor}[1]{\left\lfloor #1 \right\rfloor}
\newcommand{\ceil}[1]{\left\lceil #1 \right\rceil}
\newcommand{\cnj}[1]{\overline{#1}}

\def\Xint#1{\mathchoice
{\XXint\displaystyle\textstyle{#1}}%
{\XXint\textstyle\scriptstyle{#1}}%
{\XXint\scriptstyle\scriptscriptstyle{#1}}%
{\XXint\scriptscriptstyle\scriptscriptstyle{#1}}%
\!\int}
\def\XXint#1#2#3{{\setbox0=\hbox{$#1{#2#3}{\int}$ }
\vcenter{\hbox{$#2#3$ }}\kern-.58\wd0}}

\def\avint{\Xint-}

\def\grad{\nabla}

\theoremstyle{plain}
\newtheorem{theorem}{Theorem}

\newtheorem{corollary}[theorem]{Corollary}

\newtheorem{lemma}[theorem]{Lemma}

\newtheorem{proposition}[theorem]{Proposition}

\theoremstyle{definition}

\newtheorem{definition}[theorem]{Definition}

\numberwithin{equation}{section}
\numberwithin{theorem}{section}

\newcommand\eqn[1]{\eqref{e:#1}}

\newcommand\Theorem[1]{Theorem \ref{t:#1}}
\newcommand\Lemma[1]{Lemma \ref{l:#1}}
\newcommand\Corollary[1]{Corollary \ref{c:#1}}
\newcommand\Definition[1]{Definition \ref{d:#1}}
\newcommand\Proposition[1]{Proposition \ref{p:#1}}

\begin{document}

  \def\tab{\qquad}

\title[Bi-Lipschitz parts of quasisymmetric mappings]{Bi-Lipschitz parts of Quasisymmetric mappings}
\author{Jonas Azzam}
\address{Departament de Matem\`atiques\\ Universitat Aut\`onoma de Barcelona \\ Edifici C Facultat de Ci\`encies\\
08193 Bellaterra (Barcelona) }
\email{jazzam "at" mat.uab.cat}
\thanks{Supported in part by the grants RTG DMS 08-38212  and DMS-0856687}
\begin{abstract}
A natural quantity
that measures how well a map $f:\mathbb{R}^{d}\rightarrow
\mathbb{R}^{D}$ is approximated by an affine transformation is
\[\omega_{f}(x,r)=\inf_{A}\ps{\avint_{B(x,r)}\ps{\frac{|f-A|}{|A'|r}}^{2}}^{\frac{1}{2}},\] 
where the infimum ranges over all non-zero affine transformations $A$. This is
natural insofar as it is invariant under rescaling $f$ in
either its domain or image. We show that if
$f:\mathbb{R}^{d}\rightarrow \mathbb{R}^{D}$ is quasisymmetric
and its image has a sufficient amount of rectifiable structure
(although not necessarily $\mathcal{H}^{d}$-finite), then
$\omega_{f}(x,r)^{2}\frac{dxdr}{r}$ is a Carleson measure on
$\mathbb{R}^{d}\times(0,\infty)$. Moreover, this is an
equivalence: if this is a Carleson measure, then, in every ball $B(x,r)\subseteq \mathbb{R}^{d}$, there is
a set $E$ occupying 90$\%$ of $B(x,r)$, say, upon which $f$ is
bi-Lipschitz (and hence guaranteeing rectifiable pieces in the
image).

En route, we make a minor adjustment to a theorem of Semmes to show that quasisymmetric maps of subsets of $\bR^{d}$ into $\bR^{d}$ are bi-Lipschitz on a large subset quantitatively.
\end{abstract}
 
\maketitle
  
  \tableofcontents

\section{Introduction}
\subsection{Background}

Recall that a non-constant map $f:\bR^{d}\rightarrow \bR^{D}$ is {\it $\eta$-quasisymmetric} if there is an increasing homeomorphism $\eta:(0,\infty)\rightarrow(0,\infty)$ such that, for all $x,y,z\in \bR^{d}$ distinct,
\[\frac{|f(x)-f(y)|}{|f(x)-f(z)|}\leq \eta\ps{\frac{|x-y|}{|x-z|}}.\]

The goal of this manuscript is to determine when one can detect or guarantee that a quasisymmetric embedding is bi-Lipschitz on some portion of its support.

Recall that a subset of $\bR^{D}$ is {\it $d$-rectifiable} if it may be covered up to a set of Hausdorff $d$-dimensional measure zero by Lipschitz images of $\bR^{d}$. In general, the image of a quasisymmetric map can be highly irregular. One example can be obtained as follows: by Assouad's theorem \cite{NN12}, for $\alpha\in (0,1)$ and $d\geq 1$, there are $L=L(d,\alpha)$, $D=D(d)$ and an  $L$-bi-Lipschitz mapping of $\bR^{d}$ equipped with the metric $d(x,y)=|x-y|^{\alpha}$ into $\bR^{D}$. Such a map can easily be checked to be quasisymmetric, and one can show that the image of such a map is purely $k$-unrectifiable for any $k=1,2,...,d$, in the sense that the image has Hausdorff $k$-measure zero intersection with any Lipschitz image of $\bR^{k}$. The dimension $D$ depends on $d$ and can be quite larger, but see also \cite{Bishop99} or David and Toro \cite{David-Toro-snowballs} for particular ``snowflake" embeddings of $\bR^{d}$ into $\bR^{1}$. In light of these examples, a priori conditions that rule out such examples is a natural question.

Most results in this vein typically deal with a codimension 1 situation. Speciifcally, they deal with functions that are restrictions of a globally defined quasiconformal map $f:\bR^{d}\rightarrow\bR^{d}$, $d\geq 2$, and give conditions that guarantee $f(\bS^{d-1})$ is $(d-1)$-rectifiable. Before discussing these results, we recall the definition of quasiconformality. For $x\in \bR^{d}$, define
\[K_{f}(x)=\max\left\{\frac{|Df(x)|^{d}}{J_{f}(x)},\frac{J_{f}(x)}{\inf_{y\in\bS^{d-1}}|Df(x) y|^{d}}\right\}.\]
For a domain $\Omega\subseteq \bR^{d}$, a map $f:\Omega \rightarrow \bR^{d}$ that is a homeomorphism onto its image with $f\in W^{1,d}_{\mbox{loc}}(\Omega)$ and $||K_{f}(x)||_{L^{\infty}(\Omega)}\leq K<\infty$ is said to be {\it $K$-quasiconformal}. A surjective $K$-quasiconformal map $f:\bR^{d}\rightarrow \bR^{d}$ is $K$-quasiconformal if and only if it is $\eta$ quasisymmetric, where $K$ and $\eta$ depend on each other (c.f. \cite{Vaisala}). Set $A_{t}=\{x\in \bR^{d}: 1-t<|x|<1+t\}$, $\tilde{K}_{f}(t)=\esssup\{K_{f}(x): x\in A_{t}\}-1$. The smaller this quantity is, the closer $f$ is to being conformal in the $t$-neighborhood of $\bS^{d-1}$.

In \cite{ABL88}, it is shown that if $d=2$, $f|_{\bB}$ is conformal, and $\int_{0}^{1}\tilde{K}_{f}(t)^{2}\frac{dt}{t}<\infty$, then $f(\bS)$ is rectifiable. This was subsequently generalized to higher dimensions (although with a stronger condition on the integral) in \cite{MV90}, where it is shown that for $f:\bR^{d}\rightarrow \bR^{d}$, $d\geq 2$, if $\int_{0}^{1}\tilde{K}_{f}(t)\frac{dt}{t}<\infty$, then $f(\bS^{d-1})$ is rectifiable. By the recent results in \cite{BGRT12}, it is only necessary that $\int_{0}^{1}\ps{\tilde{K}_{f}(t)\log\frac{1}{\tilde{K}_{f}(t)}}^{2}\frac{dt}{t}<\infty$. They derive this result from a similar result involving not the quasiconformal dilatation, but the quasisymmetry: if 
\[\tilde{H}_{f}(t)=\sup\ck{\frac{|f(x)-f(y)|}{|f(x)-f(z)|}: x,y,z\in A_{t} \mbox{ are distinct and }|x-z|\leq |x-y|},\]
then \cite{BGRT12} also shows that $\int_{0}^{1} \tilde{H}_{f}(t)^{2}\frac{dt}{t}<\infty$ implies $f(\bS^{d-1})$ is rectifiable. 

Reverse implications with these quantities are not possible, as the conditions are too stringent: most quasiconformal mappings with $f(\bS^{d-1})$ rectifiable don't have $\lim_{t\rightarrow 0}K_{f}(t)=0$. Moreover, a result due to Astala, Zinsmeister, and MacManus seems to suggest that loosening these conditions will result in only partial rectifiability of the image. Before stating this result, we review some terminology. 

Recall that a {\it bounded $C$-chord-arc domain} $\cU\subseteq\bC$ is a scaled copy of a $C$-bi-Lipschitz image of the unit ball $\bB$, and a {\it $K$-quasidisk} is any image of the ball under a $K$-quasiconformal mapping $f:\bC\rightarrow\bC$. A {\it Bishop-Jones domain} $\Omega\subseteq \bC$ is a simply connected domain where, for all $z\in \Omega$ there is a $C$-chord-arc domain $\cU\subseteq \Omega$ containing $z$ such that $\cH^{1}(\d \cU\cap \d \Omega) \geq a \dist(z,\d \Omega)\geq b\cH^{1}(\d \cU)$. Also recall that a measure $\sigma$ on $\bR^{d}\times (0,\infty)$ is a {\it Carleson measure} on $\bR^{d}\times (0,\infty)$ if there is an infimal constant $C=C(\sigma)$ (the {\it Carleson norm} of $\sigma$) such that for all $x\in \bR^{d}$ and $r>0$,
\[\sigma(B(x,r)\times (0,r))\leq C|B(x,r)|.\]

\begin{theorem}[\cite{AZ91,MacManus94}] If $\Omega\subseteq \bC$ is a quasidisk, then $\Omega$ is a Bishop-Jones domain if and only if there is $f:\bC\rightarrow \bC$ quasiconformal, such that  $f(\bH)=\Omega$ (where $\bH$ is the upper half plane in $\bC$), $f$ is conformal on the lower half plane, and $\frac{\mu_{f}(x+iy)^{2}}{y}dxdy$ is a Carleson measure on $\bR\times (0,\infty)$ where $\mu_{f}=\frac{f_{\cnj{z}}}{f_{z}}$.
\label{t:astala}
\end{theorem}

Observe that $|\mu_{f}(z)|=\frac{K_{f}(z)-1}{K_{f}(z)+1}$, so that if $f$ is $K$-quasiconformal, then  $\frac{\tilde{K}_{f}(z)}{K+1}\leq |\mu_{f}(z)|\leq \tilde{K}_{f}(z)$, so one is a Carleson measure exactly when the other is. See \cite[Chapters 2 and 3]{EPDEQCM} for these facts about planar quasiconformal maps and their Beltrami coefficients $\mu$, and \cite{Semmes94} for similar results. 

The above results don't establish whether when $f:\bR^{d+1}\rightarrow \bR^{d+1}$ is bi-Lipschitz on a subset of $\bR^{d}$, only that their images are rectifiable. For showing a map is bi-Lipschitz on a large piece quantitatively, one typically requires some sort of quantitative differentiability result. To explain this notion, we go by way of a classic example due to Dorronsoro.

\begin{theorem}[\cite{Dorronsoro}]
Let $f\in L^{2}(\bR^{d})$. For $x\in \bR^{d}$, $r>0$, define
\[\Omega_{f}(x,r)=\inf_{A} \ps{\avint_{B(x,r)} \ps{\frac{|f-A|}{r}}^{2}}^{\frac{1}{2}}\]
where the infimum is over all affine maps $A:\bR^{d}\rightarrow \bR$. Then $f\in W^{1,2}(\bR^{d})$ if and only if 
\[\Omega(f):=\int_{\bR^{d}}\int_{0}^{\infty}\Omega_{f}(x,r)^{2}\frac{dr}{r}dx<\infty,\]
in which case, $||\grad f||_{2}^{2}\sim_{d}\Omega(f)$.
\label{t:dorronsoro}
\end{theorem}

This isn't the exact phrasing of his result, and the original theorem is far more general, but this special case has been more than sufficient for many applications. For the reader's convenience, we provide a well-known proof in Section \ref{s:dorronsoro} of the appendix.

While Rademacher's theorem, for example, says that at almost every $x\in \bR^{d}$ $f$ is approximately an affine function in small balls around the point $x$, it doesn't tell us how soon $f$ is within $\ve$, say, of some affine map. Using Dorronsoro's result and Chebyshev's inequality, however, shows that the largest scale $r>0$ for which $\Omega_{f}(x,r)<\ve$ can be estimated from below in terms of $||\grad f||_{2}$, $d$, and $\ve>0$. Results like this (that quantify how soon a function achieves a certain threshold of regularity, or bounds how often it doesn't) are examples of {\it quantitative differentiation} or {\it coarse differentiation}. 

Quantitative differentiation results have been used for embedding problems (\cite{Cheeger-QD}, \cite{Li-lip-bi-lip}), geometric group theory (\cite{EFW07}) and the theory of uniform rectifiability (see \cite{David88}, \cite{Jones-lip-bilip}, \cite[Lemma 10.11]{DS}, \cite[IV.2.2]{of-and-on}, and the references therein). While the latter results are more concerned with finding out when a function is approximately affine, there are situations involving, say, a metric space \cite{AS12prep}, or Carnot groups \cite{Li-lip-bi-lip}, where ``affine" is replaced with some other form of regularity. 

In \cite{Jones-lip-bilip}, for example, the author shows that if $f$ is $1$-Lipschitz, then for every $\delta>0$ one can partition $[0,1]^{d}$ into sets $G,K_{1},...,K_{M}$, where $M\leq M(\delta)$, such that $\cH_{d}^{\infty}(f(G))<\delta$ and $f$ is $\frac{2}{\delta}$-bi-Lipschitz on each $K_{j}$. To prove this, one can use something like \Theorem{dorronsoro} and a clever algorithm to sort the domain of $f$ into the desired sets $G,K_{1},...,K_{M}$ (see also \cite{David-wavelets}, p. 62). We won't replicate this method, but the condition in our main result will resemble Dorronsoro's theorem. In particular, instead of $\Omega_{f}$, we will use a similar quantity: define
\begin{equation}
\omega_{f}(x,r)=\inf_{A}\ps{\avint_{B(x,r)} \ps{\frac{|f-A|}{|A'|r}}^{2}}^{\frac{1}{2}}
\label{e:omega}
\end{equation}
where the infimum is over affine maps $A:\bR^{d}\rightarrow \bR^{D}$ with $|A'|\neq 0$. Here, $A'$ is the derivative of the mapping $A$, so that $A(x)=A'(x)+A(0)$. The appeal of this quantity, as opposed to $\Omega_{f}$, is that it is invariant under dilations in the domain {\it and} scaling the function $f$ in its image: if $s,t>0$ and $b\in \bR^{d}$, then
\[\omega_{f}(tx+b,tr)= \omega_{g}(x,r) \mbox{ if } g(y)=sf(tx+b).\]
Thus, if $\omega_{f}(x,r)$ is small, then $f$ is well-approximated by a {\it nontrivial} affine map inside $B(x,r)$, even if the image of $f(B(x,r))$ is very small. \\

In the main result below, much like \Theorem{astala}, we don't give a sufficient condition for when the image of $f$ is rectifiable, but when it contains a uniform amount of rectifiable parts within it in a sense we make precise in the following definition.

\begin{definition}
We'll say a set $\Sigma$ contains {\it big pieces of $d$-dimensional bi-Lipschitz images} with constants $\kappa>0$ and $L\geq 1$ (or BPBI($\kappa,L,d$) for short) if, for all $\xi\in \Sigma$ and $s>0$, there is $E\subseteq B(\xi,s)\cap \Sigma$ with $\cH^{d}(E)\geq \kappa s^{d}$ and $g:E\rightarrow \bR^{d}$ $L$-bi-Lipschitz. We will simply write BPBI($\kappa,L$) if the dimension $d$ is understood from context.
\end{definition}

Note that this ``big pieces" terminology is already prevalent in the literature (see \cite{DS} and \cite{of-and-on}), but usually includes the assumption that $\Sigma$ is Ahlfors regular, meaning that $\cH^{d}(\Sigma\cap B(x,r))$ is comparable to $r^{d}$. We emphasize, however, that the sets we'll be dealing with will not necessarily be $\cH^{d}$-finite, let alone regular. \\

We can now state our main result, which obtains a classification of all quasisymmetric mappings with uniformly rectifiable image in terms of the behavior of $\omega_{f}$, and can be considered as high dimensional analogue of \Theorem{astala}:

\begin{theorem}
Let $f:\bR^{d}\rightarrow \bR^{D}$ be quasisymmetric, $d\geq 2$. Then the following are equivalent:
\begin{enumerate}
\item The measure $\omega_{f}(x,r)^{2}\frac{dxdr}{r}$ is a $C$-Carleson measure on $\bR^{d}\times (0,\infty)$.
\item For all $\tau>0$ there is $L>0$ such that for all $x\in \bR^{d}$, $r>0$, there is $E \subseteq B(x,r)$ such that $|B(x,r)\backslash E|<\tau |B(x,r)|$ and $\ps{ \frac{\diam f(B(x,r))}{\diam B(x,r)}}^{-1} f|_{E}$ is $L$-bi-Lipschitz.
\item There are $c,L>0$ such that for all $x\in \bR^{d}$ and $r>0$, there is $E\subseteq B(x,r)$ such that $|E|\geq c|B(x,r)|$ and $\ps{ \frac{\diam f(B(x,r))}{\diam B(x,r)}}^{-1} f|_{E}$ is $L$-bi-Lipschitz.
\item The set $f(\bR^{d})$ has BPBI($\kappa,L$).
\end{enumerate} 
The equivalences are quantitative in the sense that, the constants in each item depend (in addition to $D$ and $\eta)$ only upon those in the other items. 

If $d=1$, then we just have $(1)\Rightarrow (2)\Rightarrow (3)\Rightarrow (4)$.
\label{t:main}
\end{theorem}
There is no equivalence in the case of $d=1$ (that is , $(4)\not\Rightarrow (1)$), since there are quasisymmetric maps of the real line that are uniformly oscillatory at every scale and location. We will give a counter-example in \Proposition{counter}.

We also mention that one can construct a single rectifiable piece in the image (or bi-Lipschitz part of $f$) without using the full strength of the Carleson measure; indeed, we prove a local version of (1)$\Rightarrow$(2) in \Theorem{1implies4} below. 

A similar result appears in \cite{ABT}, where the authors show that if $f:\bR^{D}\rightarrow \bR^{D}$ is quasisymmetric, $2\leq d <D$, $\tilde{H}_{f}(w,t)^{2}\frac{dwdt}{t}$ is a Carleson measure on $\bR^{d}\times (0,\infty)$ where
\[\tilde{H}_{f}(w,t)=\sup\ck{\frac{|f(x)-f(y)|}{|f(x)-f(z)|}: x,y,z\in B(w,t) \mbox{ are distinct and }|x-z|\leq |x-y|},\]
then $f(\bR^{d})$ has big pieces of bi-Lipschitz images, though the implication only holds with $d\geq 2$ and doesn't have a reverse implication. Also, while $\omega_{f}(x,r)$ is perhaps not as simple or ideal a quantity to compute than $K_{f}$ and $\tilde{H}_{f}$ mentioned above, it does handle a broader class of mappings (maps that are not restrictions of maps $f:\bR^{d}\rightarrow\bR^{d}$ to the sphere $\bS^{d-1}$, for example) and, more importantly, classifies those quasisymmetric mappings that have BPBI in their image. Moreover, the advantage in \cite{BGRT12} and \cite{ABT} is that $\tilde{H}_{f}$ has the monotonicity property that $\tilde{H}_{f}(x,r)\leq \tilde{H}_{f}(y,s)$ whenever $B(x,r)\subseteq B(y,s)$, which doesn't hold for $\omega_{f}$. On the other hand, \cite{ABT} has its own unique challenges: the main tool in our paper is Dorronsoro's theorem, for which $\omega_{f}(x,t)$ is naturally suited, but it's not clear whether we can apply this using only information about the values $\tilde{H}_{f}(x,r)$. \\

Our final result in the vein of finding bi-Lipschitz pieces of quasisymmetric maps is a  generalization of the following result of Semmes.

\begin{theorem}[\cite{Semmes-question-of-heinonen}]
Suppose $E\subseteq \bR^{d}$, $d\geq 2$, and $f:E\rightarrow \bR^{d}$ is $\eta$-quasisymmetric for some $\eta$. Then $|E|>0$ if and only if $|f(E)|>0$.
\label{t:semmes}
\end{theorem}

While this is a beautiful result, with just a  bit more work one can actually achieve a quantitative version that bounds how small we can make $|f(E)|$ in terms of only $\eta$, $d$, and the density of $E$.

\newtheorem*{semmes}{\Proposition{strong-semmes}}

\begin{proposition}
Let $E\subseteq Q_{0}\subseteq \bR^{d}$, $\rho\in(0,\frac{1}{2})$, and set $\delta=\frac{|E|}{|Q_{0}|}>0$. Let $f:E\rightarrow \bR^{d}$ be $\eta$-quasisymmetric. Then there is $E'\subseteq E$ compact with $|E'|\geq (1-\rho)|E|$ and $\ps{\frac{\diam f(E')}{\diam E'}}^{-1}f|_{E'}$ is $L$-bi-Lipschitz for some $L$ depending on $\eta$, $d$, $\rho$, and $\delta$. 
\label{p:strong-semmes}
\end{proposition}

We will cite several tools from \cite{Semmes-question-of-heinonen}, and with them, the modifications required to obtain Proposition \ref{p:strong-semmes} aren't too difficult,  hence the above proposition should really be credited to Semmes; in addition to Dorronsoro's theorem, however, it is a cornerstone to our paper, so we find it worth mentioning.

\subsection{Outline of proof}
Below we indicate where in the paper to find the proofs of each link in the chain of implications implying \Theorem{main}.
\begin{description}
\item[$(1)\Rightarrow (2)$] We prove this in \Theorem{1>2} in Section \ref{s:1>2}.
\item[$(2)\Rightarrow (3)$] This case is trivial.
\item[$(3)\Rightarrow (4)$] Although brief, we prove this implication in \Theorem{3>4} in Section \ref{s:3>4}.
\item[$(4)\Rightarrow (1)$] This is proven in \Theorem{4>1} in Section \ref{s:4>1}.
\end{description}
Section \ref{s:semmes} is devoted to showing Proposition \ref{p:strong-semmes}, a prerequisite for \Theorem{3>4}. Some basic preliminaries and notation are covered in Section \ref{s:prelims}, although a few tools will appear throughout whose proofs are delayed to the appendix in Section \ref{s:appendix}.

\subsection{Acknowledgements}
The author would like to thank Xavier Tolsa for his help in understanding Dorronsoro's theorem, Tatiana Toro for providing the inception for this project, and Robert Shukrallah and Michael Lacey for their helpful discussions. Part of this work was done while the author was attending the Interactions Between Analysis and Geometry program at the Institute for Pure and Applied Mathematics. The author is also very grateful for the anonymous referee for suffering through a poorly written draft.

\section{Preliminaries}
\label{s:prelims}

\subsection{Notation}

Many of the techniques and notation in this paper, if not mentioned or proven here, can be found in \cite{Heinonen}, \cite{Mattila}, and \cite{big-stein}. 

For nonnegative numbers or functions $A$ and $B$, we will write $A\lec B$ to mean $A\leq CB$ where $C$ is some constant, and $A\lec_{t} B$ if $C$ depends on some parameter $t$. Similarly, we will write $A\sim B$ if $A\lec B\lec A$ and $A\sim_{t} B$ if $A\lec_{t} B\lec_{t} A$. The Euclidean norm will be denoted by $|\cdot|$ and the ball centered at $x$ of radius $r$ by $B(x,r)=\{y:|x-y|\leq r\}$. Let $\Delta(\bR^{d})$ denote the collection of dyadic cubes in $\bR^{d}$ of the form
\[\Delta(\bR^{d})=\bigcup_{n\in \bZ}\ck{\prod_{i=1}^{d}[2^{n}j_{i},2^{n}(j_{i}+1)]\subseteq \bR^{d}: (j_{1},...,j_{d})\in \bZ^{d}}\]
and for $Q_{0}\in \Delta(\bR^{d})$, let $\Delta(Q_{0})$ the set of dyadic cubes contained in a dyadic cube $Q_{0}$. We will simply write $\Delta=\Delta(\bR^{d})$ if the dimension is clear from the context. For $Q\in \Delta$, set $Q^{1}$ to be the {\it parent} of $Q$, that is, the smallest dyadic cube properly containing $Q$, and inductively, for $N>1$, define $Q^{N}$ to be the smallest dyadic cube properly containing $Q^{N-1}$ (so $Q^{N}$ is the {\it $N$th generation ancestor of $Q$}). We will also refer to any cube $R$ with $R^{1}=Q^{1}$ as a {\it sibling} of $Q$.  We will denote the side length of a cube $Q$ by $\ell(Q)$ and its center by $x_{Q}$. For $\lambda>0$, $\lambda Q$ will denote the cube with center $x_{Q}$ and side length $\lambda \ell(Q)$

For a subset $A\subseteq \bR^{d}$, we will let $|A|$ denote the Lebesgue measure of $A$, $A^{\circ}$ its interior, $\d A$ its boundary, and $\one_{A}$ the indicator function for $A$ (that is, it is exactly one on $A$ and zero on the complement of $A$). For a Lebesgue measurable function $f$ and a measurable set $A$ of positive measure, we set $\avint_{A}f=|A|^{-1}\int_{A}f$. For $\delta>0$ and $A\subseteq \bR^{d}$, set
\[\cH^{d}_{\delta}(A)=w_{d}\inf\ck{\sum r_{i}^{d}: A\subseteq \bigcup B(x_{i},r_{i}), r_{i}<\delta},\]
where $w_{d}=|B(0,1)|$ and define the {\it (spherical) $d$-dimensional Hausdorff measure} $\cH^{d}(A)=\lim_{\delta\rightarrow 0} \cH^{d}_{\delta}(A)$.

If $A,B\subseteq \bR^{d}$, we set 
\[\diam(A)=\sup\{|x-y|:x,y\in A\},\]
\[ \dist(A,B)=\sup\{|x-y|:x\in A,y\in B\},\]
and for $x\in \bR^{d}$,
\[\dist(x,A)=\dist(\{x\},A).\]

For an affine transformation $A:\bR^{d}\rightarrow \bR^{D}$, we will write $A(x)=A'(x)+A(0)$, where $A'$ is a linear transformation (and the derivative of the map $A$), and we'll let $|A'|$ denote its operator norm.

\subsection{Basic facts about $\Omega_{f}$ and $\omega_{f}$}

Let $\Omega\subseteq \bR^{d}$ and $f:\Omega \rightarrow \bR^{D}$ be a locally bounded continuous function. It will be more convenient throughout the paper to work with dyadic versions of $\omega_{f}$ and $\Omega_{f}$: for $Q\subseteq \Omega$ a cube, define
\[\omega_{f}(Q)=\inf_{A}\ps{\avint_{Q}\frac{|f-A|}{|A'|\diam Q}^{2}}^{\frac{1}{2}}\;\; \mbox{ and }\;\;\Omega_{f}(Q)=\inf_{A}\ps{\avint_{Q}\frac{|f-A|}{\diam Q}^{2}}^{\frac{1}{2}}\]
where again the infimums are over all nonzero affine maps $A$. We will use the following monotonicity property often and without mention: if $R\subseteq Q$ and $\ell(R)\geq \delta \ell(Q)$, then $\omega_{f}(R)\leq \delta^{-d}\omega_{f}(Q)$. This is easily proven using the definition of $\omega_{f}$.

Moreover,  for any cube $Q$,
\begin{equation}
\omega_{f}(Q)\leq \frac{1}{2} .
\label{e:w<1/2}
\end{equation}
To see this, let $A_{j}=jA+f(x_{Q})$ where $A:\bR^{d}\rightarrow\bR^{D}$ is a fixed nonzero affine map. Then
\begin{align*}
\omega_{f}(Q)
& \leq \liminf_{j\rightarrow\infty} \ps{\avint_{Q}\ps{\frac{|f-A_{j}|}{|A_{j}'|\diam Q}}^{2}}^{\frac{1}{2}} \\
& \leq  \liminf_{j\rightarrow\infty}\ps{\avint_{Q}\ps{\frac{|f-f(x_{Q})|}{j\diam Q}}^{2}}^{\frac{1}{2}}+ \liminf_{j\rightarrow\infty}\ps{\avint_{Q}\ps{\frac{|A_{j}-A_{j}(x_{Q})|}{|A_{j}'|\diam Q}}^{2}}^{\frac{1}{2}}\\
& \leq  \liminf_{j\rightarrow\infty}\frac{\diam f(Q)}{j\diam Q}+ \liminf_{j\rightarrow\infty}\ps{\avint_{Q}\ps{\frac{|A_{j}'||x-x_{Q}|}{|A_{j}'|\diam Q}}^{2}}^{\frac{1}{2}}
\leq 0+ \frac{1}{2}.
\end{align*}

\begin{lemma}
Let $\delta>0$.  If $f$ is an $\eta$-quasisymmetric embedding of a cube $Q\subseteq \bR^{d}$ into $\bR^{D}$, then there is $\ve_{1}=\ve_{1}(\eta,d,\delta)>0$ so that if
\begin{equation}\avint_{Q}\frac{|f-A|}{|A'|\diam Q}<\ve_{1}
\label{e:f-a<e}
\end{equation}
then
\[|f(x)-A(x)| <\delta |A'|\diam Q \;\;\; \mbox{ for }x\in Q.\]
Moreover, 
\[ (1-2\sqrt{d}\delta)|A'|\ell(Q) \leq  \diam f(Q)  \leq (1+2\sqrt{d}\delta)|A'|\diam Q.\]
\label{l:epsilondelta}
\end{lemma}

We postpone the proof to Section \ref{s:epsilondelta-proof} in the appendix, and now use it  to show that the infimum in the definition of $\omega_{f}(Q)$ is actually achieved by a nonzero affine map if $\omega_{f}(Q)$ is small enough.

%
%
%

\begin{lemma}
Let $\eta:(0,\infty)\rightarrow (0,\infty)$ be an increasing homeomorphism, and $1\leq d\leq D$ integers. There is $\ve'=\ve'(\eta,d)>0$ so that if $Q\in\Delta(\bR^{d})$ and  $f:Q\rightarrow \bR^{D}$ is $\eta$-quasisymmetric with $\omega_{f}(Q)<\ve'$, then there is an affine transformation $A:\bR^{d}\rightarrow \bR^{D}$ so that
\begin{equation}
\omega_{f}(Q)^{2}=\avint_{Q}\ps{\frac{|f-A|}{|A'|\diam Q}}^{2}.
\label{e:AQ}
\end{equation}
\label{l:AQ}
\end{lemma}

\begin{proof}
Assume $\omega_{f}(Q)<\ve':= \ve_{1}(\eta,d,d^{-1/2}/2)/2$. Let $A_{i}$ is a sequence of affine maps such that
\[\avint_{Q}\ps{\frac{|f-A_{i}|}{|A_{i}'|\diam Q}}^{2} \rightarrow \omega_{f}(Q)^{2}.\]
For $i$ large enough,  we know
\[
\avint_{Q}\frac{|f-A_{i}|}{|A_{i}'|\diam Q} 
 \ps{\avint_{Q}\ps{\frac{|f-A_{i}|}{|A_{i}'|\diam Q}}^{2}}^{\frac{1}{2}} \leq 2\omega_{f}(Q)  \leq \ve_{1}\ps{\eta,d,\frac{1}{2\sqrt{d}}}.
\]
Hence, by \Lemma{epsilondelta}, $|A_{i}'|\sim \diam f(Q)/\diam Q$, so $|A_{i}'|$ is uniformly bounded above and below. Moreover, there is $x\in Q$ so that $|f(x)-A_{i}(x)|\leq \ve |A_{i}'| \omega_{f}(Q) \diam Q$. Hence, the sequence $A_{i}$ is uniformly bounded on $Q$ and uniformly bi-Lipschitz, and by Arzela-Ascoli, we may pick a subsequence converging uniformly to a nonconstant affine map $A$ on $Q$ satisfying \eqn{AQ}.
\end{proof}

\subsection{A counter example}

Here, we show that if $d=1$, then $(4)\not\Rightarrow (1)$ in \Theorem{main}.

\begin{proposition}
There is a quasisymmetric map $f:\bR\rightarrow \bR$ such that $\omega_{f}(Q)\gec 1$ for all $Q\in \Delta(\bR)$ with $\ell(Q)\leq 1$.
\label{p:counter}
\end{proposition}

\begin{proof}
To see this, let $\cI$ be the set of triadic half-open intervals in $[0,1)$ obtained inductively by taking an interval $I$ already in $\cI$, dividing it into three half-open subintervals $I_{\ell},I_{m}$, and $I_{r}$ (the left, middle, and right intervals) of equal size so that $I_{m}$ is between the other two, and adding these to $\cI$. Now let $\rho\in (0,1/3)$ and $\mu$ be the measure on $\bR$ satisfying $\mu([0,1))=1$, $\mu(I_{\ell})=\mu(I_{r})=\rho\mu(I)$ for all $I\in \cI$, and for any $n\in \bZ$ and $A\subseteq [n,n+1)$, set $\mu(A)=\mu(A-n)$.  This is the so-called {\it Kahane measure} on $\bR$ (although not his exact construction in \cite{trois-notes}), and is singular with respect to Lebesgue measure. This is a {\it doubling measure}, meaning there is $C>0$ such that $\mu(B(x,2r))\leq C\mu( B(x,r))$ and singular with respect to Lebesgue measure (see \cite{Garnett-Killip-Schul} for a proof of these facts). 

Define $f:\bR\rightarrow \bR$ by setting $f(x)=\mu([0,t])$ for $x\geq 0$ and $\mu([x,0])$ for $x\leq 0$. It is not hard to show this is an increasing quasisymmetric mapping since $\mu$ is doubling (c.f. \cite[Remark 13.20(b)]{Heinonen}). For any $Q\in \Delta(\bR)$ with $\ell(Q)\leq 1$, we may find a triadic interval $I\subseteq 3Q$ of length at least $\ell(Q)/3$, and if $a,b$ are the endpoints of $I$ and $a,c$ of $I_{\ell}$, then 
\[|f(a)-f(c)|=\mu(I_{\ell})=\rho\mu(I)=\rho |f(a)-f(b)|.\] 
Let $\delta>0$ and suppose we may find $x\in \bR$ and $Q\in \Delta(\bR)$ with $\ell(Q)\leq 1$ so that $\omega_{f}(3Q)<\ve_{1}(\eta,d,\delta)$. We will show this results in a contradiction if $\delta>0$ is small enough, proving the proposition. By \Lemma{epsilondelta}, there is $A$ a nonconstant affine map such that 
\begin{equation}
|A'|\sim \frac{\diam f(Q)}{\diam Q}=\frac{\mu(Q)}{\diam Q}
\label{e:kah1}
\end{equation}
and
\begin{equation}
||f-A||_{L^{\infty}(3Q)}<\delta|A'|\diam 3Q.
\label{e:kah2}
\end{equation}
Hence
\begin{align*}
\rho|f(a)-f(b)|
& = |f(a)-f(c)|
\stackrel{\eqn{kah2}}{>} |A(a)-A(c)|-2\delta|A'|\diam 3Q \\
& =\frac{1}{3}|A(a)-A(b)|-2\delta|A'|\diam 3Q\\
& \stackrel{\eqn{kah2}}{>} \frac{1}{3}|f(a)-f(b)|-\frac{8}{3}\delta|A'|\diam 3Q.
\end{align*}
Thus,
\[ \ps{\frac{1}{3}-\rho}
\mu(I)= \ps{\frac{1}{3}-\rho}|f(a)-f(b)|
< \frac{8}{3}\delta|A'|\diam 3Q
\stackrel{\eqn{kah1}}{\sim} \delta \mu(Q)\]
and since $\mu$ is doubling, we know $\mu(Q)\lec_{\mu} \mu(I)$, hence we have 
\[\ps{\frac{1}{3}-\rho}
\mu(I)\lec_{\mu} \delta \mu(I),\] 
which is a contradiction for $\delta$ small enough.
\end{proof}

\subsection{Dyadic Carleson conditions}

Suppose now $f:\bR^{d}\rightarrow\bR^{D}$ is a quasisymmetric mapping such that $\omega_{f}(x,r)^{2}\frac{dxdr}{r}$ is a Carleson measure, meaning there is an infimal $C>0$ (the Carleson norm of this measure) such that
\begin{equation}
\int_{B(z,t)} \int_{0}^{t} \omega_{f}(x,r)^{2}\frac{dr}{r}dx\leq C|B(z,t)| \mbox{ for }z\in \bR^{d}\mbox{ and }t>0.
\label{e:carleson}
\end{equation}

If $M>1$, \eqn{carleson} is quantitatively equivalent to the condition that there is an infimal $C_{M}$ such that
\begin{equation}
\sum_{Q\subseteq Q_{0}}\omega_{f}(M Q)^{2}|Q|\leq C_{M}|Q_{0}|
\label{e:dyadic-carleson}
\end{equation}
for any dyadic cube $Q_{0}$. We show this in the following lemma.

\begin{lemma}
If $M>1$ and either \eqn{carleson} or \eqn{dyadic-carleson} hold, then the other holds, and $C\sim_{d,M} C_{M}$.
\label{l:dyadic-version}
\end{lemma}


\begin{proof}
We will only show this lemma for $M=3$, as the general case is similar, and we'll only show $C_{M}\lec_{d} C$ as the opposite inequality is proven similarly.

Let $A$ be a  nonconstant affine map. Then, for $Q\in \Delta$, $x\in Q$, and $r\in [2\diam Q,4\diam Q]$,
\begin{align*}
\omega_{f}(3Q)^{2}
& \leq \avint_{3Q}\ps{\frac{|f-A|}{|A'|\diam 3Q}}^{2}
\leq \frac{r^{2}|B(x,r)|}{|3Q|(\diam 3Q)^{2}} \avint_{B(x,r)}\ps{\frac{|f-A|}{|A'|r}}^{2}\\
& \lec_{d} \avint_{B(x,r)}\ps{\frac{|f-A|}{|A'|r}}^{2}
\end{align*}
and infimizing over non constant affine maps $A$ gives
\[\omega_{f}(3Q)^{2}\lec_{d}\omega_{f}(x,r) \mbox{ for } x\in Q, \;\;r\in [2\diam Q,4\diam Q].\]
Thus, for any $Q_{0}\in \Delta$,
\begin{align*}
\sum_{Q\subseteq Q_{0}} \omega_{f}(3Q)^{2}|Q|
& \lec\sum_{Q\subseteq Q_{0}}\int_{Q}\int_{2\diam Q}^{4\diam Q} \omega_{f}(3Q)^{2}\frac{dr}{r}dx \\
& \lec_{d} \sum_{Q\subseteq Q_{0}}\int_{Q} \int_{2\diam Q}^{4\diam Q}\omega_{f}(x,r)^{2}\frac{dr}{r}dx \\
& =\sum_{n\geq 0}\sum_{Q^{n}=Q_{0}}\int_{Q}\int_{2^{-n+1}\diam Q_{0}}^{2^{-n+2}\diam Q_{0}}\omega_{f}(x,r)^{2}\frac{dr}{r}dx\\
& \leq \int_{Q_{0}}\int_{0}^{4\diam Q_{0}}\omega_{f}(x,r)^{2}\frac{dr}{r}dx \\
& \leq C|B(x_{Q_{0}},4\diam Q_{0})|\lec_{d} C|Q_{0}|.
\end{align*}

\end{proof}

We can prove a similar relation for $\Omega_{f}$.

\begin{lemma}
For $\Omega\subseteq \bR^{d}$, $f:\Omega\rightarrow\bR^{D}$, and $Q\subseteq \Omega$, define
\[\Omega_{f}(Q)^{2} = \inf_{A}\avint_{Q}\ps{\frac{|f-A|}{\diam Q}}^{2}\]
where the infimum is over all affine maps $A:\bR^{d}\rightarrow \bR^{D}$. If $M>1$ and $f\in W^{1,2}(\bR^{d},\bR^{D})$, then
\[\sum_{Q\in\Delta}\Omega_{f}(MQ)^{2}|Q| \sim_{D,M} ||D f||_{2}.\]
\label{l:dyadicdorronsoro}
\end{lemma}

\begin{proof}
Note that  \Theorem{dorronsoro} holds for functions $f:\bR^{d}\rightarrow \bR^{D}$ (with $D$ not necessarily equal to one) if we replace $\grad f$ with $Df$. The proof now is similar to \Lemma{dyadic-version}, so we omit it.
\end{proof}

\section{Carleson condition implies $f$ is bi-Lipschitz on a very large set}
\label{s:1>2}

In this section, we prove the first part of \Theorem{main} by establishing that (1) implies  (2). We state this implication as a theorem below.

\begin{theorem}
Suppose $f:\bR^{d}\rightarrow \bR^{D}$ is $\eta$-quasisymmetric and $\omega_{f}(x,r)^{2}\frac{dr}{r}dx$ is a Carleson measure. Then for all $\tau>0$ there is $L>1$ such that for all $x\in \bR^{d}$, $r>0$, there is $E \subseteq B(x,r)$ such that $|B(x,r)\backslash E|<\tau |B(x,r)|$ and $\ps{ \frac{\diam f(B(x,r))}{\diam B(x,r)}}^{-1} f|_{E}$ is $L$-bi-Lipschitz.
\label{t:1>2}
\end{theorem}

 \subsection{Stopping-time regions}

The ideas behind this section are taken from the theory of uniform rectifiability (see \cite{DS} and \cite{of-and-on}, for example). 
Let 
\[M=30000d.\]
We will keep $M$ fixed throughout the rest of Section \ref{s:1>2}.

\begin{definition}
(\cite[I.3.2]{of-and-on}) A {\it stopping-time region} $S\subseteq \Delta$ is  a collection of cubes such that
\begin{enumerate}
\item all cubes $Q\in S$ are contained in a maximal cube $Q(S)\in S$;
\item $S$ is {\it coherent}, meaning $R\in S$ for all $Q\subseteq R \subseteq Q(S)$ whenever $Q\in S$;
\item for all $Q\in S$, each of its siblings of $Q$ are also in $S$.
\end{enumerate}
\label{d:stopping-time}
We let $m(S)$ denote the set of minimal cubes of $S$, i.e. those cubes $Q\in S$ such that there are no cubes $R\in S$ properly contained in $Q$. We also set 
\[z(S)=Q(S)\backslash \bigcup\{Q:Q\in m(S)\}\]
which a the set of points in $Q(S)$ that are contained in infinitely many cubes in $S$. \\
\end{definition}

For an $\eta$-quasisymmetric map $f:\Omega \rightarrow \bR^{D}$ defined on a domain $\Omega\subseteq \bR^{d}$ and $Q\in \Delta$, if $MQ\subseteq \Omega$ and $\omega_{f}(MQ)<\ve'(\eta,d)$, by \Lemma{AQ} we may assign to $Q$ an affine map $A_{Q}:\bR^{d}\rightarrow \bR^{D}$ such that 
\[\omega_{f}(MQ)^{2}=\avint_{MQ}\ps{\frac{|f-A_{Q}|}{|A_{Q}'|\diam MQ}}^{2}.\]

\begin{definition}
For $\Omega\subseteq \bR^{d}$ and $f:\Omega\rightarrow \bR^{D}$ $\eta$-quasisymmetric, $\ve \in(0,\ve'(\eta,d))$, $\tau\in(0,1)$, we will call a stopping-time region $S$ an {\it $(\ve,\tau)$-region} for $f$ if $MQ(S)\subseteq\Omega$ and if for any $Q\in S$,
\begin{enumerate}
\item $\sum_{Q\subseteq R\subseteq Q(S)} \omega_{f}(MR)^{2}<\ve^{2}$,
\item $|A_{Q(S)}'-A_{Q}'| \leq \tau |A_{Q(S)}'|$, and
\item all siblings of $Q$ in $S$ satisfy (1) and (2).
\end{enumerate}
\label{d:vt-region}
\end{definition}

Note that, if $Q$ is in a $(\ve,\tau)$-region $S$, then (2) implies
\begin{equation}
(1-\tau)|A_{Q(S)}'|\leq |A_{Q}'|\leq (1+\tau)|A_{Q(S)}'| \mbox { for all }Q\in S.
\label{e:A<AS}
\end{equation}

The first major step toward proving \Theorem{1>2} is the following.

\begin{theorem}
Let $\tau\in (0,1)$, $C_{M}>0$, and $\eta:(0,\infty)\rightarrow (0,\infty)$ be an increasing homeomorphism. There is $\ve_{0}=\ve_{0}(\eta,D,\tau,C_{M})>0$  so that the following holds. If $\Omega\subseteq \bR^{d}$, $f:\Omega\rightarrow \bR^{D}$ is $\eta$-quasisymmetric, $0<\ve<\ve_{0}$, $Q_{0}\in\Delta$, $MQ_{0}\subseteq \Omega$, and
\begin{equation}
\sum_{Q\subseteq Q_{0}}\omega_{f}(MQ)^{2}|Q|\leq C_{M}|Q_{0}|
\label{e:sumomegamq}
\end{equation}
then we may partition $\Delta(Q_{0})$ into a set of ``bad" cubes $\cB$ and a collection $\cF$ of $(\ve,\tau)$-stopping time regions so that
\begin{equation}
\sum_{Q\in \cB} |Q|\leq \frac{C_{M}}{\ve^{2}} |Q_{0}|
\label{e:sumB}
\end{equation}
and
\begin{equation}
\sum_{S\in \cF}|Q(S)|\leq  \ps{4+\frac{2^{d+1}C_{M}}{\ve^{2}}}|Q_{0}|.
\label{e:sumF}
\end{equation}
\label{t:corona}
\end{theorem}

\begin{proof}
{\it Step 1:} We first show that for any $Q_{1}\in \Delta(Q_{0})$, if $\omega_{f}(MQ_{1})<\ve$, we may construct a $(\ve,\tau)$-region $S(Q_{1})$ with $Q(S(Q_{1}))=Q_{1}$. First, enumerate the cubes in $\Delta(Q_{1})$  as $\{Q_{j}\}_{j=1}^{\infty}$ so that $\ell(Q_{i})>\ell(Q_{j})$ implies $i<j$. Set $S_{1}=\{Q_{1}\}$, and for $j>1$, set $S_{j}=S_{j-1}\cup \{Q_{j}\}$ if the following hold:
\begin{enumerate}[(a)]
\item $Q_{j}^{1}\in S$,
\item $\sum_{Q\subseteq R\subseteq Q_{1}}\omega_{f}(MR)^{2}<\ve^{2}$,  
\item $|A_{Q_{1}}'-A_{Q}'| \leq \tau |A_{Q_{1}}'|$, and
\item all siblings of $Q_{j}$ in $S$ satisfy the above properties.
\end{enumerate}

Otherwise, set $S_{j}=S_{j-1}$. Define $S(Q_{1})=\bigcup_{j=1}^{\infty} S_{j}$. Clearly, it's a stopping-time region and satisfies (1), (2), and (3) in \Definition{vt-region}. Observe that, when constructed in this way, for $Q\in m(S_{1})$, there is a child $R$ of $Q$ such that either (1) or (2) fails.

{\it Step 2:} Next, we define the sets $\cB$ and $\cF$. Set 
\[
\cB=\{Q\subseteq Q_{0}: \omega_{f}(MQ)\geq \ve\}\]
and enumerate the cubes $\Delta(Q_{0})\backslash \cB$ as $\{Q(j)\}_{j=1}^{\infty}$ so that $\ell(Q(j))<\ell(Q(i))$ implies $i<j$. We let $\cF=\bigcup_{j=1}^{\infty} \cF_{j}$ where the sets $\cF_{j}$ are defined inductively as follows:  set $\cF_{1}=\{S(Q_{1})\}$ and let $\cF_{j+1}=\cF_{j}\cup \{S(Q(j))\}$ if $Q(j)\not\in S$ for any $S\in \cF_{j}$; otherwise, set $\cF_{j+1}=\cF_{j}$. Note that if $\cF_{j+1}\neq\cF_{j}$, then $Q(j+1)\in \cB$ or in $S$ for some $S\in \cF_{j}$.

{\it Step 3:} We now set out to verify \eqn{sumB} and \eqn{sumF} for the sets $\cB$ and $\cF$. The first inequality follows easily, since
\[\sum_{Q\in \cB}|Q|\leq \ve^{-2}\sum_{Q\in \cB}\omega_{f}(MQ)^{2}\leq C_{M}\ve^{-2}|Q_{0}|\]
so now we focus on \eqn{sumF}.

For $S\in\cF$, set 
\[m_{1}(S)=\ck{Q\in m(S): \sum_{Q'\subseteq R\subseteq Q(S)}\omega_{f}(MR)^{2}\geq \ve^{2} \mbox{ for some child } Q'\subseteq Q}\]
and
\begin{multline}
m_{2}(S)=m(S)\backslash m_{1}(S) \\
=\left\{ Q\in m(S):\sum_{Q'\subseteq R\subseteq Q(S)}\omega_{f}(MR)^{2}<\ve^{2} \mbox{ for all children } 
 Q'\subseteq Q\right.  \\
 \left. \mbox{ but }\frac{|A_{Q'}'-A_{Q(S)}'|}{|A_{Q(S)}'|}>\delta \mbox{ for some child of }Q'\mbox{ of }Q\right\}.
 \label{e:m2}
\end{multline}
Also set
\[M_{j}(S) = \bigcup_{Q\in m_{j}(S)}Q, \;\;\; j=1,2.\]
Then
\begin{equation}
Q(S)=M_{1}(S)\cup M_{2}(S)\cup z(S).
\label{e:m1m2z}
\end{equation}

\begin{lemma}
There is $\upsilon=\upsilon(D)>0$ so that if 
 \begin{equation}
0<\ve<\ve_{0}:=\min\ck{ \ve'(d,\eta), \upsilon C_{M}^{-\frac{1}{2}}\tau},
\label{e:ve0} 
\end{equation}
and $S$ is an $(\ve,\tau)$-region $S$ for an $\eta$-quasisymmetric map $f:\Omega\rightarrow \bR^{D}$ where $MQ(S)\subseteq\Omega\subseteq\bR^{d}$, then
\begin{equation}
|M_{2}(Q)|<|Q(S)|/2.
\label{e:m2-estimate}
\end{equation}
\label{l:m2-estimate}
\end{lemma}

Let us assume this lemma and finish the proof of \Theorem{corona}. Let
\begin{equation}
\cF_{1}=\{S\in \cF: |z(S)|\geq |Q(S)|/4\}
\label{e:F1}
\end{equation}
and
\begin{equation}\cF_{2}=\{S\in \cF: |M_{1}(S)|\geq |Q(S)|/4\}.
\label{e:F2}
\end{equation}
Note that the sets $z(S)$ intersect only at the boundaries of dyadic cubes. To see this, observe that if $S$ and $S'$ were such that they intersected in the interior of a cube, then the interiors of $Q(S)$ and $Q(S')$ intersect, so one must be contained in the other. Suppose $Q(S)\subseteq Q(S')$. Then $Q(S)$ is contained inside a minimal cube of $S'$ (since otherwise $Q(S)\in S\cap S'=\emptyset$), but $z(S)$ is the complement of these minimal cubes and so $z(S')\cap Q(S)=\emptyset$, and thus $z(S)\cap z(S')=\emptyset$, a contradiction. Thus, the $z(S)$ intersect only at the boundaries of dyadic cubes, which have measure zero, hence the $z(S)$ are essentially disjoint. Since they are contained in $Q(S)$, 
\begin{equation}
\sum_{S\in \cF_{1}}|Q(S)|\leq 4\sum_{S\in \cF_{1}}|z(S)|\leq 4|Q_{0}|.
\label{e:F1<4Q0}
\end{equation}
If $Q\in m_{1}(S)$, there is a child $Q'$ of $Q$ so that 
\begin{align*}
\ve^{2}
& \leq \sum_{Q'\subseteq R\subseteq Q(S)}\omega_{f}(MR)^{2} 
 \leq \omega_{f}(MQ')+\sum_{Q\subseteq R\subseteq Q(S)}\omega_{f}(MR)^{2}
\end{align*}
If $\omega_{f}(MQ')^{2}<\frac{\ve^{2}}{2}$, this implies
\[\frac{\ve^{2}}{2}<\sum_{Q\subseteq R\subseteq Q(S)}\omega_{f}(MR)^{2}\leq \ve^{2},\]
and if $\omega_{f}(MQ')^{2}\geq \ve^{2}/2$, then 
\[\ve^{2} \geq \sum_{Q\subseteq R\subseteq Q(S)}\omega_{f}(MR)^{2}\geq \omega_{f}(MQ)^{2}\geq 2^{-d}\omega_{f}(MQ')^{2}\geq \frac{\ve^{2}}{2^{d+1}},\]
so that in any case,
\begin{equation}
\ve^{2}\geq \sum_{Q\subseteq R\subseteq Q(S)}\omega_{f}(MR)^{2} \geq  \frac{\ve^{2}}{2^{d+1}} \mbox{ for all }Q\in m_{1}(S).
\label{e:sumsime^2}
\end{equation}
Hence, since the $Q\in m_{2}(S)$ have disjoint interiors,
\begin{align}
\sum_{S\in \cF_{2}} |Q(S)|
& \leq 4\sum_{S\in \cF_{2}} |M_{1}(S)|
 = 4\sum_{S\in \cF_{2}}\sum_{Q\in m_{1}(S)}|Q| \notag \\
& \stackrel{\eqn{sumsime^2}}{\leq} \frac{2^{d+1}}{\ve^{2}}\sum_{S\in \cF_{2}}\sum_{Q\in m_{1}(S)} \sum_{Q\subseteq R\subseteq Q(S)} \omega_{f}(MR)^{2}|Q| \notag \\
& =\frac{2^{d+1}}{\ve^{2}}\sum_{S\in \cF_{2}}\sum_{R\in S}\omega_{f}(MR)^{2}\sum_{Q\subseteq R\atop Q\in m_{1}(S)}|Q| \notag \\
& \leq \frac{2^{d+1}}{\ve^{2}}\sum_{S\in \cF_{2}}\sum_{R\in S} \omega_{f}(MR)^{2}|R| \notag  \\
& \leq \frac{2^{d+1}}{\ve^{2}}\sum_{R\subseteq Q_{0}}\omega_{f}(MR)^{2} |R| 
 \leq\frac{2^{d+1}C_{M}}{\ve^{2}}|Q_{0}|.
 \label{e:F2<CM/e^2Q0}
\end{align}

By \eqn{m2-estimate}, $\cF=\cF_{1}\cup \cF_{2}$, so that 
\[\sum_{S\in \cF}|Q(S)|
= \sum_{i=1,2}\sum_{S\in \cF_{i}}|Q(S)|\stackrel{\eqn{F1<4Q0}\atop \eqn{F2<CM/e^2Q0}}{\leq} 4|Q_{0}|+\frac{2^{d+1}C_{M}}{\ve^{2}} \leq \ps{4+\frac{2^{d+1}C_{M}}{\ve^{2}}}|Q_{0}|.\]
This finishes the proof of the lemma, so long as we show \Lemma{m2-estimate}, which will be the focus of the next few sections.

\end{proof}

\subsection{Whitney cubes for stopping-time regions}

Before attacking \Lemma{m2-estimate}, we prove some general properties about stopping-time regions. The reader may just want to familiarize themselves with the notation and lemmas, move on to Section \ref{s:extensions}, and return to the actual proofs on second reading. Many of these estimates can be found in \cite[Section 8]{DS}\\

Let $S$ be a stopping-time region as in \Definition{stopping-time}. For $x\in \bR^{d}$, define
\[D_{S}(x)=\inf \{\dist(x,Q)+\diam Q: Q\in S\}.\]
For $Q\in \Delta$, let
\[D_{S}(Q)=\inf_{x\in Q}D_{S}(x).\]
Let $R_{j}$ be the set of maximal dyadic cubes in $\bR^{d}\backslash \cnj{z(S)}$ such that 
\begin{equation}
\diam R_{j}\leq  \frac{1}{20}D_{S}(R_{j}).
\label{e:Rj}
\end{equation}
The $R_{j}$ are essentially Whitney cubes (see \cite[Chapter IV]{little-stein}), though rather than having diameter comparable to their distance from some prescribed set (as is usually how a Whitney decomposition is tailored), they have diameter comparable to their ``distances" $D_{S}$ from $S$ (see \eqn{RD} below).

For each $R_{j}$, pick $\tilde{Q}_{j}\in S$ such that 
\begin{equation}
\dist (x_{R_{j}}, \tilde{Q}_{j})+\diam \tilde{Q_{j}} \leq \frac{3}{2} D_{S}(x_{R_{j}}).
\label{e:tildeQ}
\end{equation}
Note that since the $R_{j}$ has positive diameter, $D_{S}(x_{R_{j}})> 0$, so the above makes sense. Next, pick  a maximal parent $Q_{j}\in S$ of $\tilde{Q_{j}}$ so that 
\begin{equation}
\diam Q_{j} \leq 3D_{S}(R_{j}).
\label{e:Q<3D}
\end{equation}

\begin{lemma}\label{l:david-lemma}
Let $S$ be a stopping-time region, and define $R_{j}$ and $Q_{j}$ as in \eqn{Rj}, \eqn{tildeQ}, and \eqn{Q<3D}.
\begin{enumerate}
\item If $x\in R_{j}$, then
\begin{equation}
20\diam R_{j}\leq D_{S}(x)\leq 60\diam R_{j}\mbox{  for all }x\in R_{j}.
\label{e:RD}
\end{equation}
\item If $2R_{i}\cap 2R_{j}\neq\emptyset$, then 
\begin{equation}
 \diam R_{i}\leq 2\diam R_{j}.
 \label{e:Ri<2Rj}
 \end{equation}
 \item The cubes $2R_{j}$ have bounded overlap, in the sense that 
 \begin{equation}
 \one_{\bR^{d}\backslash\cnj{z(S)}}\leq \sum_{j} \one_{2R_{j}}\lec_{d} \one_{\bR^{d}\backslash\cnj{z(S)}}.
 \label{e:bounded-overlap}
 \end{equation}
\item The cubes $R_{j}$ and $Q_{j}$ are close, in the sense that 
\begin{equation}
\dist(x_{Q_{j}},R_{j})\leq 180\diam R_{j}.
\label{e:dxQR}
\end{equation}
\item For all $j$,
\begin{equation}
\diam Q_{j}\leq 180\diam R_{j}.
\label{e:Q<R}
\end{equation}

\item If $\diam R_{j}\leq  2\diam Q(S)$, then 
\begin{equation}
\diam R_{j}\leq 2\diam Q_{j}.
\label{e:R<Q}
\end{equation}
\item  If $\diam R_{j}\geq \diam Q(S)/60$, then $Q_{j}=Q(S)$.
\end{enumerate}
\end{lemma}

\begin{proof}
\begin{enumerate}
\item The lower bound in \eqn{RD} follows by definition, so we focus on the upper bound. Observe that, since $R_{j}$ is maximal, that means there is $y\in R_{j}^{1}$ so that 
\[ \diam R_{j}^{1}> \frac{1}{20}D_{S}(y).\]
Let $x\in R_{j}$ be the point closest to $y$. Since $D_{S}$ is $1$-Lipschitz, we have 
\[D_{S}(y)\geq D_{S}(x)-|x-y|\geq D_{S}(x)-\diam R_{j}\]
and because $\diam R_{j}^{1}=2\diam R_{j}$, we have by the maximality of $R_{j}$ that
\[2\diam R_{j}=\diam R_{j}^{1} > \frac{1}{20}D_{S}(y)\geq \frac{1}{20}(D_{S}(x)-\diam R_{j})\]
and thus
\[D_{S}(x)\leq 20(2+\frac{1}{20})\diam R_{j}< 60\diam R_{j}.\]

\item If $z\in 2R_{i}\cap 2R_{j}$ then 
\[|x_{R_{i}}-x_{R_{j}}| \leq |x_{R_{i}}-z|+|z-x_{R_{j}}| \leq \diam R_{i}+\diam R_{j}\]
so that 
\begin{align*} 20\diam R_{i} 
& \stackrel{\eqn{RD}}{\leq} D_{S}(x_{R_{i}})\leq D_{S}(x_{R_{j}})+|x_{R_{i}}-x_{R_{j}}|\\
& \leq 60 \diam R_{j}+\diam R_{i}+\diam R_{j}.
\end{align*}
Hence,
\[\diam R_{i}\leq \frac{61}{19}\diam R_{j}<4\diam R_{j}.\]
Since $R_{i}$ and $R_{j}$ are dyadic cubes, $\frac{\diam R_{i}}{\diam R_{j}}$ is a power of two, so in fact, $\diam R_{i}\leq 2\diam R_{j}$, which implies \eqn{Ri<2Rj}.

\item Note that for any $R_{i}$ and $z\in z(S)$, there are infinitely many $Q\in S$ containing $z$, so $D_{S}(R_{i})\leq |y-z|+\diam Q$ for all such $Q$, and so $D_{S}(R_{i})\leq |y-z|$ for all $z\in z(S)$, and this implies
\[\dist(R_{i},z(S))\geq D_{S}(R_{j})\stackrel{\eqn{RD}}{\geq} 20\diam R_{i}\]
and so we have $2R_{i}\subseteq \bR^{d}\backslash \cnj{z(S)}$.  The rest now follows from this and  \eqn{Ri<2Rj}.

\item For any $j$, if $z\in Q_{j}$ is closest to $R_{j}$, then
\begin{align*}
\dist (x_{Q_{j}},R_{j})
& \leq \dist(z,R_{j})+|x_{Q_{j}}-z|\leq \dist(x_{R_{j}},\tilde{Q_{j}})+\frac{\diam Q_{j}}{2}\\
& \stackrel{\eqn{tildeQ} \atop \eqn{Q<3D}}{\leq}3D_{S}(R_{j})\stackrel{\eqn{RD}}{\leq} 180\diam R_{j}
\end{align*}

\item This follows from \eqn{RD} and \eqn{Q<3D}.

\item This is trivial in the case $Q_{j}=Q(S)$, so we assume $Q_{j}\neq Q(S)$, in which case, since $Q_{j}^{1}\in S$ and since $Q_{j}$ is a maximal cube for which $\diam Q_{j}\leq 3D_{S}(R_{j})$, we have
\[ 3D_{S}(R_{j}) < \diam Q_{j}^{1}=2\diam Q_{j} \]
so that
\[\diam Q_{j} > \frac{3}{2}D_{S}(R_{j})\stackrel{\eqn{RD}}{\geq} 30\diam R_{j}.\]

\item Observe that if $Q_{j}\neq Q(S)$, then any cube $Q\in S$ properly containing $Q_{j}$ satisfies $\diam Q>3D_{S}(x_{R_{j}})$, so in particular, $\diam Q(S)>3D_{S}(x_{R_{j}})$
Thus,
\[ \diam R_{j}\leq \frac{1}{20}D_{S}(x_{R_{j}})<\frac{1}{60}\diam Q(S).\]
\end{enumerate}

\end{proof}

\begin{lemma}
For all $i$
\begin{enumerate}
\item $Q_{i}\subseteq MR_{i}$.

\item If $\diam R_{i}\leq 2\diam Q(S)$, then
\begin{enumerate}[(a)]
\item $R_{i}\subseteq B(x_{Q_{i}},M\ell(Q_{i})) \subseteq MQ_{i}$, and
\item for all $j$, if $2R_{i}\cap 2R_{j}\neq\emptyset$, we have $R_{i} \subseteq MQ_{j}$ and
\begin{equation}
\diam Q_{i}\sim \diam Q_{j}\sim \diam R_{i}\sim \diam R_{j}.
\label{e:all-sim}
\end{equation}
\end{enumerate}
\end{enumerate}
\label{l:2Rcap2R}
\end{lemma}

\begin{proof}
Before beginning the proof, we recall that we chose $M=30000d$.

\begin{enumerate}
\item If $\tilde{Q_{i}}$ is as in \eqn{tildeQ}, then

\begin{equation}
\dist(x_{R_{i}},Q_{i})+\diam Q_{i}
\stackrel{\eqn{Q<3D}}{\leq} \dist(x_{R_{i}},\tilde{Q}_{i})+3D_{S}(x_{R_{i}})\stackrel{\eqn{tildeQ}}{\leq} 5D_{S}(x_{R_{i}})\stackrel{\eqn{Rj}}{\leq} 300\diam R_{i}
\label{e:dRQ}
\end{equation}
so that
\begin{equation}
Q_{i}\subseteq B(x_{R_{i}},300\diam R_{i})\subseteq  MR_{i}.
\end{equation}

\item Assume $\diam R_{i}\leq 2\diam Q(S)$.
\begin{enumerate}[(a)]
\item By \Lemma{david-lemma},
\begin{align}
R_{i} &  \subseteq B(x_{Q_{i}},\dist(x_{Q_{i}},R_{i})+\diam R_{i}) \stackrel{\eqn{dxQR}}{\subseteq} B(x_{Q_{i}},181\diam R_{i}) \notag \\
& \stackrel{\eqn{R<Q}}{\subseteq} B(x_{Q_{i}}, 362\sqrt{d} \ell(Q_{i})).
 \subseteq MQ_{i}  \label{e:RinBall}  
\end{align}

\item If $2 R_{i}\cap 2 R_{j}\neq\emptyset$, then $\dist (R_{i},R_{j})\leq \diam R_{i}+\diam R_{j}$ and $\diam R_{j}\leq 2\diam R_{i}$ by \eqn{Ri<2Rj}, and so

\begin{align*}
\dist (x_{Q_{j}},R_{i})
& \leq \dist(x_{Q_{j}},R_{j})+\diam R_{j}+\dist (R_{j},R_{i})\\
& \stackrel{\eqn{dxQR}}{\leq} 180\diam R_{j}+\diam R_{j}+(\diam R_{i}+\diam R_{j})\\
& \stackrel{\eqn{Ri<2Rj}}{\leq} (180+1+2+1)\diam R_{j}=184\diam R_{j}.
\end{align*}
If $\diam R_{j}\leq 2\diam Q(S)$, then
\[184\diam R_{j}\stackrel{\eqn{R<Q}}{\leq} 368 \diam Q_{j}\]
and
\[\diam R_{i}\leq 2\diam R_{j}\leq 4\diam Q_{j};\]
if $\diam R_{j}>2\diam Q(S)>\frac{1}{60}\diam Q(S)$, \Lemma{david-lemma} implies $Q_{j}=Q(S)$, and since $\diam R_{i}\leq 2\diam Q(S)=2\diam Q_{j}$ by assumption,
\begin{align*}
184\diam R_{j} 
& \stackrel{\eqn{Ri<2Rj}}{\leq}368
\diam R_{i} \leq 736\diam Q(S)=736\diam Q_{j},
\end{align*}
so that in any case, we have 
\[\dist (x_{Q_{j}},R_{i})\leq 184\diam R_{j}\leq 736\diam Q_{j}\]
and
\begin{equation}
\diam R_{i}\leq 4\diam Q_{j}.
\label{e:Ri<4Qj}
\end{equation}
Hence,
\begin{align*}
R_{i} & 
\subseteq B(x_{Q_{j}},\dist(x_{Q_{j}},R_{i})+\diam R_{i})\\
& \subseteq B(x_{Q_{j}},736\diam Q_{j}+4\diam Q_{j}) 
 \subseteq MQ_{j}.
\end{align*}
Furthermore, \eqn{all-sim} follows since $\diam R_{i}\sim_{d} Q_{i}$ by (1) and (2a), and
\[\diam Q_{j}\stackrel{\eqn{Q<R}}{\lec} \diam R_{j} \stackrel{\eqn{Ri<2Rj}}{\lec} \diam R_{i} \stackrel{\eqn{Ri<4Qj}}{\lec} \diam Q_{j}.\]

\end{enumerate}
\end{enumerate}
\end{proof}

\subsection{Controlling the distances between affine maps}

In this section, we show how if $\omega_{f}$ over two intersecting cubes is small, the approximating affine maps in those cubes are approximately the same.

\begin{lemma}
If $A_{1},A_{2}$ are two affine maps and $R$ is any cube, then 
\begin{equation}
|A_{1}'-A_{2}'|\lec_{d} \avint_{R}\frac{|A_{1}-A_{2}|}{\diam R}
\end{equation}
and
\begin{equation}
|A_{1}(x)-A_{2}(x)| \lec_{d} \ps{\avint_{R}\frac{|A_{1}-A_{2}|}{\diam R}}(\dist(x,R)+\diam R) \;\; \mbox{ for all }x\in \bR^{d}.
\label{e:A1-A2}
\end{equation}
\label{l:A1-A2}
\end{lemma}

\begin{proof}
 There is $y\in \frac{1}{2}R$ such that 
\[|A_{1}(y)-A_{2}(y)| \leq \avint_{\frac{1}{2}R}|A_{1}-A_{2}| \leq 2^{d} \avint_{R}|A_{1}-A_{2}| \]
Without loss of generality, we may assume $y=0$. Then, since the norm $|||A|||:=\avint_{B(0,\ell(R))}\frac{|A(z)|}{\diam R}dz$ is a norm on the set of linear maps, it is comparable to the usual operator norm, and in a way that is independent of $\ell(R)$. Thus,
\begin{align*}
|A_{1}'-A_{2}'|
& \lec_{d} \avint_{R}\frac{|A_{1}'(z)-A_{2}'(z)|}{\diam R}dz\\
& \leq \avint_{R}\frac{|A_{1}(z)-A_{2}(z)|}{\diam R}dz+\avint_{R}\frac{|A_{1}(0)-A_{2}(0)|}{\diam R}dz\\
& \leq (1+2^{d}) \avint_{R}\frac{|A_{1}(z)-A_{2}(z)|}{\diam R}dz
\end{align*}
Hence, for $x\in \bR^{d}$,
\begin{align*}
|A_{1}(x)& -A_{1}(x)| 
 \leq |A_{1}'(x)-A_{2}'(x)|+|A_{1}(0)-A_{1}(0)| \\
& \leq |A_{1}'-A_{2}'||x|+2^{d}\avint_{R}|A_{1}-A_{2}|\\
& \lec_{d}  \avint_{R}\frac{|A_{1}(z)-A_{2}(z)|}{\diam R} |x|dz+ \avint_{R}\frac{|A_{1}(z)-A_{2}(z)|}{\diam R} \diam R dz\\
& =\ps{ \avint_{R}\frac{|A_{1}-A_{2}|}{\diam R}} (|x-y|+\diam R).
\end{align*}
\end{proof}

\begin{lemma}
Suppose $Q_{1},Q_{2}\in \Delta$, $f:MQ_{1}\rightarrow MQ_{2}\rightarrow \bR^{D}$ is an integrable function, $\max_{i=1,2}\{\omega_{f}(MQ_{i})\}<\ve$ and $R\subseteq MQ_{1}\cap MQ_{2}$. Then
\begin{equation}
|A_{Q_{1}}'-A_{Q_{2}}'|\lec \ve \ps{\frac{\max_{i}\{| Q_{i}|\} }{| R|}}^{\frac{d+1}{d}} \max_{i=1,2}\{|A_{Q_{i}}'|\}
\end{equation}
and for all $x\in \bR^{d}$,
\begin{equation}
|A_{Q_{1}}(x)-A_{Q_{2}}(x)| 
\lec_{d} \ve \ps{\frac{\max_{i}\{| Q_{i}|\} }{| R|}}^{\frac{d+1}{d}} \max_{i=1,2}\{|A_{Q_{i}}'|\} (\dist(x,R)+\diam R).
\end{equation}
\label{l:Q1-Q2}
\end{lemma}

\begin{proof}

We estimate
\begin{multline}
\avint_{R}|A_{Q_{1}}-A_{Q_{2}}|
\leq \sum_{i=1}^{2} \avint_{R}|A_{Q_{i}}-f|
\leq \sum_{i=1}^{2}  \frac{|MQ_{i}|}{|R|} \avint_{MQ_{i}}|f-A_{Q_{i}}| \\ 
= \sum_{i=1}^{2}  \frac{|MQ_{i}|}{|R|} \omega_{f}(MQ_{i}) |A_{Q_{i}}'| \diam MQ_{i}
< 2M^{d+1}\ps{\frac{\max_{i}\{| Q_{i}|\} }{| R|}}^{\frac{d+1}{d}} \diam R \max_{i=1,2}\{|A_{Q_{i}}'|\}\ve.
 \end{multline}
Now we invoke \Lemma{A1-A2}.
\end{proof}

\begin{lemma}
Let $f:\Omega\rightarrow\bR^{D}$ and $S$ be an $(\ve,\tau)$-region as in \Definition{vt-region}, and $\{R_{i}\}$ be as in \Lemma{david-lemma}. If $2R_{i}\cap 2R_{j}\neq\emptyset$, then
\begin{equation}
|A_{Q_{i}}'-A_{Q_{j}}'|\lec_{d} \ve |A_{Q(S)}'|
\label{e:Ai-Aj1}
\end{equation}
and
\begin{equation}
|A_{Q_{i}}(x)-A_{Q_{j}}(x)| \lec_{d} \ve|A_{Q(S)}'|(\dist(x,R_{i})+\diam R_{i})\mbox{ for all }x\in \bR^{d}.
\label{e:Ai-Aj2}
\end{equation}
\label{l:Ai-Aj}
\end{lemma}

\begin{proof}
Note that if $\min\{\diam R_{i},\diam R_{j}\}\geq \frac{1}{60}\diam Q(S)$, then $Q_{i}=Q_{j}=Q(S)$ by \Lemma{david-lemma}, and so \eqn{Ai-Aj1} and \eqn{Ai-Aj2} hold trivially.

Otherwise, if $\diam R_{i}<\frac{1}{60}\diam Q(S)<2\diam Q(S)$, then \Lemma{2Rcap2R} implies $R_{i}\subseteq MQ_{i}\cap MQ_{j}$ and that $\diam Q_{i}\sim \diam R_{i}\sim \diam R_{j}\sim \diam Q_{j}$. Hence, the lemma follows from \Lemma{Q1-Q2} and \eqn{A<AS}.
\end{proof}

\subsection{Extensions and the proof of \Lemma{m2-estimate}}
\label{s:extensions}

This section is dedicated to the proof of \Lemma{m2-estimate}. 

\begin{proof}[Proof of \Lemma{m2-estimate}]
Let $f$ and $S$ be as in \Lemma{m2-estimate} and let $R_{i}, Q_{i}$ be as in \Lemma{david-lemma}. Note that if $R_{i}$ and $R_{j}$ are adjacent in the sense that their boundaries intersect, then by \eqn{Ri<2Rj}, 
\begin{equation}
\frac{9}{8}R_{i}\cap \frac{1}{2}R_{j}=\emptyset.
\label{e:9/8Rj}
\end{equation}
 This and the fact that $\sum_{j}R_{j}=\bR^{d}\backslash \cnj{z(S)}$ mean we can pick $\{\phi_{j}\}$ a partition of unity subordinate to the collection $\{2R_{j}\}$ so that 
\begin{equation}  \phi_{j}\leq \one_{\frac{9}{8}R_{j}} \leq \one_{2R_{j}},\end{equation}
\begin{equation}
\sum_{j}\phi_{j}\equiv \one_{\bR^{d}\backslash \cnj{z(S)}}, \;\;\; \sum_{j}\grad\phi_{j}\equiv 0 \mbox{ on } \bR^{d}\backslash \cnj{z(S)}.
\label{e:phi}
\end{equation}
and for all indices $\alpha$,
\begin{equation}|\d^{\alpha}\phi_{j}|\lec_{d} \diam(R_{j})^{-|\alpha|}\one_{2R_{j}}.
\label{e:phi-est}
\end{equation}
Observe that by \eqn{9/8Rj}, we know that
\begin{equation}
\one_{\frac{1}{2} R_{i}}\leq \phi_{i} \leq \one_{(\frac{1}{2}R_{j})^{c}}\mbox{ for all } i\neq j.
\label{e:phi-avoids-centers}
\end{equation}

Now, define a map $F_{S}:\bR^{d}\rightarrow \bR^{D}$ by
\begin{equation}
F_{S}(x)=\sum_{j}A_{Q_{j}}(x)\phi_{j}(x)\one_{\bR^{d}\backslash\cnj{z(S)}} + f(x)\one_{\cnj{z(S)}}.
\label{e:FS}
\end{equation}

The remainder of the proof depends on two lemmas: one showing that $DF_{S}$ deviates from $A_{S}'$ a lot near $M_{2}(S)$, and the other showing that $DF_{S}$ doesn't deviate from $A_{S}'$ much overall, thus $M_{2}(S)$ must have small measure.

\begin{lemma}
For $\ve<\ve'(d,\eta)$ , $f:\Omega\rightarrow \bR^{D}$ and $S$ an $(\ve,\tau)$-region as in \Lemma{m2-estimate},
\begin{equation}
||DF_{S}-A_{Q(S)}'||_{2}^{2}\gec_{d}|A_{Q(S)}'|^{2}\tau^{2}|M_{2}(S)|.
 \label{e:m2S<D-A}
 \end{equation}
\label{l:FS>tau}
\end{lemma}

\begin{lemma}
For $\ve<\ve'(d,\eta)$ , $f:\Omega\rightarrow \bR^{D}$ and $S$ an $(\ve,\tau)$-region as in \Lemma{m2-estimate},
\begin{equation}
\sum_{Q\in \Delta}\Omega_{F_{S}}(2Q)^{2}|Q|\lec_{D}\ve^{2}|A_{Q(S)}'|^{2}.
\label{e:OmFsum}
\end{equation}
\label{l:OmFsum}
\end{lemma}

We'll postpone their proofs to Sections \ref{s:FS>tau} and \ref{s:OmFsum} for now and complete the proof of \Lemma{m2-estimate}. By Lemmas \ref{l:dyadicdorronsoro}, \ref{l:FS>tau}, and \ref{l:OmFsum}, and since $\Omega_{F_{S}}=\Omega_{F_{S}-A_{Q(S)}}$, 

\begin{align*}
|M_{2}(S)|
& \lec_{d} \frac{||DF_{S}-A_{Q(S)}'||_{2}^{2}}{\tau^{2}|A_{Q(S)}'|^{2}}
\lec_{D}\frac{\sum_{Q\in \Delta}\Omega_{F_{S}-A_{S}}(2Q)^{2}|Q|}{{\tau^{2}|A_{Q(S)}'|^{2}}}
 \lec_{D}
 \ps{\frac{\ve}{\tau}}^{2}|Q(S)|,
 \end{align*}
so that for $\upsilon=\upsilon(D)>0$ small enough, if $\ve<\upsilon\tau$, we can guarantee that $|M_{2}(S)|<\frac{1}{2}|Q(S)|$. This proves \Lemma{m2-estimate}, so long as we prove Lemmas \ref{l:FS>tau} and \ref{l:OmFsum}, which will be the focus of the next two sections.

\end{proof}

\subsection{Bounding $M_{2}(S)$ and the proof of \Lemma{FS>tau}}
\label{s:FS>tau}

\begin{proof}[Proof of \Lemma{FS>tau}]

Let $N=\ceil{\log_{2}(40\sqrt{d})}+2$. If $Q\in m_{2}(S)$, let $R$ be the dyadic cube containing $x_{Q}$ such that $R^{N}=Q$. Note that if $Q'\in S$, then $Q'$ cannot be properly contained in  $Q$ since $Q\in m_{1}(S)\subseteq m(S)$, so either 
\begin{enumerate}
\item $Q'\supseteq Q$, in which case 
\[ \diam Q'+\dist(R,Q')\geq \diam Q,\] 
\item or $Q' \not\subseteq Q$, in which case $Q'$ and $Q$ have disjoint interiors, and since $R\subseteq Q$, we have
\[ \diam Q' + \dist(R,Q')> \dist(R,Q')\geq \ell(Q)-\ell(R)=(1-2^{-N})\ell(Q)>\frac{\ell(Q)}{2}.\]
\end{enumerate}
Thus, if we infimize over all such $Q'\in S$, we get $D_{S}(R)\geq \ell(Q)/2$. By our choice of $N$,
\begin{equation}
D_{S}(R)  \geq  \frac{\ell(Q)}{2}= 2^{N-1}\ell(R)=2^{N-1}d^{-\frac{1}{2}}\diam R > 20\diam R,
\label{e:D>20R}
\end{equation}
and hence there must be $R_{i}\supseteq R$. Since $D_{S}(Q')\leq \diam Q'$ for all $Q'\supseteq Q$ with $Q'\in S$, we know that $R_{i}\subseteq Q$ (otherwise \eqn{D>20R} wouldn't hold). Thus
\begin{equation}
2^{-N}\diam Q= \diam R\leq \diam R_{i}\leq \diam Q.
\label{e:2^-N<Rj}
\end{equation}
By \eqn{phi-avoids-centers},
\begin{equation}
F_{S}(x)=\sum_{j}A_{Q_{j}}(x)\phi_{j}(x) =A_{Q_{i}}(x) \mbox{ for }x\in \frac{1}{2}R_{i}.
\label{e:FSon1/2R}
\end{equation}
Hence,
\begin{equation}
DF_{S}(y)=A_{Q_{i}}' \mbox{ for all }y\in \frac{1}{2}R_{i}.
\label{e:DF=AQi}
\end{equation}

Note that $Q,Q_{i}\in S$, so that $\omega_{f}(MQ)<\ve$ and $\omega_{f}(MQ_{i})<\ve$. Since $R\subseteq Q\subseteq Q(S)$,  we have $R_{i}\subseteq Q\cap MQ_{i}$ by \Lemma{2Rcap2R}, and so
\begin{equation}
 \diam Q_{i}\sim \diam R_{i} \stackrel{\eqn{2^-N<Rj}}{\sim}_{d} \diam Q.
 \label{e:QisimRisimQ}
 \end{equation}
 Hence, \Lemma{Q1-Q2} implies 
\[|A_{Q_{i}}'-A_{Q}'|\leq  C_{1}\ve|A_{Q(S)}'|.\]
for some $C_{1}=C_{1}(d)>0$.
Since $Q\in m_{2}(S)$, by \eqn{m2} we know that there is a child $Q''$ of $Q$ for which
\begin{equation}
\omega_{f}(MQ'')<\ve\mbox{ and } |A_{Q''}'-A_{Q(S)}'|>\tau|A_{Q(S)}'|.
\label{e:wfMQ''}
\end{equation}
Hence, again by \Lemma{Q1-Q2}, there is $C_{2}=C_{2}(d)>0$ so that
\[|A_{Q}'-A_{Q''}'|\leq C_{2}\ve \max\{|A_{Q}'|,|A_{Q''}'|\}.\]
This means
\[|A_{Q''}'|\leq (1+C_{2}\ve)|A_{Q}'|\stackrel{\eqn{A<AS}}{\leq} (1+C_{2}\ve)(1+\tau)|A_{Q(S)}'|\]
since $Q\in S$, so that 
\[|A_{Q}'-A_{Q''}'|\leq \underbrace{C_{2}(1+C_{2}\ve)(1+\tau)}_{=:C_{3}}\ve|A_{Q(S)}'|.\]
Thus, for $\ve<2\tau^{-1}(C_{1}+C_{3})$ and $y\in \frac{1}{2}R_{i}$,
\begin{align*}
|DF_{S}(y)-A_{Q(S)}'|
& \stackrel{\eqn{FSon1/2R}}{=}|A_{Q_{i}}'-A_{Q(S)}'|\\
& \geq |A_{Q(S)}'-A_{Q''}'| -|A_{Q''}'-A_{Q}'|-|A_{Q}'-A_{Q_{i}}'|\\
& \stackrel{\eqn{wfMQ''}}{\geq} (\tau-(C_{3}+C_{1}) \ve) |A_{Q(S)}'| \geq \frac{\tau}{2}|A_{Q(S)}'|.
\end{align*}
Hence, 
\[
\int_{Q}|DF_{S}(y)-A_{Q(S)}'|^{2}dy
  \geq \int_{\frac{1}{2}R_{i}}\ps{\frac{\tau}{2} |A_{Q(S)}'|}^{2}
 =2^{-d}|R_{i}|\frac{\tau^{2}}{4}|A_{Q(S)}'|^{2}
 \stackrel{\eqn{QisimRisimQ}}{\gec_{d}} |Q|\tau^{2}|A_{Q(S)}'|^{2}.
\]
Thus,
\begin{align*}
||DF_{S}-A_{Q(S)}'||_{2}^{2}
& \geq \sum_{Q\in m_{2}(S)}\int_{Q} |DF_{S}-A_{Q(S)}'|^{2}\\
& \gec_{d}\tau^{2} \sum_{Q\in m_{2}(S)}|Q| |A_{Q(S)}'|^{2}
 =\tau^{2} |m_{2}(S)||A_{Q(S)}'|^{2}
\end{align*}

\end{proof}

\subsection{The proof of \Lemma{OmFsum}}
\label{s:OmFsum}

Throughout this section (and its subsections), we have the standing assumption that $0<\ve<\ve'(d,\eta)$, $\tau\in (0,1)$, $S$ is an $(\ve,\tau)$-region for an $\eta$-quasisymmetric map $f:\Omega\rightarrow\bR^{D}$ as in \Lemma{m2-estimate}, and $F_{S}$ is constructed as in \eqn{FS}.

To estimate \eqn{OmFsum}, we divide the sum into three parts

\begin{equation}
\sum_{Q\in \Delta}\Omega_{F_{S}}(2Q)^{2}|Q| = \sum_{i=1}^{3}\sum_{Q\in \Delta_{i}}\Omega_{F_{S}}(2Q)^{2}|Q|
\label{e:omega-parts}
\end{equation}
where
\begin{align}
\Delta_{1} & =\ck{Q \in \Delta: \frac{1}{20}D_{S}(Q)<\diam Q \leq \diam Q(S)} \supseteq S
\label{e:Delta1} \\
\Delta_{2}
& =\ck{Q\in \Delta:  \diam Q\leq \frac{1}{20}D_{S}(Q) }  = \bigcup_{j}\{Q\in \Delta: Q\subseteq R_{j}\} \label{e:Delta2} \\
\Delta_{3} & =\ck{Q\in \Delta:  \diam Q> \max\{\frac{1}{20}D_{S}(Q),\diam Q(S)\} }. \label{e:Delta3}
\end{align}
 We will estimate each one separately over the next three subsections.

\subsubsection{$\Delta_{1}$}

In this section, we focus on proving the following lemma.

\begin{lemma}
\begin{equation}
\sum_{Q\in \Delta_{1}}\Omega_{F_{S}}(2Q)^{2}|Q|\lec_{d} \ve^{2}|A_{Q(S)}'|^{2}.
\label{e:delta1}
\end{equation}
\label{l:delta1}
\end{lemma}

\begin{proof}

We first need a few technical lemmas.

\begin{lemma}
If $Q\in \Delta_{1}$ and $2R_{j}\cap 2Q\neq\emptyset$, then 
\begin{equation}
\diam R_{j}\leq  \diam Q\leq \diam Q(S).
\label{e:Rj<Q}
\end{equation}
\label{l:Rj<Q}
\end{lemma}

\begin{proof}
The second inequality follows from the definition of $\Delta_{1}$, so we focus on the first. Let $y\in Q$ be such that 
\begin{equation}
D_{S}(y)=D_{S}(Q)<20\diam Q
\label{e:DS<20Q}
\end{equation}
and let $x\in R_{j}$ be closest to $y$. If $z\in 2R_{j}\cap 2Q$, then 
\begin{equation}
|x-y|\leq |x-z|+|z-y|\leq 2\diam R_{j}+2\diam Q
\label{e:1/2R+2Q}
\end{equation}
Thus,
\begin{align*}
\diam R_{j} & \leq \frac{1}{20}D_{S}(x)
\leq \frac{1}{20}(D_{S}(y)+|x-y|)\\
& \stackrel{\eqn{DS<20Q}\atop \eqn{1/2R+2Q}}{\leq} \frac{1}{20}\ps{20\diam Q+2\diam R_{j}+2\diam Q}\\
& =\frac{11}{10}\diam Q+\frac{1}{10}\diam R_{j}
\end{align*}
A bit of arithmetic shows that
\[\diam R_{j}\leq \frac{11}{9}\diam Q<2\diam Q.\]
Since $\frac{\diam R_{j}}{\diam Q}$ is an integer power of two, we in fact know $\diam R_{j}\leq \diam Q$, which proves the lemma.
\end{proof}

\begin{lemma}
If $Q\in \Delta_{1}$, then either $Q\in S$ or $Q\supsetneq R_{j}$ for some $R_{j}$ with $120\sqrt{d}\diam R_{j}\geq \diam Q$.
\label{l:RinQ}
\end{lemma}

\begin{proof}
Let $Q\in \Delta_{1}\backslash S$ so that $D_{S}(Q)/20<\diam Q\leq \diam Q(S)$. 

{\it Step 1:} We first show that $Q$ is not contained in any $R_{j}$. If $Q\subseteq R_{j}$ for some $R_{j}$, then for all $x\in Q\subseteq R_{j}$,
\[\diam Q\leq \diam R_{j}\stackrel{\eqn{RD}}{\leq}\frac{1}{20}D_{S}(x)\]
and infimizing over all $x\in Q$, we get $\diam Q\leq \frac{1}{20}D_{S}(Q)$, a contradiction since $Q\in \Delta_{1}$.

{\it Step 2:} Next, we show there is $R_{i}$ so that $x_{Q}\in R_{i}\subsetneq Q$. If $Q^{\circ}\cap z(S)\neq\emptyset$, then $Q\subseteq Q(S)$ since $\diam Q\leq \diam Q(S)$, and there exists $z\in Q^{\circ}\cap z(S)$. Since $z\in z(S)$, there are arbitrarily small cubes in $S$ containing $z$ (otherwise the smallest one would be a minimal cube, implying $z\not\in z(S)$), infinitely many of which intersect $Q^{\circ}$, so $Q$ contains a cube in $S$ and by the coherence of $S$, $Q\in S$, a contradiction since we assumed $Q\in \Delta_{1}\backslash S$. Hence, we know $Q^{\circ}\cap z(S)=\emptyset$. Thus, $Q^{\circ}\subseteq \bR^{d}\backslash\cnj{z(S)}=\bigcup R_{j}$. Since $Q$ is not contained in any $R_{j}$, there is an $R_{i}$ such that $x_{Q}\in R_{i}\subsetneq Q$.

{\it Step 3:} Now we estimate the size of $R_{i}$. Let $Q'\in S$. If $Q'\subseteq Q$, then $Q\subseteq Q(S)$ since $\diam Q\leq \diam Q(S)$, and by the coherence of $S$, $Q\in S$, a contradiction since $Q\not\in S$. Thus, we know $Q'\not\subseteq Q$, so either $Q'$ and $Q$ have disjoint interiors (in which case $\dist(x_{Q},Q')\geq \frac{1}{2}\ell(Q)$) or $Q'\supsetneq Q$ (in which case $\diam Q'\geq 2\diam Q$). Hence,
\begin{align*}
 60\diam R_{j} 
 & \stackrel{\eqn{RD}}{\geq} D_{S}(x_{Q}) =\inf_{Q'\in S}\{\dist (x_{Q},Q')+\diam Q'\} 
  \geq \min\{ \ell(Q)/2,2\diam Q\} \\
&   =\frac{\ell(Q)}{2}
  = \frac{\diam Q}{2\sqrt{d}}
 \end{align*}
which implies the lemma.

\end{proof}

\begin{lemma}
For $Q\in \Delta_{1}$, pick a cube $\tilde{Q}\in S$ as follows. If $Q\in S$, set $\tilde{Q}=Q$. Otherwise, let $\tilde{Q}=Q_{j}$, where $R_{j}$ is as in \Lemma{RinQ}. Then $2Q\subseteq M\tilde{Q}$ and $\diam \tilde{Q}\leq 180\diam Q$.
\label{l:Qt}
\end{lemma}

\begin{proof}
The lemma is clearly true if $Q\in S$, since then $\tilde{Q}=Q$, so suppose $Q\not\in S$. Since $\diam R_{j}\leq \diam Q<2\diam Q(S)$, by Lemmas \ref{l:david-lemma}  and \ref{l:RinQ} we have 
\[\diam Q\leq 120\sqrt{d}\diam R_{j}\stackrel{\eqn{R<Q}}{\leq} 240 \sqrt{d} \diam Q_{j} \leq 240d\ell(Q_{j}).\]
and
\begin{align*}
\dist(x_{Q_{j}},Q)
& \stackrel{R_{j}\subseteq Q}{\leq} \dist(x_{Q_{j}},R_{j})
\stackrel{\eqn{dxQR}}{\leq} 180\diam R_{j}\stackrel{\eqn{R<Q}}{\leq} 360\diam Q_{j}   = 360 \sqrt{d} \ell(Q_{j})
\end{align*}
Hence, the above two inequalities give
\begin{align*}
2Q  &  \subseteq B(x_{Q_{j}},\dist(x_{Q_{j}},Q)+\diam 2Q)
\subseteq B(x_{Q_{j}},(360\sqrt{d} +240d)\ell(Q_{j})) 
 \subseteq MQ_{j}.
\end{align*}
For the last part of the lemma, observe that since $R_{j}\subseteq Q$,
\[ \diam\tilde{Q}=\diam Q_{j}\stackrel{\eqn{Q<R}}{\leq} 180\diam R_{j}\leq 180\diam Q.\]
\end{proof}

We now proceed with the proof of \Lemma{delta1}. For $Q\in \Delta_{1}$, let $\tilde{Q}\in S$ be as in \Lemma{Qt}.  Using the inequality $(a+b)^{2}\leq 2(a^{2}+b^{2})$, we get
\begin{multline}
\sum_{Q\in \Delta_{1}} \frac{\Omega_{F_{S}}(2Q)^{2}|Q|}{|A_{Q(S)}'|^{2}}
 \leq \sum_{Q\in \Delta_{1}}  \avint_{2Q} \ps{\frac{|F_{S}-A_{Q(S)}|}{|A_{Q(S)}'|\diam 2Q}}^{2}|Q| \\
    \leq \frac{1}{2}\sum_{Q\in \Delta_{1}} \avint_{2Q}\ps{\frac{|f-A_{Q(S)}|}{|A_{Q(S)}'|\diam Q}}^{2}|Q|+ \frac{1}{2}\sum_{Q\in \Delta_{1}} \avint_{2Q}\ps{\frac{|f-F_{S}|}{|A_{Q(S)}'|\diam Q}}^{2}|Q|.
\label{e:omega/AS1}
 \end{multline}
We'll estimate the two summands separately, starting with the first. 

Recall from \eqn{A<AS} that for $Q\in\Delta_{1}$, since $\tilde{Q}\in S$ and $\tau\in (0,1)$, we have $|A_{\tilde{Q}}'|\leq 2|A_{Q(S)}'|$, and by \Lemma{Qt}, we know $2Q\subseteq M\tilde{Q}$ and $\diam \tilde{Q}\leq 180\diam Q$.  Hence $\diam \tilde{Q}\sim \diam Q$ and so
 \begin{align}
\sum_{Q\in \Delta_{1}}  & \avint_{2Q}\ps{\frac{|f-A_{Q(S)}|}{|A_{Q(S)}'|\diam Q} }^{2}|Q| 
   \\
   &  \leq\sum_{Q\in \Delta_{1}}   \ps{\frac{|A_{\tilde{Q}}'|M\diam \tilde{Q}}{|A_{Q(S)}'|\diam Q}}^{2}  \frac{|M\tilde{Q}|}{|2Q|} \avint_{M\tilde{Q}} \ps{\frac{|f-A_{\tilde{Q}}|}{|A_{\tilde{Q}}'|\diam M\tilde{Q}}}^{2}|Q| \notag \\
&     \sim_{d} \sum_{Q\in \Delta_{1}} \omega_{f}(M\tilde{Q})^{2}|Q|
 \leq \sum_{R\in S}\sum_{Q\in \Delta_{1}: \tilde{Q}=R} \omega_{f}(MR)^{2}|Q| 
  \lec_{d}  \sum_{R\in S}\omega_{f}(MR)^{2}|R| 
\label{e:sumOmega/AS21}
\end{align}
where in the last line we used the fact that  if $R\in S$, then the number of cubes $Q\in\Delta_{1}$ such that $\tilde{Q}=R$ is uniformly bounded by a number depending only on $d$ (since all those cubes $Q$ have size  comparable to $\diam R$ and are contained in $MR$ by \Lemma{Qt}).

Next, since $S$ is a $(\ve,\tau)$-region, we have that $\sum_{Q\subseteq R\subseteq Q(S)}\omega_{f}(MR)^{2}<\ve^{2}$ for all $Q\in S$, thus

\begin{align}
\sum_{R\in S} & \omega_{f}(MR)^{2}|R|
 =\int_{Q(S)}\sum_{R\in S}\omega_{f}(MR)^{2}\one_{R} \notag \\
& = \int_{z(S)}\sum_{R\in S}\omega_{f}(MR)^{2}\one_{R} + \sum_{Q\in m(S)}\int_{Q}\sum_{R\in S}\omega_{f}(MR)^{2}\one_{R} \notag \\
& =\int_{z(S)}\sum_{x\in R\in S}\omega_{f}(MR)^{2}dx+\sum_{Q\in m(S)}\sum_{Q\subseteq R\in S}\omega_{f}(MR)^{2}|Q| \notag \\
& <\ve^{2}|z(S)|+\sum_{Q\in m(S)}\ve^{2}|Q|=\ve^{2}|Q(S)|. \notag \\
\label{e:sumoverS}
\end{align}
Thus, 
\begin{equation}
\sum_{Q\in \Delta_{1}}  \avint_{2Q}\ps{\frac{|f-A_{Q(S)}|}{|A_{Q(S)}'|\diam Q}}^{2}|Q|
\stackrel{\eqn{sumOmega/AS21}}{\lec}_{d}   \sum_{R\in S}\omega_{f}(MR)^{2}|R|  
\stackrel{\eqn{sumoverS}}{<}\ve^{2}|Q(S)|
\label{e:sumOmega/AS2}
\end{equation}
which shows the  first sum in \eqn{omega/AS1} is at most a constant (depending on $d$) times $\ve^{2}|Q(S)|$. 

For the second sum in \eqn{omega/AS1}, set 
\[I_{Q}=\{j:2R_{j}\cap 2Q\neq\emptyset\}.\] 
Recall that $\supp\phi_{j}\subseteq 2R_{j}$ and by \Lemma{david-lemma} we have $ \sum \one_{2R_{j}}\lec_{d} \one_{\bR^{d}\backslash \cnj{z(S)}}$. Hence, by the definition of $F_{S}$, \Lemma{2Rcap2R}, the fact that $f=F_{S}$ on $\cnj{z(S)}$, and because $|A_{Q_{j}}'|\leq 2 |A_{Q(S)}'|$ by \eqn{A<AS}, we have
\begin{align}
\avint_{2Q} & \ps{\frac{|f-F_{S}|}{|A_{Q(S)}'|\diam Q}}^{2}
 \leq \avint_{2Q}\ps{\sum_{j} \frac{|f-A_{Q_{j}}|}{|A_{Q(S)}'|\diam Q}\phi_{j}}^{2} \notag \\
 & \stackrel{\eqn{bounded-overlap}}{\lec_{d}} \sum_{j}  \avint_{2Q}\ps{\frac{|f-A_{Q_{j}}|}{|A_{Q(S)}'|\diam Q}\phi_{j}}^{2} 
  \stackrel{\eqn{A<AS}}{\lec}  \frac{1}{|2Q|}\sum_{j\in I_{Q}}\int_{2R_{j}}\ps{\frac{|f-A_{Q_{j}}|}{|A_{Q_{j}}'|\diam Q}}^{2} \notag \\
&=  \frac{M^{2}}{|2Q|}\sum_{j\in I_{Q}}  |MQ_{j}| \avint_{MQ_{j}}\ps{\frac{|f-A_{Q_{j}}|}{|A_{Q_{j}}'|\diam MQ_{j}}\frac{\diam Q_{j}}{\diam Q}}^{2} \notag \\
&  =\frac{M^{d+2}}{2^{d}} \frac{1}{|Q|}\sum_{j\in I_{Q}} \omega_{f}(MQ_{j})^{2}\ps{\frac{\diam Q_{j}}{\diam Q}}^{2}|Q_{j}| .
\label{e:omega/AS3}
\end{align}

Recall by \Lemma{Rj<Q} that if $Q\in \Delta_{1}$ and $2R_{j}\cap 2Q\neq\emptyset$, then $\diam R_{j}\leq \diam Q$. Hence, if $n\geq 0$ and $I_{j,n}$ is the set of such cubes $Q\in\Delta_{1}$ with $2R_{j}\cap 2Q\neq\emptyset$ and $\ell(Q)=2^{n}\ell(R_{j})$, then $\# I_{j,n}\lec_{d}1$. Thus, for a fixed $j$,
\begin{equation}
\sum_{Q\in \Delta_{1} \atop 3Q\cap 2R_{j}\neq\emptyset} \ps{\frac{\diam Q_{j}}{\diam Q}}^{2}
\stackrel{\eqn{Q<R}}{\leq}
\sum_{n\geq 0}\sum_{Q\in I_{j,n}}\ps{\frac{180\diam R_{j}}{\diam Q}}^{2}
\lec_{d} \sum_{n\geq 0}2^{-2n}\lec 1.
\label{e:Ijn}
\end{equation}
Therefore, 
\begin{align}
\sum_{Q\in \Delta_{1}} \avint_{2Q} & \ps{\frac{|f-F_{S}|}{|A_{Q(S)}'|\diam Q}}^{2}|Q| 
 \stackrel{\eqn{omega/AS3}}{\lec_{d}}
\sum_{Q\in \Delta_{1}}\sum_{j\in I_{Q}} \omega_{f}(MQ_{j})^{2}\ps{\frac{\diam Q_{j}}{\diam Q}}^{2}|Q_{j}| \notag \\
  & =\sum_{\diam R_{j}\leq \diam Q(S)} \omega_{f}(MQ_{j})^{2}|Q_{j}| \sum_{Q\in \Delta_{1} \atop 2R_{j}\cap 2Q\neq\emptyset}\ps{\frac{\diam Q_{j}}{\diam Q}}^{2}\notag \\
&  \stackrel{\eqn{Ijn}}{\lec_{d}} \sum_{\diam R_{j}\leq \diam Q(S)} \omega_{f}(MQ_{j})^{2}|Q_{j}| \notag \\
&  \leq \sum_{R\in S} \#\{j:Q_{j}=R, \diam R_{j}\leq \diam Q(S)\}\omega_{f}(MR)^{2}|R| \notag \\
&  \lec_{d} \sum_{R\in S}\omega_{f}(MR)^{2}|R|\leq  \ve^{2}|Q(S)|.
 \label{e:sumOmega/AS3}
\end{align}
where, to get to the last line, we used the fact that if $\diam R_{j}\leq \diam Q(S)$, then \Lemma{2Rcap2R} implies $\#\{j:Q_{j}=R,\diam R_{j}\leq Q(S)\}\lec_{d} 1$. 

Combining \eqn{omega/AS1}, \eqn{sumOmega/AS2}, and \eqn{sumOmega/AS3} together, we obtain
\[
\sum_{Q\in \Delta_{1}}\frac{\Omega_{F_{S}}(2Q)^{2}|Q|}{|A_{Q(S)}'|^{2}}
 \lec_{d}  \ve^{2}|Q(S)|
\]

\end{proof}

%
%

\subsubsection{$\Delta_{2}$}

In this section, we will focus on proving the following lemma.

\begin{lemma}
\begin{equation}
 \sum_{Q\in \Delta_{2}}\Omega_{F_{S}}(2Q)^{2}|Q|\lec_{d} \ve^{2}|A_{Q(S)}|^{2} |Q(S)|.
 \end{equation}
\label{l:Delta2}
\end{lemma}

The main idea is that near a cube $Q\in \Delta_{2}$, $F_{S}$ is smooth and so we can get better control of $\Omega_{F_{S}}$ using Taylor's theorem.

\begin{lemma}
For all $j$
\begin{equation}
\frac{1}{|R_{j}|}\sum_{Q\subseteq R_{j}}\Omega_{F_{S}}(2Q)^{2}|Q| \lec_{d} \ve^{2}|A_{Q(S)}'|^{2}.
\label{e:sumOinRj}
\end{equation}
\label{e:Delta2lemma}
\end{lemma}

\begin{proof}
For normed vector spaces spaces $U,V$, let $\cL(U,V)$ denote the set of bounded linear transformations from $U$ into $V$ and write $\cL(U)=\cL(U,U)$. Then $\cL(U,V)$ is also a normed space with the operator norm, which we will also denote $|\cdot|$. For vectors $u,v\in \bR^{d}$, $u\otimes v\in \cL(\bR^{d})$ is the linear transformation defined by $(u\otimes v)(x)=\ip{v,x}u$; for $A\in \bR^{d}$, $A\otimes v,v\otimes A\in \cL(\bR^{d},\cL(\bR^{d}))$ are the linear transformations $(A\otimes v)(x)=\ip{v,x}A$ and $(v\otimes A)(x)=v\otimes (A(x))$ respectively.

Let $y\in  \bR^{d}\backslash\cnj{z(S)}$. Since $F_{S}|_{\bR^{d}\backslash\cnj{z(S)}}$ is smooth, 

\begin{equation}
|D^2F_{S}(y)|=
 \sup_{|u|=|v|=1} \av{\ps{\sum_{m,n=1}^{d} u_{m}v_{n}\frac{\d^2 F_{S,l}}{\d x_{m} \d x_{n}}(y)}_{l=1}^{D}}
\label{e:DF^2}
\end{equation}
where $F_{S,l}$ is the $l$th component of the vector function $F_{S}$ and $D^{2}F_{S}(y)\in \cL(\bR^{d},\cL(\bR^{d}))$ is the derivative of the map $y\mapsto DF_{S}(y)\in \cL(\bR^{d},\bR^{D})$ at $y$, (so above, $|\cdot|$ also denotes this operator norm). Let $R_{i}$ be such that $y\in R_{i}$. Then, if $A$ denotes the first order Taylor approximation to $F_{S}$ at $x=x_{Q}$, then 
\begin{align}
\Omega_{F(S)} & (2Q)\diam 2Q
 \leq \sup_{|u|=1\atop y\in 2Q}u\cdot (F_{S}(y)-A(y)) =\sup_{|u|=1\atop y\in 2Q}\sum_{l=1}^{D}u_{l} (F_{S,l}(y)  -A(y)) \notag \\
& \sup_{|u|=1\atop y\in 2Q}\sum_{l=1}^{D} \int_{0}^{1}\sum_{m=1}^{d}u_{l}\frac{\d (F_{S,l}-A)}{\d x_{m}}(x+t(y-x))(y_{m}-x_{m})dt \notag \\
& \sup_{|u|=1\atop y\in 2Q}\sum_{l=1}^{D}  \int_{0}^{1}\int_{0}^{1}\sum_{m,n=1}^{d}t u_{l}\frac{\d^{2} F_{S,l}}{\d x_{m} \d x_{n}}(x+st(y-x))(y_{m}-x_{m})(y_{n}-x_{n}) dtds \notag \\
& \leq  \sup_{y\in 2Q}  \int_{0}^{1}\int_{0}^{1}t\av{\ps{\sum_{m,n=1}^{d} \frac{\d^{2} F_{S,l}}{\d x_{m} \d x_{n}}(x+st(y-x))(y_{m}-x_{m})(y_{n}-x_{n})}_{l=1}^{D}} dtds \notag \\
& \stackrel{\eqn{DF^2}}{\leq } (\diam Q)^{2}\sup_{y\in 2Q}|D^{2}F_{S}(y)|.
\label{e:taylor2}
\end{align}

For $y\in R_{i}$, if $\phi_{j}(y)\neq 0$, then $y$ is also in $2R_{j}$, and so we may use \Lemma{Ai-Aj} and the fact that $\d^{\alpha}\phi_{j}\lec_{d,\alpha}(\diam R_{j})^{-|\alpha|}$ to estimate
\begin{align}
| D^{2}  F_{S}(y)|
& =\av{\sum_{j}\ps{2A_{Q_{j}}' \otimes \grad\phi_{j}(y)+A_{Q_{j}}(y)\otimes D^{2}\phi_{j}(y)}} \notag \\
&\hspace{-5pt}  \stackrel{\eqn{phi}}{=}\av{\sum_{j}\ps{2(A_{Q_{j}}'-A_{Q_{i}}')  \otimes \grad\phi_{j}(y) + (A_{Q_{j}}(y)-A_{Q_{i}}(y))\otimes D^{2}\phi_{j}(y)}} \notag \\
& \stackrel{\eqn{Ai-Aj1} \atop \eqn{Ai-Aj2}}{\lec_{d}} \sum_{j}\ps{ \ve|A_{Q(S)}'|\grad \phi_{j}(y)|+\ve|A_{Q(S)}'| \diam Q_{i} |D^{2}\phi_{j}(y)|} \notag  \\
& \stackrel{\eqn{phi-est}}{\lec_{d}} \sum_{y\in 2R_{j}}\ps{ \frac{\ve|A_{Q(S)}'|}{\diam R_{j}}+\frac{\ve|A_{Q(S)}'| \diam Q_{i}}{(\diam R_{i})^{2}}} 
\stackrel{\eqn{bounded-overlap}\atop\eqn{Q<R}}{\lec}_{d} \frac{\ve|A_{Q(S)}'|}{\diam R_{i}}.
\label{e:taylor0}
\end{align}

Thus,
\begin{align*}
\sum_{Q\subseteq R_{i}}\Omega_{F_{S}}(2Q)^{2}|Q|
& \stackrel{\eqn{taylor2}}{\leq} \sum_{Q\subseteq R_{i}} (\diam Q)^{2}\ps{\sup_{y\in 2Q}|D^{2}F_{S}(y)|}^{2} |Q|\\
& \stackrel{\eqn{taylor0}}{\lec_{d}} \sum_{Q\subseteq R_{i}} \ps{\frac{\ve |A_{Q(S)}'| \diam Q}{\diam R_{i}}}^{2}|Q|\\
& =\ve^{2}|A_{Q(S)}'|^{2}\sum_{n=0}^{\infty}\sum_{Q\subseteq R_{i} \atop \ell(Q)=2^{-n}\ell(R_{i})}2^{-2n}|Q| 
 \lec \ve^{2}|A_{Q(S)}'|^{2}|R_{i}|
\end{align*}

\end{proof}

Define
\[B_{S}=B(x_{Q(S)},3\diam Q(S)).\]

\begin{lemma}
If $\dist(x,Q(S))\geq 2\diam Q(S)$, then $F_{S}(x)=A_{Q(S)}(x)$. In particular, if $2Q\cap B_{S}=\emptyset$, then $\Omega_{F_{S}}(2Q)=0$.
\label{l:F=AQS}
\end{lemma}

\begin{proof}
Let $x\in R_{i}$. If $\dist(x,Q(S))\geq 2\diam Q(S)$, then
\[\diam R_{i}\stackrel{\eqn{RD}}{\geq} \frac{1}{60}D_{S}(x) \geq \frac{1}{60}\dist(x,Q(S))\geq \frac{1}{30}\diam Q(S),\]
so if $x\in 2R_{j}$, then $2R_{i}\cap 2R_{j}\neq\emptyset$, and
\[\diam R_{j}\stackrel{\eqn{Ri<2Rj}}{\geq} \frac{1}{2}\diam R_{i}\geq \frac{1}{60}\diam Q(S),\]
hence $Q_{j}=Q(S)$ by \Lemma{david-lemma}. Since this holds for all $j$ with $2R_{j}\ni x$, we know  that
\[F_{S}(x)=\sum_{j} A_{Q_{j}}(x)\phi_{j}(x) = A_{Q(S)}(x) \mbox{ if }\dist(x,Q(S))\geq 2\diam Q(S).\] 
If $2Q\cap B_{S}=\emptyset$, then $\dist(x,Q(S))\geq 2\diam Q(S)$ for all $x\in 2Q$, hence $F_{S}|_{2Q}\equiv A_{Q(S)}$, so that $\Omega_{F_{S}}(2Q)=0$.  
\end{proof}

\begin{proof}[Proof of \Lemma{Delta2}]
We estimate
\begin{align*}
\sum_{Q\in \Delta_{2}}\Omega_{F_{S}}(2Q)^{2}|Q| 
& =\sum_{j} \sum_{Q\subseteq R_{j}}\Omega_{F_{S}}(2Q)^{2}|Q|  =\sum_{2R_{j}\cap B_{S}\neq \emptyset} \sum_{Q\subseteq R_{j}}\Omega_{F_{S}}(2Q)^{2}|Q| \\
& \stackrel{\eqn{sumOinRj}}{\lec_{d}} \ve^{2}|A_{Q(S)}'|^{2}  \sum_{2R_{j}\cap B_{S}\neq \emptyset}  |R_{j}|.
\end{align*}

The lemma will follow from the previous inequality once we verify
\begin{equation}
 \sum_{2R_{j}\cap B_{S}\neq \emptyset}  |R_{j}| \lec_{d} |Q(S)|.
 \label{e:sumRj}
 \end{equation}
\label{l:3Rin2BS}

If $2R_{j}\cap B_{S}\neq\emptyset$, then $\dist(2R_{j},Q(S))\leq 3\diam Q(S)$, so that
 \begin{align*}
 \diam 2R_{j}
&  =2\diam R_{j}
 \stackrel{\eqn{RD}}{\leq} \frac{1}{10}D_{S}(R_{j})
 \leq \frac{1}{10}(\dist(R_{j},Q(S))+\diam Q(S))\\
&  \leq \frac{1}{10}(\diam R_{j}+\dist(2R_{j},Q(S))+\diam Q(S))\\
&  \leq \frac{1}{10}(\diam R_{j}+4\diam Q(S))
  =\frac{1}{20}\diam 2R_{j}+\frac{2}{5}\diam Q(S)
\end{align*}
 which implies
 \begin{equation}
 \diam 2R_{j}\leq \frac{20}{19}\cdot \frac{2}{5}\diam Q(S)<\diam Q(S) \mbox{ if }2R_{j}\cap B_{S}\neq\emptyset
 \label{e:3RcapBS}
 \end{equation}
 hence, 
 \[2R_{j}\subseteq B(x_{Q(S)},3\diam Q(S)+\diam 2R_{j})\subseteq B(x_{Q(S)},4\diam Q(S)) \subseteq 2B_{S}.\]
This and the disjointness of the $R_{j}$ imply
\[
 \sum_{2R_{j}\cap B_{S}\neq \emptyset}  |R_{j}|\leq |2B_{S}|\lec_{d} |Q(S)|,
 \]
 which proves \eqn{sumRj}.

\end{proof}

\subsubsection{$\Delta_{3}$}

Finally, we estimate the third sum in \eqn{omega-parts}. 

\begin{lemma}
\begin{equation}
\sum_{Q\in \Delta_{3}}\Omega_{F_{S}}(2Q)^{2}|Q|\lec_{d} \ve^{2}|A_{Q(S)}'|^{2}|Q(S)|.
\end{equation}
\label{l:Delta3}
\end{lemma}

\begin{proof}
 Again, set  $B_{S}=B(x_{Q(S)},3\diam Q(S))$. For $n\geq 0$, let
\[B_{n}=\{Q\in \Delta_{3}: 2Q\cap B_{S}\neq\emptyset, \ell(Q)=2^{n}\ell(Q(S))\}.\]
Then by \Lemma{F=AQS},
\begin{align*}
\sum_{Q\in \Delta_{3}} \Omega_{F_{S}}(2Q)^{2}|Q|
& =\sum_{Q\in \Delta_{3} \atop  2Q\cap B_{S}\neq\emptyset}  \Omega_{F_{S}}(2Q)^{2}|Q|
\leq \sum_{Q\in \Delta_{3} \atop  2Q\cap B_{S}\neq\emptyset}  \avint_{2Q}\ps{\frac{|F_{S}-A_{Q(S)}|}{\diam 2Q}}^{2}|Q|\\
& = 2^{-d-2}\sum_{n\geq 0}\sum_{Q\in B_{n}}  \int_{2Q} \ps{\frac{|F_{S}-A_{Q(S)}|}{\diam Q}}^{2}
 \lec_{d}\sum_{n\geq 0} \int_{\bR^{d}} \ps{\frac{|F_{S}-A_{Q(S)}|}{2^{n} \diam Q(S)}}^{2}.\\
\end{align*}
We claim that
\begin{equation}  \int_{\bR^{d}} \ps{\frac{|F_{S}-A_{Q(S)}|}{\diam Q(S)}}^{2} \lec_{d} \ve^{2}|A_{Q(S)}'|^{2}|Q(S)|
\label{e:F-AS/Q}
\end{equation}
after which the lemma will follow from
\begin{align*}
\sum_{Q\in \Delta_{3}} \Omega_{F_{S}}(2Q)^{2}|Q|
& \lec_{d}\sum_{n\geq 0} \int_{\bR^{d}} \ps{\frac{|F_{S}-A_{Q(S)}|}{2^{n} \diam Q(S)}}^{2}\\
& \stackrel{\eqn{F-AS/Q}}{\lec_{d}}\sum_{n\geq 0} 2^{-2n} \ve^{2}|A_{Q(S)}'|^{2}|Q(S)|  \lec \ve^{2}|A_{Q(S)}'|^{2}|Q(S)|.
\end{align*}
Now we prove \eqn{F-AS/Q}. By \Lemma{F=AQS} and the $L^{2}$ triangle inequality
\begin{multline}
\ps{\int \ps{\frac{|F_{S}-A_{Q(S)}|}{\diam Q(S)}}^{2} }^{\frac{1}{2}}
 =\ps{\int_{B_{S}} \ps{\frac{|F_{S}-A_{Q(S)}|}{\diam Q(S)}}^{2}}^{\frac{1}{2}} \\
 \leq  \ps{\int_{B_{S}} \ps{\frac{|F_{S}-f|}{\diam Q(S)}}^{2}}^{\frac{1}{2}} 
 + \ps{\int_{B_{S}} \ps{\frac{|f-A_{Q(S)}|}{\diam Q(S)}}^{2}}^{\frac{1}{2}} .\\
 \label{e:F-AS/Qparts}
 \end{multline}
 We'll estimate the two parts separately. The second part we may bound as follows:
 \begin{multline*}
 \int_{B_{S}} \ps{\frac{|f-A_{Q(S)}|}{\diam Q(S)}}^{2}
 \leq |A_{Q(S)}'|^{2} M^{2} |MQ(S)|  \avint_{MQ(S)}\ps{\frac{|f-A_{Q(S)}|}{|A_{Q(S)}'|\diam MQ(S)}}^{2} \\
 = |A_{Q(S)}'|^{2}M^{d+2}|Q(S)|\omega_{f}(MQ(S))^{2}<|A_{Q(S)}'|^{2}M^{d+2}|Q(S)|\ve^{2}.
 \end{multline*}
 since $Q(S)\in S$ and hence $\omega_{f}(MQ(S))<\ve$ by definition. For the first part of \eqn{F-AS/Qparts}, recall that if $2R_{j}\cap B_{S}\neq\emptyset$, then \eqn{3RcapBS} implies $\diam R_{j}< \diam Q(S)$, so \Lemma{2Rcap2R} implies $R_{j}\subseteq MQ_{j}$. This, \Lemma{F=AQS}, and the fact that $\supp \phi_{j}\subseteq 2R_{j}$ (which have bounded overlap by \Lemma{david-lemma}) imply
 
\begin{multline}
\int_{B_{S}} \ps{\frac{|F_{S}-f|}{\diam Q(S)}}^{2}
\stackrel{\eqn{bounded-overlap}}{\lec_{d}}   \sum_{j}\int_{B_{S}}\ps{\frac{|A_{Q_{j}}-f|}{\diam Q(S)}\phi_{j}}^{2} \\
  \leq \sum_{2R_{j}\cap B_{S}\neq\emptyset} |A_{Q_{j}}'|^{2}\ps{\frac{\diam Q_{j}}{\diam Q(S)}}^{2}M^{2} |MQ_{j}| \avint_{MQ_{j}} \ps{\frac{|A_{Q_{j}}-f|}{|A_{Q_{j}}'|\diam MQ_{j}}}^{2}\\
  = M^{d+2} \sum_{2R_{j}\cap B_{S}\neq\emptyset} |A_{Q_{j}}'|^{2}\ps{\frac{\diam Q_{j}}{\diam Q(S)}}^{2}|Q_{j}|\omega_{f}(MQ_{j})^{2}
  \label{e:intBSFS-F}
  \end{multline}
  Next, recall from \eqn{Q<R}, the definition of $\Delta_{3}$, and \eqn{3RcapBS} that if $2R_{j}\cap B_{S}\neq\emptyset $, then
\begin{equation}
\diam Q(S) \stackrel{\Delta_{3}}{<}\diam Q_{j}\stackrel{\eqn{Q<R}}{\leq} 180\diam R_{j}\stackrel{\eqn{3RcapBS}}{<} 180\diam Q(S).
\label{e:RjQS}
\end{equation}
Moreover, since $Q_{j}\in S$, we know $\omega_{f}(MQ_{j})<\ve$ and $|A_{Q_{j}}'|\leq (1+\tau)|A_{Q(S)}'|$. These facts and \eqn{intBSFS-F} imply that 
\begin{align*}
\int_{B_{S}} \ps{\frac{|F_{S}-f|}{\diam Q(S)}}^{2}
& \lec_{d} \sum_{2R_{j}\cap B_{S}\neq\emptyset} |A_{Q(S)}'|^{2}\ve^{2} |Q_{j}| \leq 180^{d} \sum_{2R_{j}\cap B_{S}\neq\emptyset} |A_{Q(S)}'|^{2}\ve^{2} |R_{j}| \\
& \stackrel{\eqn{sumRj} \atop \eqn{RjQS}}{\lec_{d}} |A_{Q(S)}'|^{2}\ve^{2}|Q(S)|.
\end{align*}

\end{proof}

\subsection{Finding a bi-Lipschitz part}

In this section, we focus on the following theorem.

\begin{theorem}
Let $Q_{0}\in\Delta(\bR^{d})$ and $f:MQ_{0}\rightarrow\bR^{D}$ be $\eta$-quasisymmetric such that 
\[\sum_{Q\subseteq Q_{0}} \omega_{f}(MQ)^{2}|Q|\leq C_{M}|Q_{0}|.\]
Then for all $\theta>0$, there is $L=L(\eta,\theta,D,C_{M})$ and $E\subseteq Q_{0}$ such that $|E|\geq (1-\theta)|Q|$ and $\ps{\frac{\diam f(Q_{0})}{\diam  Q_{0}}}^{-1} f|_{E}$ is $L$-bi-Lipschitz. 
\label{t:1implies4}
\end{theorem}

\begin{proof}[Proof of \Theorem{1implies4}]
Recall that $\omega_{f}$ is  invariant under dilations and translations in the domain of $f$ and under scaling of $f$ by a constant factor. Moreover, if $f$ is $\eta$-quasisymmetric, the map $x\mapsto rf(sx+b)$ is also $\eta$-quasisymmetric for any nonzero $r,s$ and any $b\in \bR^{d}$. Thus, it suffices to prove the theorem in the case  that  $\diam Q_{0}=\diam f(Q_{0})=1$ so $\diam f(Q_{0})/\diam Q_{0}=1$.

Let $\tau\in (0,1)$, $\delta<d^{-1/2}/4$, and 
\[0<\ve<\min\{\ve_{0}(\eta,D,\tau,C_{M}),\ve_{1}(\eta,d,\delta)\}\]
where $\ve_{1}$ is as in \Lemma{epsilondelta} and $\ve_{0}$ as in \Theorem{corona}. By \Theorem{corona}, we may partition $\Delta(Q_{0})$ into a set of ``bad" cubes $\cB$ and a collection of $(\ve,\tau)$-regions $\cF$ so that
\begin{equation}
\sum_{Q\in \cB} |Q|\leq \frac{C_{M}}{\ve^{2}} |Q_{0}|\;\;
\mbox{ and }
\sum_{S\in \cF_{S}}|Q(S)|\leq \ps{4+\frac{2^{d+1}C_{M}}{\ve^{2}}}|Q_{0}|.
\label{e:cF2}
\end{equation}
Let
\[T=\{Q(S):S\in \cF\}\cup \ps{\bigcup_{S\in \cF}m(S)}\cup \cB.\]
Observe that since, for each $S\in \cF$, the cubes in $m(S)$ have disjoint interiors, we know $\sum_{Q\in m(S)}|Q|\leq |Q(S)|$, and hence
\begin{equation}
\sum_{Q\in T} |Q|  = \sum_{Q\in\cB}|Q| + \sum_{S\in \cF}\ps{|Q(S)|+\sum_{Q\in m(S)}|Q|}  \stackrel{ \eqn{cF2}}{\leq} \ps{(1+2^{d+2})\frac{C_{M}}{\ve^{2}}+8}|Q_{0}|.
\label{e:sumQinT}
\end{equation}
Let $N$ be an integer. For $Q\in \Delta$, define
\[k(Q)=\#\{R\in T: R\supseteq Q\}, \;\;\; T_{N}=\{Q\in T:k(Q)\leq N\},\]
and
\[E=Q_{0}\backslash \ps{\bigcup_{Q\not\in T_{N}}Q}.\]
If $x\in E$, let $Q$ be the smallest cube in $T_{N}$ containing $x$. Then $Q=Q(S)$ for some $S\in \cF$, for otherwise, if $Q\in \cB$ or $Q\in m(S)$ for some $S\in \cF$, then the child $R$ of $Q$ containing $x$ is either of the form $Q(S')$ for some $S'\in \cF$ or is in $\cB$ and hence is also in $T$, but since $Q$ was minimal in $T_{N}$, $R\not\in T_{N}$, which means $k(R)\geq N+1$, implying $z\not\in E$, a contradiction. Thus,
\begin{equation}
E=\bigcup_{Q(S)\in T_{N}} z(S).
\label{e:E=UzS}
\end{equation}
Moreover,
\begin{align*}
|Q_{0}\backslash E| & =\av{\bigcup_{Q\in T_{N}}Q} 
 = \av{\bigcup_{k(Q)=N+1}Q}
=\int \sum_{k(Q)=N+1}\one_{Q}
\leq \int \frac{\sum_{Q\in T_{N+1}}\one_{Q}}{N+1} \\
& =\frac{\sum_{Q\in T_{N+1}}|Q|}{N+1}
\stackrel{\eqn{sumQinT}}{ \leq} \frac{\ps{(1+2C_{d})\frac{C_{M}}{\ve^{2}}+2C_{d}}}{N+1}|Q_{0}|<\theta|Q_{0}|
\end{align*}
if we set $N=\ceil{\theta^{-1}\ps{(1+2C_{d})\frac{C_{M}}{\ve^{2}}+2C_{d}}}$, so now it suffices to show that $f$ is bi-Lipschitz upon $E$.

Define 
\[\cM=T_{N}\cup \bigcup_{Q(S)\in T_{N}}S.\]

\begin{lemma}
Let $Q\in \cM$. Then
\begin{equation}
 \beta^{-N-1}\leq \frac{\diam f(Q)}{\diam Q} \leq \beta^{N+1} 
 \label{e:betaN}
 \end{equation}
where 
\[\beta=\max\ck{2,\eta(2),d^{\frac{1}{2}}\frac{1+2\sqrt{d}\delta}{1-2\sqrt{d}\delta }(1-\tau)^{-1}}.\]
\label{l:betaN}
\end{lemma}

\begin{proof}
First, we'll focus on the case $Q\in T_{N}$. Let $Q(j)\subsetneq Q(j-1)$ be any sequence of cubes in $T_{N}$ such that $k(Q(j))=j$ for $j=1,2,...,N$, so that $Q(1)=Q_{0}=[0,1]^{d}$. We claim that for $j=1,2,...,N$,
\begin{equation}
 \beta^{-j}\leq \frac{\diam f(MQ(j))}{\diam MQ(j)} \leq \beta^{j} .
 \label{e:betaj}
 \end{equation}
We'll prove this inductively using the following lemma:

\begin{lemma}(\cite[Proposition 10.8]{Heinonen}) If $\Omega\subseteq\bR^{d}$, $f:\Omega\rightarrow \bR^{D}$ is an $\eta$-quasisymmetric map, and $A\subseteq B$ are subsets such that $0<\diam A\leq \diam B<\infty$, then 
\begin{equation}
\frac{1}{2\eta\ps{\frac{\diam B}{\diam A}}} \leq \frac{\diam f(A)}{\diam f(B)}\leq \eta\ps{\frac{2\diam A}{\diam B}}.
\label{e:QS-comp}
\end{equation}
\label{l:QS-comp}
\end{lemma}

The lemma is stated more generally in \cite{Heinonen} for metric spaces, but this is all we'll need.

Let $1\leq j< N$ and assume we've shown $j$ satisfies \eqn{betaj} (also recall that we are assuming $\diam f(Q_{0})=\diam Q_{0}$, and so the $j=1$ case holds).
\begin{enumerate}
\item If $Q(j+1)\in \cB$ or $Q(j+1)=Q(S)$ for some $S\in \cF$, then $Q(j+1)$ is a child of $Q(j)$. Hence, $MQ(j+1)\subseteq MQ(j)$ and $\diam MQ(j)=2\diam MQ(j+1)$, so that by \eqn{QS-comp}
\begin{align*}
\frac{\diam f(MQ(j+1))}{\diam MQ(j+1)}
 & =2 \frac{\diam f(MQ(j+1))}{\diam MQ(j)}
 \geq 2  \frac{\ps{2\eta\ps{\frac{\diam MQ(j)}{\diam MQ(j+1)}}}^{-1}\diam f(MQ(j))}{\diam MQ(j)}\\
 & =\frac{2}{\eta(2)} \frac{\diam f(MQ(j))}{\diam MQ(j)} 
\geq \beta^{-1}  \frac{\diam f(MQ(j))}{\diam MQ(j)}\geq \beta^{-j-1}
\end{align*}
and since $MQ(j+1)\subseteq MQ(j)$,
\begin{align*}
\frac{\diam f(MQ(j+1))}{\diam MQ(j+1)}
& = \frac{2\diam f(MQ(j+1))}{\diam MQ(j)}
 \leq \beta \frac{\diam f(MQ(j))}{\diam MQ(j)}  \leq \beta^{j+1}.
\end{align*}
\item If $Q(j+1)\in m(S)$ for some $S$, then $Q(j)=Q(S)$, so in particular, $Q(j),Q(j+1)\in S$. By \Lemma{epsilondelta}, \Lemma{QS-comp}, and since $\frac{1}{1-\tau}>1+\tau$,
\begin{align*}
\frac{\diam f(MQ(j+1))}{\diam MQ(j+1)}
&  \leq   (1+2\sqrt{d}\delta) |A_{Q(j+1)}'|
 \leq   (1+2\sqrt{d}\delta) (1+\tau)|A_{Q(S)}'|\\
& \leq   \sqrt{d}\frac{1+2\sqrt{d}\delta }{1-2\sqrt{d}\delta }(1+\tau) \frac{\diam f(MQ(j))}{\diam MQ(j)} 
 \leq \beta\frac{\diam f(MQ(j))}{\diam MQ(j)}
  \leq \beta^{j+1}
\end{align*}
and
\begin{multline*}
  \frac{\diam f(MQ(j+1))}{\diam MQ(j+1)}
 \geq   d^{-\frac{1}{2}}(1-2\sqrt{d}\delta)  |A_{Q(j+1)}'| 
 \geq   d^{-\frac{1}{2}}(1-2\sqrt{d}\delta)  (1-\tau)|A_{Q(S)}'| \\
  \geq   d^{-\frac{1}{2}}\frac{1-2\sqrt{d}\delta }{1+2\sqrt{d}\delta}(1-\tau)\frac{\diam f(MQ(j))}{\diam MQ(j)}  
 \geq \beta^{-1}\frac{\diam f(MQ(j))}{\diam MQ(j)} \geq \beta^{-j-1}.
\end{multline*}
\end{enumerate}
This proves the induction step, and hence proves \eqn{betaj}.

Now we prove \eqn{betaN}. If $Q\in T_{N}$, this follows from \eqn{betaj}. If $Q\in S$ for some $Q(S)\in T_{N}$, let $Q(j)\in T_{N}$ be a nested chain of cubes so that $k(Q(j))=j$ for all $j<n:=k(Q(S)) $, so in particular $Q(n)=Q(S)$. Then, since $S$ is a $(\ve,\tau)$-region, \eqn{A<AS} applies, and by \Lemma{epsilondelta},
\begin{align*}
\frac{\diam f(MQ)}{\diam MQ}
& \leq (1+2\sqrt{d}\delta)|A_{Q}'|
\leq (1+2\sqrt{d}\delta)(1+\tau)|A_{Q(S)}'|\\
& \leq \sqrt{d}\frac{1+2\sqrt{d}\delta}{1-2\sqrt{d}\delta}(1+\tau)\frac{\diam f(MQ(n))}{\diam MQ(n)}
 \leq \beta^{n+1}\leq \beta^{N+1}
\end{align*}
and the lower bound follows similarly.
\end{proof}

Let $x,y\in E$ be distinct. We claim there is a chain of cubes  $Q_{j}$ such that $Q_{j}^{j}=Q_{0}$, and
 \begin{equation}
 x\in Q_{j}\in\cM \mbox{ for all }j\geq 0
 \label{e:Q(j)inTN}
 \end{equation}
 Since $x\in E$, $x\in z(S)$ for some $S\in \cF$ with $Q(S)\in T_{N}$ by \eqn{E=UzS}, and hence $x$ is contained in a chain of cubes $R_{j}\in S$ such that $R_{j}^{j}=Q(S)$. Let $n$ be such that $Q(S)^{n}=Q_{0}$ and define $Q_{j}=R_{j-n}$ for $j\geq n$ (so $Q_{j}^{j}=R_{j-n}^{j}=Q(S)^{n}=Q_{0}$) and for $j<n$ let $Q_{j}$ be the unique ancestor of $Q(S)$ with $Q_{j}^{j}=Q_{0}$. We now just need to show \eqn{Q(j)inTN}. For $j\geq n$, $Q_{j}=R_{j-n}\in S$; for $j<n$, note that since $\cB$ and the sets $S'\in \cF$ partition $\Delta(Q_{0})$, $Q_{j}$ is always in $\cB$ or in some $S'\in \cF$. If $Q_{j}\in \cB$ or $Q_{j}\in m(S')$ for some $S'\in \cF$, then $k(Q_{j})<k(Q(S))\leq N$ (note that $S'\neq S$), and so $Q_{j}\in T_{N}\subseteq \cM$; otherwise, if $Q_{j}\in S'$ for some $S'\in \cF$ and is not a minimal cube, then $k(Q_{j})\leq k(Q(S'))<k(Q(S))\leq N$, and so $Q_{j}\in \cM$. This proves \eqn{Q(j)inTN}.
 
Let $j$ is the largest integer for which $y\in 3Q_{j}$. since $y\in Q_{0}\subseteq 3Q_{0}$ and $x\neq y$, this integer is well defined. Moreover, 
\begin{equation}
|x-y|\geq \frac{\ell(Q_{j})}{2}
\label{e:xy>1/2Q}
\end{equation}
for otherwise, $|x-y|<\ell(Q_{j})/2=\ell(Q_{j+1})$ and $x\in Q_{j+1}$ imply $y\in 3Q_{j+1}$, contradicting the maximality of $j$. Then \Lemma{betaN} implies

\begin{equation}
|f(x)-f(y)|
 \leq \frac{\diam f(MQ_{j})}{\diam MQ_{j}}\diam MQ_{j}
 \stackrel{\eqn{betaN}}{\leq} \beta^{N+1}M\sqrt{d}\ell(Q_{j}) 
 \stackrel{\eqn{xy>1/2Q}}{ \leq}  \beta^{N+1}2M\sqrt{d}|x-y| 
\label{e:uplip}
\end{equation}
Furthermore, by \Lemma{QS-comp}, since $x,y\in 3Q_{j}\subseteq MQ_{j}$
\begin{align}
|f(x)-f(y)|
 & =\frac{\diam f(\{x,y\})}{\diam f(MQ_{j})} \frac{\diam f(MQ_{j})}{\diam MQ_{j}}{\diam MQ_{j}}
 \stackrel{\eqn{betaN} \atop \eqn{QS-comp}}{\geq} \frac{|x-y|}{2\eta\ps{\frac{\diam MQ_{j}}{|x-y|}}\beta^{N+1}} \notag \\
&   \stackrel{ \eqn{QS-comp} \atop \eqn{xy>1/2Q}}{\geq} \frac{|x-y|}{\eta(1)\beta^{N+1}}.
  \label{e:lowlip}
\end{align}
Thus  \eqn{uplip} and \eqn{lowlip} imply $f$ is $\beta^{N+1}\max\{2M\sqrt{d},\eta(1)\}$-bi-Lipschitz on $E$, and this finishes the proof.

\end{proof}

\subsection{The proof of \Theorem{1>2}}

We finally combine our estimates into a proof of \Theorem{1>2}, which will first require a lemma.

\begin{proof}[Proof of \Theorem{1>2}]

Let $\tau>0$ and suppose $f:\bR^{d}\rightarrow \bR^{D}$ is such that $\omega_{f}(x,r)^{2}\frac{dr}{r}dx$ is a $C$-Carleson measure. Let $B(x_{0},r_{0})\in \bR^{d}$ be any ball and let $Q_{0}=[0,1]^{d}$. Since $\omega_{f}$ is invariant under translations and dilations in the domain, the Carleson norm remains unchanged if we replace $f(x)$ with the function $f\ps{\frac{x-x_{Q_{0}}}{\frac{1}{2}\ell(Q_{0})}}$, so we may assume without loss of generality that $B(x_{0},r_{0})=B(x_{Q_{0}},\frac{1}{2})$ (that is, the largest ball contained in $Q_{0}$). By \Lemma{dyadic-version}, we know
\[\sum_{Q\subseteq Q_{0}}\omega_{f}(MQ)^{2}|Q|\leq C_{M} |Q_{0}|\]
where $C_{M}=C_{M}(C,d)$. \Theorem{1implies4} implies for all $\theta>0$ there is $E'\subseteq Q_{0}$ with $|E'|\geq (1-\theta)|Q_{0}|$ and $\ps{\frac{\diam f(Q_{0})}{\diam Q_{0}}}^{-1}f$ is $L$-bi-Lipschitz upon $E'$. By \Lemma{QS-comp}, it follows that
\[\frac{\diam f(Q_{0})}{\diam Q_{0}}\sim_{\eta,d} \frac{\diam f(B(x_{0},r_{0}))}{\diam B(x_{0},r_{0})}.\]
By picking $\theta$ small enough, we may guarantee that the set $E=E'\cap B(x_{0},r_{0})$ satisfies $|E|\geq (1-\tau)|B(x_{0},r_{0})|$. Since this holds for all $x_{0}\in \bR^{d}$ and $r_{0}>0$, \Theorem{1>2} is proven.
\end{proof}

\section{Finding bi-Lipschitz pieces of a general quasisymmetric map}
\label{s:semmes}

In this section, we focus on proving \Proposition{strong-semmes}. For the first few subsections, however, we will recall some basic facts about $A_{\infty}$ and BMO spaces and review some material from \cite{Semmes-question-of-heinonen}, as well as the technical modifications of Semmes' work we will need.

\subsection{$A_{\infty}$-weights}

For a locally integrable function $w$ on $\bR^{d}$, we will write, for any measurable subset $A$, $w(A)=\int_{A}w$, and $w_{A}=\frac{w(A)}{|A|}$.  We'll call $w$ an {\it $A_{\infty}$-weight} if it is nonnegative, locally integrable, and there is $q>0$ such that for all cubes $Q\subset\bR^{d}$ and measurable sets $E\subset Q$, 
\begin{equation}
w(E)\geq \frac{w(Q)}{1+\exp\left(q\frac{|Q|}{|E|}\right)}
\label{e:Afin}
\end{equation} 
This isn't how $A_{\infty}$ is described in most texts, but it is equivalent to the usual definition equivalent (see \cite{Hru84}).


 An important property we will use is that if $w\in A_{\infty}$, then $||\log w||_{BMO}\lec_{q}1$ (where $q$ is as in \eqn{Afin}). Recall that $\log w\in BMO(\bR^{d})$ implies there is an infimal number $||\log w||_{BMO}$ such that for all cubes $Q\subseteq \bR^{d}$,
\begin{equation}
\avint_{Q}|\log w-(\log w)_{Q}|\leq ||\log w||_{BMO}.
\label{e:w<elw}
\end{equation}
Another property is the reverse Jensen's inequality: for $w$ satisfying \eqn{Afin},
\[w_{Q}\leq C_{q} e^{(\log w)_{Q}}\]
where $C_{q}>0$ depends on $q$ and $d$.

For good references on $A_{\infty}$-weights and $BMO$ with proofs of these facts, see \cite[Chapter V]{big-stein} and \cite[Chapter VI]{BAF}.

One last technique we will need is following lemma, which is essentially known and is a good exercise with $A_{\infty}$-weight theory (a similar proof appears in \cite[Theorem 3.22]{FKP91}). 

\begin{lemma}
Let $w\in A_{\infty}(\bR^{d})$. For all $\tau\in (0,1)$, and $Q_{0}\in\Delta$, there is $E_{Q_{0}}\subseteq Q_{0}$ with
\begin{enumerate}
\item for all $Q\subseteq Q_{0}$ with $Q\cap E_{Q_{0}}\neq\emptyset$, we have  $M^{-1} \leq w_{Q}/w_{Q_{0}}\leq M$ where $M=\exp(C_{q}+2^{d}||\log w||_{BMO}/\tau)$, and
\item $|E_{Q_{0}}|\geq (1-\tau)|Q_{0}|$. 
\end{enumerate}
\label{l:A-lemma}
\end{lemma}

\begin{proof}
Since $w\in A_{\infty}$, $g:=\log w\in BMO(\bR^{d})$. Let 
\[E_{Q_{0}} = \{x\in Q_{0}: M_{\Delta}(g-g_{Q_{0}})\leq 2^{d}||g||_{BMO}\tau^{-1}\} \]
where $M_{\Delta}$ is the dyadic maximal function
\[M_{\Delta}h(x):=\sup_{x\in Q\in \Delta} \avint_{Q}|h|.\]
Since $||M_{\Delta}||_{L^{1}\rightarrow L^{1,\infty}}\leq 2^{d}$, we have 
\begin{align*}
 |Q_{0}\backslash E_{Q_{0}}| & 
 = \{x\in Q_{0}:M_{\Delta}(g-g_{Q_{0}})>2^{d}||g||_{BMO}\tau^{-1}\}\\
&  \leq  2^{d}\frac{\int_{Q_{0}}|g-g_{Q_{0}}|}{2^{d}||g||_{BMO}\tau^{-1}} 
  \leq \tau |Q_{0}|.
 \end{align*}
Let $Q\subseteq Q_{0}$ be a dyadic cube such that $Q\cap E_{Q_{0}}\neq\emptyset$. If $x\in Q\cap E_{Q_{0}}$, then 
\begin{equation}
|g-g_{Q_{0}}|_{Q}\leq M_{\Delta}(g-g_{Q_{0}})(x)\leq 2^{d}||g||_{BMO}\tau^{-1}.
\label{e:g-g0}
\end{equation}
Moreover, since $w\in A_{\infty}$, we have by \eqn{w<elw} that 
\begin{equation}
\log w_{Q} \leq C_{q}+ (\log w)_{Q}.
\label{e:w<C+w}
\end{equation}
Using this and Jensen's inequality, we get
\begin{align*}
\log w_{Q}- \log w_{Q_{0}}
& \stackrel{\eqn{w<C+w}}{\leq} C_{q}+(\log w)_{Q}-(\log w)_{Q_{0}}
= C_{q}+(g-g_{Q_{0}})_{Q}
 \stackrel{\eqn{g-g0}}{\leq} C_{q}+2^{d}||g||_{BMO}\tau^{-1}
\end{align*}
and
\begin{align*}
\log w_{Q_{0}}-\log w_{Q}
& \stackrel{\eqn{w<C+w}}{\leq} C_{q}+(\log w)_{Q_{0}}-(\log w)_{Q}
= C_{q}-(g-g_{Q_{0}})_{Q}
 \stackrel{\eqn{g-g0}}{\leq} C_{q}+2^{d}||g||_{BMO}\tau^{-1}.
\end{align*}
Thus, 
\[\av{\log \frac{w_{Q}}{w_{Q_{0}}}}\leq C_{q}+2^{d}\tau^{-1}||g||_{BMO}.\]

\end{proof}

\subsection{Metric doubling measures and strong $A_{\infty}$-weights}

We recall the following definition from \cite{Semmes-strong-A°}.

\begin{definition}
We say a Borel measure $\mu$ on $\bR^{d}$ is {\it $C_{\mu}$-doubling} on its support if 
\[\mu (B(x,2r))\leq C_{\mu}\mu(B(x,r))\]
for all $x\in\supp \mu$ and $r>0$.  For  $E\subseteq \bR^{d}$ closed, we say that a doubling measure $\mu$ is a {\it metric doubling measure} on $E$ if $\supp\mu =E$ and 
\begin{equation}
\mu(B(x,|x-y|)\cup B(y,|x-y|))^{\frac{1}{d}}\sim \dist(x,y)
\label{e:metric-doubling}
\end{equation}
for some metric $\dist(x,y)$ on $E$.
\end{definition}

In \cite{Gehring73}, Gehring showed that the pullback of Lebesgue measure under a quasisymmetric map $f:\bR^{d}\rightarrow \bR^{d}$ is an $A_{\infty}$-weight; Semmes observed in \cite{Semmes-strong-A°} that this holds more generally for all metric doubling measures on $\bR^{d}$, with a proof essentially the same as Gehring's. 

\begin{lemma}(\cite[Proposition 3.4]{Semmes-strong-A°}
If $\nu$ is a metric doubling measure on $\bR^{d}$, $d\geq 2$, then $\nu$ is absolutely continuous and it is given by $w(x)dx$, where $w\in A_{\infty}$, and $q$ in \eqn{Afin} depends upon $d,C_{\nu}$ and the constants in \eqn{metric-doubling}. We call the weight $w$ a {\it strong $A_{\infty}$-weight}.
\label{l:strongA°}
\end{lemma}

If $f:\bR^{d}\rightarrow \bR^{d}$, then the pullback of Lebesgue measure under $f$ is an example of a metric doubling measure, where in this case $\dist(x,y)=|f(x)-f(y)|$, and \Lemma{strongA°} recovers Gehring's original result. 

Metric doubling measures and strong $A_{\infty}$-weights arose in studying the so-called ``quasiconformal Jacobian problem" (for discussions of this problem, see \cite{Semmes-strong-A°},\cite{Semmes-nonexistence}, and \cite{Bishop-A1}). While the aforementioned papers gradually demonstrated the intractability of this problem, its pursuit has developed many useful techniques (and counterexamples) in the theory of quasisymmetric mappings.

\subsection{Serious and strong sets}

Here we recall some definitions and results from \cite{Semmes-question-of-heinonen} about serious and strong sets.

\begin{definition}
Let $E_{0}\subseteq E\subseteq \bR^{d}$. We say $E_{0}$ is a {\it serious subset} of $E$ if there is $C>0$ so that if $x\in E_{0}$ and $0<t<\diam E_{0}$, then there is $y\in E$ such that 
\begin{equation}
\frac{t}{C}\leq |x-y|\leq t.
\end{equation}
We will call $C$ the {\it seriousness constant} of the pair $(E_{0},E)$. If $E_{0}=E$, we say $E$ is a {\it serious set}.
\end{definition}

In \cite[Lemma 1.8]{Semmes-question-of-heinonen}, Semmes shows that all compact subsets with positive measure contain a serious subset whose measure is as close to the measure of the original set as you wish, although there is no control given on the seriousness constant of this set. Without too much effort, though, this dependence can be determined, and allows us to make \Lemma{semmes-lemmas} depend quantitatively only on $d$, $\eta$, and the  density of $E$ inside a prescribed dyadic cube.

\begin{lemma}
Let $E\subseteq\bR^{d}$ be a compact set of positive measure contained in a dyadic cube $Q_{0}$. Then for each $\delta\in (0,1)$, there is $E_{0}\subseteq E$ compact such that
\begin{enumerate}
\item $|E_{0}|\geq (1-\delta)|E|$, and
\item $E_{0}$ is a $\ps{\delta|E|/|Q_{0}|}^{\frac{1}{d}}/(8d^{\frac{1}{2}}3^{d})$-serious subset of $E$.
\end{enumerate}
\label{l:very-serious}
\end{lemma}

\begin{proof}
Let $Q_{j}$ be the collection of maximal cubes contained in $Q_{0}$ for which 
\[\frac{|E\cap Q_{j}|}{|Q_{j}|}<\delta\frac{|E|}{|Q_{0}|}\] 
and set 
\[E_{0}=E\backslash \bigcup Q_{j}^{\circ}.\]
Observe that this is a countable intersection of bounded closed sets and hence is compact. Since the $Q_{j}$ have disjoint interiors, we have
\[
|E\backslash E_{0}|=\sum_{j}|E\cap Q_{j}^{\circ}|
< \delta\frac{|E|}{|Q_{0}|}\sum |Q_{j}|\leq 
\delta\frac{|E|}{|Q_{0}|}|Q_{0}|=\delta|E|,\]
which implies the first item of the lemma. Next, set
\begin{equation}
N=\floor{ \frac{\log \ps{3^{d}\frac{|Q_{0}| }{\delta |E|}}}{d}}+1.
\label{e:seriousN}
\end{equation}
We claim that for any dyadic cube $Q$ intersecting $E_{0}$ such that $Q^{N+1}\subseteq Q_{0}$, we have
\begin{equation}
(Q^{N+1}\backslash 3Q)\cap E \neq \emptyset.
\label{e:Q^N+1}
\end{equation}
If not, then since $Q^{N+1}$ is not contained in any $Q_{j}$, 
\[\frac{|E|}{|Q_{0}|}\delta  \leq \frac{|E\cap Q^{N+1}|}{|Q^{N+1}|}\leq \frac{|E\cap 3Q|}{2^{d(N+1)}|Q|}\leq 3^{d} 2^{-d(N+1)}
\stackrel{\eqn{seriousN}}{\leq} \frac{|E|}{|Q_{0}|}\delta  2^{-d},\]
which is a contradiction, hence proving \eqn{Q^N+1} and the claim. 

Now let $x\in E_{0}$, $t\leq \diam E$. Let $Q\ni x$ be contained in $Q_{0}$ such that 
\[\diam Q^{N+1}\leq t<\diam Q^{N+2}.\]
Since $t\leq \diam E\leq \diam Q_{0}$, $Q^{N+1}\subseteq Q_{0}$, and by \eqn{Q^N+1} there is $y\in (Q^{N+1}\backslash 3Q)\cap E$, so that
\[|x-y|\leq \diam Q^{N+1}\leq t\]
and 
\[|x-y|\geq \ell(Q)=2^{-N-2}d^{-\frac{1}{2}}(\diam Q^{N+2})\geq 2^{-N-2}d^{-\frac{1}{2}}t \geq \frac{1}{8d^{\frac{1}{2}}3^{d}}\ps{\frac{\delta|E|}{|Q_{0}|}}^{\frac{1}{d}} t\]
and this finishes the second part of the lemma.
\end{proof}

\begin{definition}
A closed set $\tilde{E}\subseteq \bR^{d}$ is a {\it strong set} if there is a constant $C>0$ so that for each $x\in \bR^{d}\backslash \tilde{E}$, there is $y\in \tilde{E}$ so that 
\begin{equation}
|x-y|\leq C\dist(x,\tilde{E})
\label{e:y-to-x}
\end{equation}
and
\begin{equation}
\dist(y,\bR^{d}\backslash \tilde{E})\geq C^{-1}\dist(x,\tilde{E}).
\label{e:y-from-E^c}
\end{equation}
\end{definition}

Thus, to each point $x\in \tilde{E}^{c}$, we can assign a ball in $\tilde{E}$ with radius and distance to $x$ comparable to the distance from $x$ to $\tilde{E}$; in Semmes' words, this says $\tilde{E}$ is at least as big as its complement.

\begin{lemma}(\cite[Proposition 1.16]{Semmes-question-of-heinonen})
If $\tilde{E}\subseteq \bR^{d}$ is a $C$-strong set, then for all $x\in \tilde{E}$ and $r>0$,
\begin{equation}
|\tilde{E}\cap B(x,r)|\sim_{d,C}r^{d}.
\label{e:AD}
\end{equation}
\label{l:seriousAD}
\end{lemma}

\begin{lemma}(\cite[Proposition 1.15]{Semmes-question-of-heinonen}
If $\tilde{E}\subseteq \bR^{d}$ is a $C$-strong set and $g:\tilde{E}\rightarrow \bR^{d}$ is $\eta$-quasisymmetric, then $g(\tilde{E})$ is $C'$-serious with $C'$ depending on $\eta,C,$ and $d$.
\label{l:alsoserious}
\end{lemma}

The next lemma is an amalgamation of Propositions 1.10, 1.14, 1.22, and 1.23 from \cite{Semmes-question-of-heinonen}. 

\begin{lemma}
Suppose $E$ is a compact subset of $\bR^{d}$, $E_{0}\subseteq E$ is a $C$-serious subset of $E$, and $f:E\rightarrow \bR^{d}$ is $\eta$-quasisymmetric. Then the following hold:
\begin{enumerate}
\item There is $\hat{E}\supseteq E_{0}$ that is $\hat{C}$-serious, with $\hat{C}>0$ depending only on $C$ and $d$.
\item There is $g:\hat{E}\rightarrow \bR^{d}$ that agrees with $f$ on $E_{0}$ and is $\hat{\eta}$-quasisymmetric, with $\hat{\eta}$ depending on $C,d,$ and $\eta$.
\item There is a $\tilde{C}$-strong set $\tilde{E}\supseteq \hat{E}$, where $\tilde{C}$ depends only on $\tilde{C}$ and $d$.
\item The map $g$ admits an $\tilde{\eta}$-quasisymmetric extension $G:\tilde{E}\rightarrow \bR^{d}$. Here, $\tilde{\eta}$ depends only on $\hat{\eta}$, $\tilde{C}$, and $d$.
\item The measure $\mu$ defined by $\mu(A)=|G(A)|$ is a metric doubling measure on $\tilde{E}$ , with data depending only on $\tilde{C}$, $\tilde{\eta}$, and $d$.
\item There is a metric doubling measure $\nu$ on $\bR^{d}$ such that $\nu(A)=|G(A)|$ for all $A\subseteq \tilde{E}$. The doubling constant $C_{\nu}$ and metric doubling constants of $\nu$ depend only on those for $\nu$, $d$, and $\tilde{C}$.
\end{enumerate}
\label{l:semmes-lemmas}
\end{lemma}

\begin{corollary}
If $E\subseteq Q_{0}\subseteq \bR^{d}$ has positive measure, and $f:E\rightarrow \bR^{d}$ is an $\eta$-quasisymmetric map, then \Lemma{semmes-lemmas} still holds and all the implied constants depend on $d$, $\eta$, and $\frac{|E|}{|Q_{0}|}$.
\end{corollary}

\begin{proof}
This follows from Lemmas \ref{l:very-serious} and \ref{l:semmes-lemmas}.
\end{proof}

\begin{lemma}
With $\nu,\tilde{E}$, $f$, $\eta$, and $G$ as in \Lemma{semmes-lemmas}, we have that for all $x\in \tilde{E}$ and $r>0$,
\begin{equation}
\nu(B(x,r))\sim_{d,\tilde{\eta},\tilde{C},C_{\nu}} (\diam G(\tilde{E}\cap B(x,r)))^{d}.
\label{e:nusimG}
\end{equation}
where $C_{\nu}$ is the doubling constant of $\nu$. 
\end{lemma}

\begin{proof}
Let $x\in \tilde{E}$, $r>0$, and $Q$ be a cube containing $B(x,r)$ of side length $2r$. 

First, note that since $\nu$ is an $A_{\infty}$-weight and $\tilde{E}$ is strong, \eqn{Afin} and \eqn{AD} imply 
\[\nu(\tilde{E}\cap B(x,r))\sim \nu(Q)\sim \nu(B(x,r))\]
with implied constants depending on $d$, the $A_{\infty}$-data of $\nu$, and the constants in \eqn{AD}. By quasisymmetry, it is not hard to show that there is $\rho<1$ depending only on $\tilde{\eta}$ so that 
\[ G(\tilde{E})\cap B(G(x), \rho \diam G(B(x,r)\cap \tilde{E}))\subseteq G(\tilde{E}\cap B(x,r))\subseteq  B(G(x),\diam  G(\tilde{E}\cap B(x,r)))\]
and since $G(\tilde{E})$ is also serious by \Lemma{alsoserious}, \Lemma{seriousAD} and the above containments imply 
\[\nu(B(x,r))\sim \nu(B(x,r)\cap \tilde{E})=|G(\tilde{E}\cap B(x,r))|\sim \diam (G(\tilde{E}\cap B(x,r)))^{d}.\]
\end{proof}

\subsection{A slightly stonger Semmes theorem}

We are now in a position to prove \Proposition{strong-semmes}, which strengthens  Semmes' original result, \Theorem{semmes}. While Semmes shows that if $E\subseteq \bR^{d}$ and $f:E\rightarrow \bR^{d}$ is quasisymmetric, then $|E|=0$ if and only if $|f(E)|=0$, we show here that $f$ is in fact bi-Lipschitz on a large subset of $E$ quantitatively. We restate this below.

\begin{semmes}
Let $E\subseteq Q_{0}\subseteq \bR^{d}$, $\rho\in(0,\frac{1}{2})$, $d\geq 2$, and set $\delta=\frac{|E|}{|Q_{0}|}>0$. Let $f:E\rightarrow \bR^{d}$ be $\eta$-quasisymmetric. Then there is $E''\subseteq E$ compact with $|E''|\geq (1-\rho)|E|$ and $\ps{\frac{\diam f(E'')}{\diam E''}}^{-1}f|_{E''}$ is $L$-bi-Lipschitz for some $L$ depending on $\eta$, $d$, $\rho$, and $\delta$. 
\end{semmes}

\begin{proof}[Proof of \Proposition{strong-semmes}]

By \Lemma{very-serious}, there $\hat{E}\subseteq E$ that is $\hat{C}$-serious and 
\[ |\hat{E}|\geq \ps{1-\frac{\rho}{2}}|E|\]
with $\hat{C}$ depending on $d,\delta$, and $\rho$. According to \Lemma{semmes-lemmas}, $\hat{E}\subseteq \tilde{E}$ for some $\tilde{C}$-serious set $\tilde{E}$, to which $f$ has an $\tilde{\eta}$-quasisymmetric extension $G:\tilde{E}\rightarrow \bR^{d}$ and a metric doubling measure $\nu$ on $\bR^{d}$ with $\nu(A)=|G(A)|$ for all $A\subseteq \tilde{E}$. We can write $d\nu=wdx$ where $w$ is an $A_{\infty}$-density by \Lemma{strongA°}. Applying \Lemma{A-lemma} with $\tau=\frac{\rho}{2}$, there is  $M>1$ depending on $d,\rho$, and the $A_{\infty}$-data of $w$ and  $E'\subseteq Q_{0}$ with $|E'|\geq (1-\frac{\rho}{2})|Q_{0}|$ such that 
\begin{equation}
\frac{1}{M}\leq \frac{w_{Q}}{w_{Q_{0}}}\leq M
\label{e:w/wQ0}
\end{equation}
for all $Q\subseteq Q_{0}$ such that $Q\cap E\neq\emptyset$. Let $E''=E'\cap \hat{E}$, so that  $|E''|\geq (1-\rho)|E|$. We will now show $\ps{\frac{\diam f(E'')}{\diam E''}}^{-1}f$ is bi-Lipschitz upon $E''$. 

Let $x,y\in E''$ be distinct points and $Q\subseteq Q_{0}$ be a minimal dyadic cube containing $x$ so that $y\in 3Q$. Since $\tilde{E}$ is serious, and $G$ is $\tilde{\eta}$-quasisymmetric and $\{x,y\}$ and $B(x,|x-y|)\cap \tilde{E}$ have comparable diameters,
\begin{align}
 |f(x)-f(y)|
&   =|G(x)-G(y)| 
  \stackrel{\eqn{QS-comp}}{\sim}_{\tilde{\eta}} \diam G(B(x,|x-y|)\cap \tilde{E})    \stackrel{\eqn{nusimG}}{\sim}_{\tilde{\eta},d,\tilde{C}}   \nu(B(x,|x-y|))^{\frac{1}{d}} .
  \label{e:fsimvB}
 \end{align}
 Since $3Q$ is minimal, we know 
 \begin{equation}
 \frac{1}{2}\ell(Q) \leq |x-y|\leq \diam 3Q.
 \label{e:1/2Q<x-y<3Q}
 \end{equation}
 Hence, since $\nu$ is doubling, $\nu(Q)\sim_{d,C_{\nu}}\nu(B(x,|x-y|))$. Thus, continuing our chain of estimates, (and using the fact that $\ell(Q)^{n}=|Q|$) we have 
 \begin{align}
\eqn{fsimvB}   & \sim_{d,C_{\nu}} \nu(Q)^{\frac{1}{d}} =w(Q)^{\frac{1}{d}}
   = \ell(Q) \ps{w_{Q}}^{\frac{1}{d}} 
       \stackrel{\eqn{w/wQ0}}{\sim}_{M,d} \ell(Q) \ps{w_{Q_{0}}}^{\frac{1}{d}} \notag \\
&     \stackrel{\eqn{1/2Q<x-y<3Q}}{\sim}_{d}|x-y|\ps{w_{Q_{0}}}^{\frac{1}{d}} 
      = |x-y| \frac{\nu(Q_{0})^{\frac{1}{d}}}{\ell(Q_{0})} \notag \\
 & \sim_{C_{\mu},d} |x-y| \frac{\nu(B(x_{Q_{0}},\diam Q_{0}))^{\frac{1}{d}}}{\ell(Q_{0})} \notag \\
 & \sim_{\tilde{C},\tilde{\eta},d} |x-y| \frac{\diam G(B(x_{Q_{0}},\diam Q_{0})\cap \tilde{E})}{\ell(Q_{0})}. 
 \label{e:sim|x-y|G/Q}
 \end{align}
Since $E''\subseteq  Q_{0}\cap \tilde{E} \subseteq B(x_{Q_{0}},\diam Q_{0})$ and $|E''|\geq (1-\rho)|E|\geq \frac{\delta}{2}|Q_{0}|$, we know $\diam E''\sim_{d,\delta}\diam Q_{0}$, and so \Lemma{QS-comp} implies
\[ \eqn{sim|x-y|G/Q} \stackrel{\eqn{QS-comp}}{\sim}_{\tilde{\eta}} |x-y| \frac{\diam G(E'')}{\ell(Q_{0})}  \sim_{d,\delta,\tilde{\eta}} |x-y| \frac{\diam f(E'')}{\diam E''} 
\]
Combining this with \eqn{fsimvB} and \eqn{sim|x-y|G/Q}, we see $|f(x)-f(y)|\sim |x-y| \frac{\diam f(E'')}{\diam E''} $ with implied constants depending on $\tilde{\eta},\tilde{C},d,M$, and $C_{\nu}$. Finally, we recall that these constants depend only on $d,\eta, \rho$, and $\delta$. This finishes the proof of the proposition.
 
\end{proof}

In the last part of this section, we adapt \Proposition{strong-semmes} to the case that $f$ maps a set to a large bi-Lipschitz image of $\bR^{d}$, which is the case we will need later on.

\begin{lemma}
Suppose $B_{0}\subseteq \bR^{d}$, $f:B_{0}\rightarrow \bR^{D}$ is $\eta$-quasisymmetric, and there is $E'\subseteq B_{0}$ such that $\cH^{d}(f(E'))\geq c(\diam f(B_{0}))^{d}$, and there is $g:f(E')\rightarrow \bR^{d}$  that is $L$-bi-Lipschitz. Then there is $E_{0}\subseteq E'$ and $M=M(\eta,d,L,c)\geq 1$ such that $|E_{0}|\gec_{d,L,\eta,c}|B_{0}|$ and $\ps{\frac{\diam f(E_{0})}{\diam E_{0}}}^{-1}f|_{E_{0}}$ is $M$-bi-Lipschitz.
\label{l:one-piece}
\end{lemma}

\begin{proof}
Let $E_{1}=g\circ f(E')$. If $B$ is a ball centered on $E_{1}$ with radius $\diam E_{1}$, then
\begin{align}
 c(\diam f(B_{0}))^{d}
& \leq \cH^{d}(f(E'))
\leq L^{d}|E_{1}|
\leq L^{d}|B|   =L^{d}w_{d}(\diam E_{1})^{d}
\leq L^{2d}w_{d}(\diam f(E'))^{d}
\label{e:f(B)^d>f(E')^d}
\end{align}
so that
\begin{equation}
\diam f(E')\geq \frac{c^{\frac{1}{d}}}{L^{2}w_{d}^{\frac{1}{d}}}\diam f(B_{0}).
\label{e:f(E')>f(B0)}
\end{equation}
Set 
\[h:=f^{-1}\circ g^{-1}:E_{1}\rightarrow \bR^{d}.\]
Since $f$ is $\eta$-quasisymmetric, $f^{-1}:f(\bR^{d})\rightarrow \bR^{d}$ is $\eta'$-quasisymmetric with 
\[\eta'(t)=\eta^{-1}(t^{-1})^{-1}\] 
(see \cite[Proposition 10.6]{Heinonen}), and it is not hard to show using the definition of quasisymmetry that $h$ is $\eta'(L^{2}\cdot)$-quasisymmetric.

Let $Q_{1}$ be a cube containing $E_{1}$ with $\ell(Q_{1})=\diam E_{1}$, so that 
\begin{align}
|E_{1}|
& \geq L^{-d}\cH^{d}(f(E')) 
\geq  \frac{c}{L^{d}}(\diam f(B_{0}))^{d}
\geq  \frac{c}{L^{d}}(\diam f(E'))^{d}  \geq \frac{c}{L^{2d}}(\diam E_{1})^{d} 
\label{e:E1>Q1}
\end{align}
By \Proposition{strong-semmes}, there is $E_{1}'\subseteq E_{1}$ with $|E_{1}'|\geq \frac{1}{2}|E_{1}|$ upon which $\ps{\frac{\diam h(E_{1}')}{\diam E_{1}'}}^{-1}h$ is $L'$-bi-Lipschitz, with $L'$ depending on $L,c,d$, and the function $L^{2}\eta'$. Let $E_{0}=h(E_{1}')\subseteq Q_{0}$. Using the facts that $\ps{\frac{\diam h(E_{1}')}{\diam E_{1}'}}^{-1}h$ is $L'$-bi-Lipschitz, $g^{-1}(E_{1}')=f(E_{0})$, and $g$ is $L$-bi-Lipschitz, it isn't hard to show that $\ps{\frac{\diam f(E_{0})}{\diam E_{0}}}^{-1}f$ is $L'L^{2}$-bi-Lipschitz upon $E_{0}$.

If $B'$ is a ball centered upon $E_{1}'$ with radius $\diam E_{1}'$, then
\[\omega_{d}(\diam E_{1}')^{d}
=|B'|
\geq |E_{1}'|
\geq \frac{1}{2}|E_{1}|
\stackrel{\eqn{E1>Q1}}{\geq} \frac{c}{2L^{2d}}(\diam E_{1})^{d}\]
and so
\begin{equation}
\diam E_{1}'\geq \frac{c^{\frac{1}{d}}}{2^{\frac{1}{d}}L^{2}w_{d}^{\frac{1}{d}}}\diam E_{1}.
\label{e:E1'>E1}
\end{equation}
Then
\begin{align}
\diam f(E_{0})
& \geq L^{-1}\diam g\circ f(E_{0})
=L^{-1}\diam E_{1}'
\stackrel{\eqn{E1'>E1}}{\geq} \frac{c^{\frac{1}{d}}}{2^{\frac{1}{d}}L^{3}w_{d}^{\frac{1}{d}}}\diam E_{1} \notag \\
& \geq  \frac{c^{\frac{1}{d}}}{2^{\frac{1}{d}}L^{4}w_{d}^{\frac{1}{d}}}\diam f(E')
 \stackrel{\eqn{f(E')>f(B0)}}{\geq}   \frac{c^{\frac{2}{d}}}{2^{\frac{1}{d}}L^{6}w_{d}^{\frac{2}{d}}}  \diam f(B_{0})
 \label{e:fe0>fb0}
\end{align}
where in the first and penultimate inequalities we used the fact that $g$ is $L$-bi-Lipschitz. By \Lemma{QS-comp},
\begin{equation}
\frac{\diam E_{0}}{\diam B_{0}}
 \geq \ps{2\eta'\ps{\frac{\diam f(B_{0})}{\diam f(E_{0})}}}^{-1}  \stackrel{\eqn{fe0>fb0}}{\geq} \ps{2\eta'(2^{-\frac{1}{d}}L^{-6}w_{d}^{-\frac{2}{d}}c^{\frac{2}{d}})}^{-1}.
\label{e:E0/Q0}
\end{equation}
Furthermore,
\begin{align*}
|E_{0}|
& =|h(E_{1}')|
\geq \ps{\frac{\diam h(E_{1}')}{\diam E_{1}'}}^{d}(L')^{-d}|E_{1}'| \geq  \frac{1}{2}\ps{\frac{\diam E_{0}}{\diam E_{1}}}^{d}(L')^{-d}|E_{1}|\\
& \stackrel{\eqn{E1>Q1}}{\geq}\frac{c}{2(L')^{d}L^{2d}}(\diam E_{0})^{d}
 \stackrel{\eqn{E0/Q0}}{\geq}
\frac{c}{2}\ps{L^{2}L' 2\eta'(2^{-\frac{1}{d}}L^{-6}w_{d}^{-\frac{2}{d}}c^{\frac{2}{d}})}^{-d} (\diam B_{0})^{d}\\
& \geq \frac{c}{2w_{d}}\ps{L^{2}L' 2\eta'(2^{-\frac{1}{d}}L^{-6}w_{d}^{-\frac{2}{d}}c^{\frac{2}{d}})}^{-d}|B_{0}|.
\end{align*}

\end{proof}

\section{Bi-Lipschitz parts imply big-pieces of bi-Lipschitz images}
\label{s:3>4}

The following theorem proves  (3) implies (4) in \Theorem{main}

\begin{theorem}
Suppose $f:\bR^{d}\rightarrow \bR^{D}$ is $\eta$-quasisymmetric and there are $c,L>0$ such that for all $x\in \bR^{d}$ and $r>0$, there is $E\subseteq B(x,r)$ such that $|E|\geq c|B(x,r)|$ and $\ps{ \frac{\diam f(B(x,r))}{\diam B(x,r)}}^{-1} f|_{E}$ is $L$-bi-Lipschitz. Then $f(\bR^{d})$ has big pieces of bi-Lipschitz images.
\label{t:3>4}
\end{theorem}

We first need the following lemma.

\begin{lemma}
Let : $f:\bR^{d}\rightarrow \bR^{D}$ be $\eta$-quasisymmetric. The following are equivalent:
\begin{enumerate}
\item The set $f(\bR^{d})$ has BPBI($\kappa,L$), that is, there is $\kappa>0$ such that for all $\xi\in f(\bR^{d})$ and $s>0$, there is $A\subseteq B(\xi,s)\cap f(\bR^{d})$ so that $\cH^{d}(A)\geq \kappa s^{d}$ and an $L$-bi-Lipschitz map $g:A\rightarrow \bR^{d}$. 
\item There is $c>0$ such that for all $x\in \bR^{d}$ and $r>0$, there is $E'\subseteq B(x,r)$ and an $L$-bi-Lipschitz map $g:f(E')\rightarrow \bR^{d}$ such that $\cH^{d}(f(E'))\geq c (\diam f(B(x,r)))^{d}$.
\end{enumerate}
\label{l:ran-dom}
\end{lemma}

\begin{proof}
Let $f:\bR^{d}\rightarrow \bR^{D}$ be $\eta$-quasisymmetric.
\begin{description}
\item[$(1)\Rightarrow (2)$]  Let $x\in \bR^{d}$, $r>0$, and set
\[s=\sup\{t: B(f(x),t)\cap f(\bR^{d})\subseteq f(B(x,r))\}.\]
Then by \Lemma{QS-comp} and the fact that $f^{-1}(B(f(x),s))\subseteq B(x,r)$,
\begin{equation}
\frac{2s}{\diam f(B(x,r))} =
\frac{\diam f(f^{-1}(B(f(x),s)))}{\diam f(B(x,r))}
 \geq \frac{1}{2\eta\ps{\frac{\diam f^{-1}(B(x,s))}{\diam B(x,r)}}}\\
 \geq \frac{1}{2\eta(1)}.
 \label{e:2s<1/2eta}
\end{equation}
By assumption, we know there is $E\subseteq B(f(x),s)\cap f(\bR^{d})$ and $g:E\rightarrow \bR^{d}$ $L$-bi-Lipschitz such that 
\[\cH^{d}(E)\geq \kappa s^{d} \stackrel{\eqn{2s<1/2eta}}{\geq} \frac{\kappa}{4^{d}\eta(1)^{d}}(\diam f(B(x,r)))^{d}.\]
Letting $E'=f^{-1}(E)$ and $c=\frac{\kappa}{4^{d}\eta(1)^{d}}$ proves (2).
\item[$(2)\Rightarrow (1)$]  Let $\xi\in f(\bR^{d})$ and $s>0$, $x=f^{-1}(\xi)$, and set
\[r=\sup\{t:f(B(x,t))\subseteq B(\xi,s)\}.\]
Since $r$ is supremal, there is $y\in B(x,r)$ such that $|f(y)-f(x)|=s$. Also, by assumption, there is $E'\subseteq B(x,r)$ so that if $E=f(E')\subseteq B(\xi,s)\cap f(\bR^{d})$, we have
\[\cH^{d}(E)
\geq c(\diam f(B(x,r)))^{d}\geq c|f(x)-f(y)|^{d}=cs^{d},\]
and so (1) holds with $E=f(E')$ and $\kappa=c$. 
\end{description}

\end{proof}

\begin{proof}[Proof of \Theorem{3>4}]

By \Lemma{ran-dom}, it suffices to show that there is $c>0$ such that for all $x\in \bR^{d}$ and $r>0$, there is $E'\subseteq B(x,r)$ and an $L$-bi-Lipschitz map $g:f(E')\rightarrow \bR^{d}$ such that $\cH^{d}(f(E'))\geq c (\diam f(B(x,r)))^{d}$. Let $B(x,r)\subseteq \bR^{d}$. By assumption, there is $E'\subseteq B(x,r)$ such that $|E'|\gec|B(x,r)|$ and $\ps{\frac{\diam f(B(x,r))}{\diam B(x,r)}}^{-1}f$ is $L$-bi-Lipschitz on $E'$ for some $L$. Then
\begin{align*}
\cH^{d}(f(E'))
& \sim_{L} \ps{\frac{\diam f(B(x,r))}{\diam B(x,r)}}^{d}|E'|\gec \ps{\frac{\diam f(B(x,r))}{\diam B(x,r)}}^{d}|B(x,r)| \\
&  \gec_{d} (\diam f(B(x,r)))^{d}.
\end{align*}
\end{proof}

\section{Big Pieces implies a Carleson estimate}
\label{s:4>1}
\subsection{Preliminaries}

In this section, we focus on proving (4) implies (1) in \Theorem{main} by showing the following.

\begin{theorem}
Suppose $f:\bR^{d}\rightarrow \bR^{D}$ is $\eta$-quasisymmetric,  $d\geq 2$, and $f(\bR^{d})$ has BPBI($\kappa,L$). Then $\omega_{f}(x,r)^{2}\frac{dr}{r}dx$ is a Carleson measure, with Carleson constant depending on $D,\eta$, and the constants in the big pieces condition.
\label{t:4>1}
\end{theorem}

%
%
%
%
%
%
%
%

\subsection{A reduction using John-Nirenberg and the $\frac{1}{3}$-trick}

In this section, we show how to reduce the proof of \Theorem{4>1} to the following lemma, which we will prove in the following section.

\begin{lemma}
Let $d\geq 2$, $f:\bR^{d}\rightarrow \bR^{D}$ be $\eta$-quasisymmetric, and suppose $f(\bR^{d})$ has BPBI($\kappa,L$). Then for any $v\in \bR^{d}$ and every $Q_{0}\in\Delta$, there is $E\subseteq Q_{0}$ such that $|E|\gec_{\eta,d,\kappa} |Q_{0}|$ and 
\begin{equation}
\sum_{R\subseteq Q_{0} \atop R\cap E\neq\emptyset }\omega_{f_{v}}(R)^{2}|R|\lec_{d,\eta,\kappa}|Q_{0}|.
\label{e:sumomega(R)}
\end{equation}
where $f_{v}$ is the function $f_{v}(x)=f(x+v)$.
\label{l:w-JN}
\end{lemma}

\begin{proof}[Proof of \Theorem{4>1}]

Suppose $f:\bR^{d}\rightarrow \bR^{D}$ is $\eta$-quasisymmetric and the image of $f$ has big-pieces of bi-Lipschitz images of $\bR^{d}$.

First, we recall a version of the John-Nirenberg theorem.

\begin{lemma}(\cite[Section IV.1]{of-and-on}) Let $a:\Delta\rightarrow [0,\infty)$ be given, and suppose there are $N,\delta>0$ such that 
\begin{equation}
\av{\ck{x\in R:\sum_{Q\ni x \atop Q\subseteq R}a(Q)\leq N}}\geq \delta|R| \mbox{ for all }R\in \Delta.
\end{equation}
Then 
\begin{equation}
\sum_{Q\subseteq R} a(Q)|Q|\lec_{d,N,\delta}|R| \mbox{ for all }R\in \Delta.
\end{equation}
\label{l:JN}
\end{lemma}

If we assume \Lemma{w-JN}, then each cube $Q$ contains a set $E$ for which
\[|E|\gec_{d,\eta,\kappa} 
|Q|
\gec_{d,\eta,\kappa}  \sum_{R\cap E\neq\emptyset \atop R\subseteq Q}\omega_{f}(R)^{2}|R|
\gec \sum_{R\subseteq Q}\omega_{f}(R)^{2}|R\cap E|
= \int_{E}\sum_{R\subseteq Q}\omega_{f}(R)^{2}\one_{R}.\]
Hence, if $E'=\{x\in E: \sum_{x\in R\subseteq Q}\omega_{f}(R)^{2} \leq 2C\}$ where $C$ is the product of the implied constants in the above inequalities, we get that $|E'|\geq \frac{1}{2}|E|\gec |Q|$, and so \Lemma{JN} implies
\begin{equation}
\sum_{R\subseteq Q}\omega_{f}(R)^{2}\lec_{d,\eta,c,L}|Q| \mbox{  for all } Q\in\Delta.
\label{e:one-grid}
\end{equation}
\Theorem{4>1} doesn't follow just yet. We'd like to employ \Lemma{dyadic-version}, but this only works if we know
\[\sum_{R\subseteq Q}\omega_{f}(MR)^{2}\lec_{d,\eta,c,L}|Q|\]
for some $M>1$. However, \eqn{sumomega(R)} implies
\begin{equation}
\sum_{R\subseteq Q \atop R\in \Delta+v}\omega_{f}(R)^{2}\lec_{d,\eta,c,L}|Q| \mbox{ for all }Q\in \Delta+v, \;\; v\in \bR^{d}.
\label{e:any-grid}
\end{equation}
where $\Delta+v=\{Q+v:Q\in \Delta\}$ is the set of dyadic cubes translated by the vector $v$. 

We now invoke the so-called $\frac{1}{3}$-trick, which says that, for any cube $R$ with $\ell(R)=\frac{2^{-k}}{3}$, $k\in\{0,1,2,...\}$, there is $Q\in \Delta+v$ for some $v\in \{0,\frac{1}{3}\}^{d}$ such that $\ell(Q)=2^{-k}$ and $R\subseteq Q$. For a proof, see \cite[p. 339-40]{O-TSP}. Thus, if $R\in \Delta$ and $\ell(R)=2^{-k-2}$ for some $k\geq 0$, then $\ell(\frac{4}{3}R)= \frac{2^{-k}}{3}$, so there is 
\[Q_{R}\in\tilde{\Delta}:=\bigcup_{v\in \{0,\frac{1}{3}\}^{d}} (\Delta+v)\] 
with $\ell(Q_{R})=2^{-k}$ containing $\frac{4}{3}R$. Moreover, since $\ell(Q_{R})=4\ell(R)$ and $Q_{R}\supseteq R$, we know $Q_{R}\subseteq 12 R$ and there there is $C=C(d)>0$ such that for any $Q\in \tilde{\Delta}$, there are at most $C$ many cubes $R\in \Delta$ such that $Q_{R}=Q$. Thus, for any $Q_{0}\in \Delta$ with $\ell(Q_{0})\leq \frac{1}{4}$, 
\begin{align}
\sum_{R\subseteq Q_{0}}\omega_{f}\ps{\frac{4}{3}R}^{2}|R|
 & \lec_{d} \sum_{R\subseteq Q_{0}}\omega_{f}(Q_{R})^{2}|Q|
\lec_{d}\sum_{Q\in \tilde{\Delta} \atop Q\subseteq 12Q_{0}}\omega_{f}(Q)^{2}|Q|
=\sum_{v\in \{0,\frac{1}{3}\}^{d}} \sum_{Q\in \Delta+v \atop Q\subseteq 12 Q_{0}}\omega_{f}(Q)^{2}|Q| \notag \\
 & \stackrel{\eqn{any-grid}}{\lec}_{d,\eta,c,L}\sum_{v\in \{0,\frac{1}{3}\}^{d}}|Q_{0}|
\lec_{d} |Q_{0}|.
\label{e:sum4/3R}
\end{align}
Note that this holds for any $\eta$-quasisymmetric embedding of $\bR^{d}$ into $\bR^{D}$ whose image has BPBI($\kappa,L$), and since $\omega_{f}$ is dilation and translation invariant, we know that \eqn{sum4/3R} holds for any $Q_{0}\in\Delta$, not just those with $\ell(Q_{0})\leq \frac{1}{4}$. We can now employ \Lemma{dyadic-version} to finish the theorem, at least if we assume \Lemma{w-JN} holds.

\end{proof}

\subsection{Proof of \Lemma{w-JN}}

We now devote ourselves to the proof of \Lemma{w-JN}.

\begin{proof}[Proof of \Lemma{w-JN}]
Note that if $f(\bR^{d})$ has BPBI($\kappa,L$), then so does $f_{v}(\bR^{d})$ (where $f_{v}(x):=f(x+v)$), so without loss of generality, we will assume $v=0$, since the other cases have the same proof.

Let $Q_{0}\in \Delta$. By \Lemma{ran-dom}, we know there is 
\[E'\subseteq B_{0}:=B(x_{Q_{0}},\ell(Q_{0})/2)\subseteq Q_{0}\] 
and $g:f(E')\rightarrow \bR^{d}$ $L$-bi-Lipschitz such that $\cH^{d}(f(E'))\geq c(\diam f(B_{0}))^{d}$. By \Lemma{one-piece}, there is $E_{0}\subseteq B_{0}$ such that 
\begin{equation}
\frac{|E_{0}|}{|Q_{0}|}\gec_{d}\frac{|E_{0}|}{|B_{0}|}\gec_{\eta,d,L,c}1
\label{e:e0/q0}
\end{equation}
and $\ps{\frac{\diam g\circ f(E_{0})}{\diam E_{0}}}^{-1}g\circ f|_{E_{0}}$ is bi-Lipschitz and hence  $\ps{\frac{\diam  f(E_{0})}{\diam E_{0}}}^{-1} f|_{E_{0}}$  is $M$-bi-Lipschitz for some $M=M(d,\eta,L,c)>0$. Recall that $\omega_{f}$ (and hence \eqn{sumomega(R)}) are invariant under a scaling of $f$ in its image, thus we may assume $\diam f(E_{0})/\diam E_{0}=1$ without loss of generality, so that $f$ is $M$-bi-Lipschitz on $E_{0}$.

The following theorem of MacManus tells us that we can extend $f|_{E_{0}}$ to a bi-Lipschitz homeomorphism of $\bR^{2D}\rightarrow \bR^{2D}$.

\begin{theorem}(\cite{MacManus-bilip-extensions})
If $K$ is a compact subset of $\bR^{D}$ and $\Psi$ is an $M$-bi-Lipschitz map of $K$ into $\bR^{D}$, then $\Psi$ has an extension to a $CM^{2}$ bi-Lipschitz map from $\bR^{2D}$ onto itself, where $C$ is some universal constant.
\end{theorem}

Viewing $E_{0}$ as being a subset of $\bR^{D}$, we can extend $f$ from the set $E_{0}$ to a $CM^{2}$ bi-Lipschitz self-map of $\bR^{2D}$. Let $F:\bR^{d}\rightarrow \bR^{2D}$ be the restriction of this extension to $\bR^{d}$, so that $F$ is $CM^{2}$ bi-Lipschitz embedding of $\bR^{d}$ into $\bR^{2D}$ that agrees with $f$ on $E_{0}$. By \Lemma{very-serious}, we may find $E\subseteq E_{0}$ compact such that $|E|\geq \frac{1}{2}|E_{0}|\gec_{\eta,d,L,c}|Q_{0}|$ and $E$ is $\hat{C}$-serious for some constant $\hat{C}$ depending on the constants in \eqn{e0/q0}. We'll show $E$ is the desired set such that \eqn{sumomega(R)} holds.

For $Q\in \Delta$, let $A_{Q}$ be the orthogonal projection of $F|_{Q}\in L^{2}(Q)$ onto the finite dimensional subspace of $L^{2}(Q)$ consisting of linear $\bR^{D}$ valued functions. Then
\[\Omega_{F}(Q)=\ps{\avint_{Q}\ps{\frac{|F-A_{Q}|}{\diam Q}}^{2}}^{\frac{1}{2}}.\]
Let $Q_{1},Q_{2}\subseteq Q$ be such that $Q_{j}^{2}=Q$ and $\dist(Q_{1},Q_{2})=\frac{1}{2}\diam Q$. Then there are $x_{j}\in Q_{j}$ such that 
\begin{align*}
 |F(x_{j})-A_{Q}(x_{j})|^{2} 
 & \leq \avint_{Q_{j}} |F-A_{Q}|^{2}\leq 2^{2d}\avint_{Q}|F-A_{Q}|^{2} 
 =2^{2d}\ps{\Omega_{F}(Q)\diam Q}^{2}.
 \end{align*}
Then
\begin{multline}
\frac{1}{CM^{2}} 
 \leq \frac{|F(x_{1})-F(x_{2})|}{|x_{1}-x_{2}|}
 \leq \frac{|F(x_{1})-A_{Q}(x_{1})|+|A_{Q}(x_{1})-A_{Q}(x_{2})|+|A_{Q}(x_{2})-F(x_{2})|}{\frac{1}{2}\diam Q}\\
 \leq 2^{d+2}\Omega_{F}(Q)+2|A_{Q}'|.
 \label{e:fxj-axj}
\end{multline}
Set 
\[\cB=\ck{Q\in \Delta: \Omega_{F}(Q)\geq \frac{1}{2^{d+3}CM^{2}}}\]
and set 
\[\cG_{E}=\{Q\in\Delta\backslash \cB: Q\subseteq Q_{0}, Q\cap E\neq\emptyset\}\]
so that \eqn{fxj-axj} implies
\begin{equation}
|A_{Q}'|\geq \frac{1}{4CM^{2}}\mbox{ for all }Q\in \cG_{E}
\label{e:AQ'>1/M}
\end{equation}

We now begin the process of showing \eqn{sumomega(R)} holds for the set $E$:

\begin{align}
\sum_{Q\subseteq Q_{0} \atop Q\cap E\neq\emptyset} \omega_{f}(Q)^{2}|Q|
 & \leq \sum_{Q\subseteq Q_{0} \atop Q\in \cB}\omega_{f}(Q)^{2}|Q|
+\sum_{Q\in \cG_{E}} \omega_{f}(Q)^{2}|Q| \notag \\
&  \leq \sum_{Q\subseteq Q_{0} \atop Q\in \cB}\omega_{f}(Q)^{2}|Q|+ \sum_{Q\in \cG_{E}}  \avint_{Q}\ps{\frac{|f-A_{Q}|}{|A_{Q}'|\diam Q}}^{2}|Q| \notag \\
&  \leq  \sum_{Q\subseteq Q_{0} \atop Q\in \cB}\omega_{f}(Q)^{2}|Q|+\sum_{Q\in \cG_{E}} \int_{Q} \ps{\frac{|F-A_{Q}|+|f-F|}{|A_{Q}'|\diam Q}}^{2}  \notag \\
&  \leq \sum_{Q\subseteq Q_{0} \atop Q\in \cB}\omega_{f}(Q)^{2}|Q|+2 \sum_{Q\in \cG_{E}} \frac{\Omega_{F}(Q)^{2}}{|A_{Q}'|^{2}}|Q|    + 2\sum_{Q\in\cG_{E}} \int_{Q}\ps{\frac{|f-F|}{|A_{Q}'|\diam Q}}^{2}.
\label{e:sumQcapE}
\end{align}
We'll estimate the three summands separately, starting with the first.

Let $\phi_{Q_{0}}$ be a smooth bump function such that
\[\one_{3Q_{0}}\leq \phi_{Q_{0}}\leq \one_{4\phi_{Q_{0}}} \mbox{ and } |\d^{\alpha}\phi_{Q_{0}}|\lec_{d,\alpha}\ell(Q_{0})^{-|\alpha|}.\] 
Then by Dorronsoro's theorem, and since $\omega_{f}(Q)\leq 1$ for all $Q$,

\begin{align}
\sum_{Q\in Q_{0}  \atop Q\in \cB} & \omega_{f}(Q)^{2}|Q|
 \leq \sum_{Q\in Q_{0} \atop Q\in \cB}|Q|
\leq \sum_{Q\in Q_{0} \atop Q\in \cB} \ps{2^{d+3}CM^{2}}^{2}\Omega_{F}(Q)^{2}|Q| \notag \\
& \lec_{d,M} \sum_{Q\subseteq Q_{0}}\Omega_{F}(Q)^{2}|Q|
 = \sum_{Q\subseteq Q_{0}}\Omega_{\phi_{Q_{0}}(F-F(x_{Q_{0}}))}(Q)^{2}|Q| \notag \\
& \lec_{D} ||\grad (\phi_{Q_{0}} (F-F(x_{Q_{0}})))||_{2}^{2} \notag \\
& \leq || \grad \phi_{Q_{0}} (F-F(x_{Q_{0}})) + \phi_{Q_{0}}\grad (F-F(x_{Q_{0}}))||_{2}^{2} \notag \\
& \lec_{d} \frac{1}{\ell(Q_{0})^{2}} \int_{4Q_{0}}(F-F(x_{Q_{0}}))^{2}+\int_{4Q_{0}}|\grad (F-F(x_{Q_{0}}))|^{2}\notag  \\ 
& \leq \frac{1}{\ell(Q_{0})^{2}}\int_{4Q_{0}}(CM^2\diam 4Q_{0})^{2}+\int_{4Q_{0}}(CM^{2})^{2} 
 \lec_{d,M} |Q_{0}|.
 \label{e:sumomegaf1}
\end{align}

For the second summand in \eqn{sumQcapE}, we use \eqn{AQ'>1/M} and Dorronsoro's theorem to estimate

\begin{equation}
\sum_{Q\in \cG_{E}} \frac{\Omega_{F}(Q)^{2}}{|A_{Q}'|^{2}}|Q|
 \stackrel{\eqn{AQ'>1/M}}{ \leq} \sum_{Q\in \cG_{E}}(4CM^{2})^{2}\Omega_{F}(Q)^{2}|Q| 
 \lec_{d,M}\sum_{Q\subseteq Q_{0}}\Omega_{F}(Q)^{2}
 \stackrel{\eqn{sumomegaf1}}{\lec_{D,M}} |Q_{0}|.
 \label{e:sumomegaf2}
\end{equation}

Now we focus on the final sum in \eqn{sumQcapE}. For any $Q\subseteq Q_{0}$ such that $Q\cap E\neq\emptyset$, if $x\in Q\cap E$, then there is $y\in E_{0}$ such that $\hat{C}^{-1}\diam Q \leq |x-y|\leq \diam Q$ (because $E$ is a $\hat{C}$-serious subset of $E_{0}$). Hence, if $z\in Q$ is such that $|f(x)-f(z)|\geq \frac{1}{2}\diam f(Q)$,
\begin{align}
 \diam f(Q)
& \leq 2|f(x)-f(z)| 
=2\frac{|f(x)-f(z)|}{|f(x)-f(y)|}|f(x)-f(y)|
\notag \\
& \leq 2\eta \ps{\frac{|x-z|}{|x-y|}}|f(x)-f(y)|  \notag \\
& 
 \leq 2\eta\ps{\frac{\diam Q}{\hat{C}^{-1}\diam Q}} |F(x)-F(y)| \notag \\
& \leq 2\eta(\hat{C})CM^{2}|x-y|  \leq 2\eta(\hat{C})CM^{2}\diam Q.
\label{e:f(Q)<MQ}
\end{align}

We will require some estimates on the H\"older continuity of $f$.

\begin{corollary}
Let $f:\bR^{d}\rightarrow \bR^{D}$ be $\eta$-quasisymmetric and $K\subseteq \bR^{d}$ a bounded set. Then there are constants $C>0$ and $\alpha\in (0,1)$, depending only on $\eta$, such that for all $x,y\in K$ distinct,
\[ \frac{1}{2C}  \ps{\frac{|x-y|}{\diam K}}^{\frac{1}{\alpha}}\leq \frac{|f(x)-f(y)|}{\diam f(K)}\leq 2^{\alpha}C\ps{\frac{|x-y|}{\diam K}}^{\alpha}.\]
\label{c:holder}
\end{corollary}
We will prove this in Section \ref{s:holderproof} in the appendix. 

Let $\{Q_{j}\}$ be the Whitney cube decomposition for $E^{c}$, comprised of those maximal dyadic cubes $Q_{j}\subseteq E^{c}$ for which $3Q_{j}\cap E=\emptyset$. Then it is not too hard to show that, for each $j$,
\begin{equation}
 \ell(Q_{j}) \leq \dist(x,E)\leq 4\diam Q_{j} \mbox{ for all }x\in Q_{j}.
 \label{e:whitney}
 \end{equation}

Let $Q\subseteq Q_{0}$. For $x\in Q$, let $x'$ denote a point in $E$ such that $|x-x'|=\dist(x,E)$, and pick $Q_{j}$ containing $x$. By \Corollary{holder}, and since $f=F$ on $E$, there are constants $\alpha\in (0,1)$ and $C_{\eta}>0$ such that 
\begin{align}
|f(x) & -F(x)|
 \leq |f(x)-f(x')|+|f(x')-F(x')|+|F(x')-F(x)|\\
& \leq 2^{\alpha}C_{\eta}\ps{\frac{|x-x'|}{\diam Q}}^{\alpha}\diam f(Q) + 0 + CM^{2}|x-x'| \notag \\
& \leq  C_{\eta}\ps{2\frac{\dist(x,E)}{\ell(Q)}}^{\alpha}\diam f(Q)+ CM^{2}\dist(x,E) \notag \\
& =C_{\eta}\ps{2\frac{\dist(x,E)}{\ell(Q)}}^{\alpha}\diam f(Q)+ CM^{2}\frac{\dist(x,E)}{\ell(Q)} \ell(Q) \notag \\
& \stackrel{\eqn{whitney}}{\leq}  \ps{\frac{ 4\sqrt{d}\ell(Q_{j})}{\ell(Q)}}^{\alpha}C_{\eta}\diam f(Q)+CM^{2}\frac{8\sqrt{d}\ell(Q_{j})}{\ell(Q)}\ell(Q) \notag \\
& \stackrel{\eqn{f(Q)<MQ} }{\leq} 8\sqrt{d} \ps{\frac{\ell(Q_{j})}{\ell(Q)}}^{\alpha}(2C_{\eta}\eta(\hat{C})CM^{2}+CM^{2})\ell(Q) \notag \\
& = (1+2C_{\eta}\eta(\hat{C}))8CM^{2}\ps{\frac{\ell(Q_{j})}{\ell(Q)}}^{\alpha} \ell(Q).
\label{e:f-F<qj/q^a}
\end{align}

Before proceeding, we will need the following geometric lemma.

 \begin{lemma}
Let $\alpha>0$, $K\subseteq Q_{0}\in\Delta(\bR^{d})$ be any compact subset, and $\{Q_{j}\}$ be a Whitney decomposition for $K^{c}$. For $Q\subseteq Q_{0}$, define
\[\lambda_{K,\alpha}(Q):=\sum_{Q_{j}\subseteq Q}\ps{\frac{\ell(Q_{j})}{\ell(Q)}}^{d+\alpha}.\]
Then, for all $Q\subseteq Q_{0}$, 
\begin{equation}
\lambda_{K,\alpha}(Q)\leq \frac{|Q\backslash K|}{|K|}\leq 1
\label{e:lambda<1}
\end{equation}
and
\begin{equation}
\sum_{Q\subseteq Q_{0}} \lambda_{K,\alpha}(Q)|Q| \leq  \frac{1}{1-2^{-\alpha}} |Q_{0}\backslash K|.
\label{e:lambda-sum}
\end{equation}
\end{lemma}

\begin{proof}
Fix $\alpha>0$ and set $\lambda=\lambda_{K,\alpha}$. For the first part of the lemma, observe that since the $Q_{j}$ are disjoint and $\frac{\ell(Q_{j})}{\ell(Q)}\leq 1$ if $Q_{j}\subseteq Q$,
\[
\lambda(Q)
  =\sum_{Q_{j}\subseteq Q} \ps{\frac{\ell(Q_{j})}{\ell(Q)}}^{n+\alpha} 
  \leq \sum_{Q_{j}\subseteq Q} \ps{\frac{\ell(Q_{j})}{\ell(Q)}}^{n} 
  = \frac{1}{|Q|}\sum_{Q_{j}\subseteq Q}|Q_{j}| 
  = \frac{1}{|Q|}{|Q\backslash K|}\leq 1.
\]
 Now we show \eqn{lambda-sum}. By Fubini's theorem,
\begin{align*}
\sum_{Q\subseteq Q_{0}}\lambda(Q)|Q|
& =\sum_{Q\subseteq Q_{0}}\sum_{Q_{j}\subseteq Q} \ps{\frac{\ell(Q_{j})}{\ell(Q)}}^{n+\alpha}|Q|\\
& =\sum_{Q_{j}\subseteq Q_{0}}\ell(Q_{j})^{n+\alpha}\sum_{Q_{j}\subseteq Q\subseteq Q_{0}} \frac{|Q|}{\ell(Q)^{n+\alpha}}\\
& =\sum_{Q_{j}\subseteq Q_{0}}\ell(Q_{j})^{n+\alpha}\sum_{Q_{j}\subseteq Q\subseteq Q_{0}} \ell(Q)^{-\alpha}\\
& =\sum_{Q_{j}\subseteq Q_{0}}\ell(Q_{j})^{n+\alpha}\sum_{j=0}^{\log_{2}\frac{\ell(Q_{0})}{\ell(Q_{j})}}\ell(Q_{j})^{-\alpha}2^{-j\alpha} \\
&  \leq \sum_{Q_{j}\subseteq Q_{0}}\ell(Q_{j})^{n}\frac{1}{1-2^{-\alpha}} = \frac{1}{1-2^{-\alpha}} |Q_{0}\backslash K|.
\end{align*}
\end{proof}

We continue with the proof. Since $|f(x)-F(x)|=0$ on $E$ and $E^{c}=\bigcup Q_{j}$ since $E$ is closed,
\begin{align}
\sum_{Q\in \cG_{E}} \int_{Q} & \ps{\frac{|f(x)-F(x)|}{\diam Q} }^{2}dx
 = \sum_{Q\in \cG_{E}} \sum_{Q_{j}\subseteq Q} \int_{Q_{j}}\ps{\frac{|f(x)-F(x)|}{\diam Q}}^{2}dx \notag \\
& \stackrel{\eqn{f-F<qj/q^a}}{\lec}_{\eta,d,M} \sum_{Q\in \cG_{E}} \sum_{Q_{j}\subseteq Q} \ps{\frac{\ell(Q_{j})}{\ell(Q)}}^{2\alpha} |Q_{j}| 
 = \sum_{Q\in \cG_{E}} \ \sum_{Q_{j}\subseteq Q} \ps{\frac{\ell(Q_{j})}{\ell(Q)}}^{d+2\alpha} |Q| \notag \\
&  =\sum_{Q\in \cG_{E}} \lambda_{E,2\alpha}(Q)|Q|
  \stackrel{\eqn{lambda-sum}}{\lec}_{\alpha}  |Q_{0}|.
 \label{e:sumomegaf3}
\end{align}
and this bounds the third sum in \eqn{sumQcapE}. 

Combining \eqn{sumQcapE}, \eqn{sumomegaf1}, \eqn{sumomegaf2}, and \eqn{sumomegaf3}, we obtain 
\begin{align*}
\sum_{Q\subseteq Q_{0} \atop Q\cap E'\neq\emptyset} \omega_{f}(Q)^{2}|Q|
\lec_{d,M,\eta}|Q_{0}|.
\end{align*}
Finally, we recall that $M$ and $\alpha$ depend on $\eta,c,D,$ and $L$, and this finishes the proof.

\end{proof}

%
%
%

\section{Appendix}
\label{s:appendix}

\subsection{H\"older estimates: The proof of \Corollary{holder}}
\label{s:holderproof}

\begin{lemma}(\cite[Theorem 11.3]{Heinonen}) An $\eta$-quasisymmetric embedding $f:\Omega\rightarrow \bR^{D}$, where $\Omega\subseteq \bR^{d}$ is connected, is $\tilde{\eta}$-quasisymmetric with $\tilde{\eta}$ of the form
\begin{equation}
\tilde{\eta}=C\max\{t^{\alpha},t^{\frac{1}{\alpha}}\}
\label{e:t^at^1/a}
\end{equation}
where $C\geq 1$ and $\alpha\in (0,1]$ depend only on $\eta$.
\label{l:holder}
\end{lemma}

The original lemma is stated for $A$-uniformly perfect spaces (metric spaces $X$ such that $B(x,r)\backslash B(x,r/A)$ is nonempty for all $x\in X$ and $r>0$), and the constants $C$ and $\alpha$ depend also on the constant associated with being uniformly perfect, but connected sets happen to be uniformly perfect with $A=1$ (see the beginning of Chapter 11 of \cite{Heinonen} for a discussion and the original statement). As a corollary, we have the following:

\newtheorem*{holder}{\Corollary{holder}}

\begin{holder}
Let $f:\bR^{d}\rightarrow \bR^{D}$ be $\eta$-quasisymmetric and $E\subseteq \bR^{d}$ a bounded set. Then for all $x,y\in E$ distinct,
\[ \frac{1}{2C}  \ps{\frac{|x-y|}{\diam E}}^{\frac{1}{\alpha}}\leq \frac{|f(x)-f(y)|}{\diam f(E)}\leq 2^{\alpha}C\ps{\frac{|x-y|}{\diam E}}^{\alpha}\]
where $\alpha$ and $C$ are as in \Lemma{holder}.
\end{holder}

\begin{proof}
Let $x,y\in E$. Pick $y'\in E$ so that $|x-y'|\geq \max\{\frac{1}{2}\diam E,|x-y|\}$. Then
\[
\frac{|f(x)-f(y)|}{\diam f(E)} 
\leq \frac{|f(x)-f(y)|}{|f(x)-f(y')|} 
\leq \eta\ps{\frac{|x-y|}{|x-y'|}}
 \leq C\ps{\frac{|x-y|}{|x-y'|}}^{\alpha}
  \leq C2^{\alpha}\ps{\frac{|x-y|}{\diam E}}^{\alpha}.
\]
Now, let $y''\in E$ be such that $|f(x)-f(y'')|\geq \frac{1}{2}\diam f(E)$. Then
\begin{equation}
\frac{\diam f(E)}{|f(x)-f(y)|}
\leq 2\frac{|f(x)-f(y'')|}{|f(x)-f(y)|}
\leq 2\eta\ps{\frac{|x-y''}{|x-y|}}.
\label{e:f(E)/f(x)-f(y)}
\end{equation}
If $|x-y''|\leq |x-y|$, then
\[\eqn{f(E)/f(x)-f(y)} \stackrel{\eqn{t^at^1/a}}{\leq} 2C\ps{\frac{|x-y''|}{|x-y|}}^{\alpha}
\leq 2C\ps{\frac{\diam E}{|x-y|}}^{\alpha} 
\leq 2C\ps{\frac{\diam E}{|x-y|}}^{\frac{1}{\alpha}}.\]
If $|x-y''|>|x-y|$, then
\[\eqn{f(E)/f(x)-f(y)} \stackrel{\eqn{t^at^1/a}}{\leq} 2C\ps{\frac{|x-y''|}{|x-y|}}^{\frac{1}{\alpha}}
\leq 2C\ps{\frac{\diam E}{|x-y|}}^{\frac{1}{\alpha}}.
\]
Hence, in either case,
\[\frac{\diam f(E)}{|f(x)-f(y)|}\leq 2C\ps{\frac{\diam E}{|x-y|}}^{\frac{1}{\alpha}}\]
which proves the lemma.
\end{proof}

\subsection{Proof of \Lemma{epsilondelta}}
\label{s:epsilondelta-proof}

\newtheorem*{epsilondelta}{\Lemma{epsilondelta}}
\begin{epsilondelta}
Let $\delta>0$.  If $f$ is $\eta$-quasisymmetric on a cube $Q\subseteq\bR^{d}$, then there is $\ve_{1}=\ve_{1}(\eta,d,\delta)>0$ so that if
\begin{equation}\avint_{Q}\frac{|f-A|}{|A'|\diam Q}<\ve_{1}.
\label{e:f-a<e}
\end{equation}
then
\[|f(x)-A(x)| <\delta |A'|\diam Q.\]
Moreover, 
\[ (1-2\sqrt{d}\delta)|A'|\ell(Q) \leq  \diam f(Q)  \leq (1+2\sqrt{d}\delta)|A'|\diam Q.\]
\end{epsilondelta}

\begin{proof}
Fix $K>0$ and let
\[E_{K}=\{x\in Q: |f(x)-A_{Q}(x)|\leq K\ve_{1} |A_{Q}'|\diam Q\} \]
so that by Chebichev's inequality,
\[|Q\backslash E_{K}|\leq \frac{1}{K}|Q|.\]
Note that if $B=B(x,r)$ is any ball contained in $E_{K}^{c}$, not necessarily contained in $Q$ but with center $x\in Q\backslash E_{K}$, then at least $\frac{1}{2^{d}}$ percent of it is contained in $Q$, and so
\[ w_{d}r^{d}=|B|\leq 2^{d} |Q\backslash E_{K}|\leq \frac{2^{d}}{K}|Q|,\]
so that 
\[r\leq 2(kw_{d})^{-d}\ell(Q).\]
Pick $K=w_{d}^{-1}2^{d}\ve_{1}^{-\frac{1}{d}}$ so that $r\leq \ve\ell(Q)$. Then 
\[\sup_{x\in Q\backslash E_{K}}\dist(x,E_{K}) \leq \ve_{1} \ell(Q).\]

For $x\in Q\backslash E_{K}$, let $x'\in E_{K}$ be such that $|x-x'|=\dist(x,E_{K})$. Then by \Corollary{holder}, \begin{multline*}
|f(x)-A(x)|
 \leq |f(x)-f(x')|+|f(x')-A(x')|+|A(x')-A(x)|\\
 \leq 2^{\alpha}C\ps{\frac{|x-x'|}{\diam Q}}^{\alpha}\diam f(Q)+K\ve_{1}|A'|\diam Q+|A'||x-x'|\\
 \leq 2^{\alpha}C \ps{\frac{\ve_{1}\ell(Q)}{\diam Q}}^{\alpha} \diam f(Q)+  w_{d}2^{d} \ve_{1}^{1-\frac{1}{d}} |A'| \diam Q+2|A'| \ve \ell(Q)\\
 \leq  2^{\alpha}C\ve_{1}^{\alpha} \diam f(Q)+ (w_{d}2^{d}+2)\ve_{!}^{1-\frac{1}{d}} |A'|\diam Q.
\end{multline*}

We claim that $|A'| \diam Q \geq \frac{1}{4}\diam f(Q)$ if $\ve_{1}$ is small enough. If $|A'| \diam Q < \frac{1}{4}\diam f(Q)$, pick $x_{0}\in Q$ so that $|f(x_{0})-A(x_{0})|\leq \ve|A'|\diam Q$ and pick $x\in Q$ so that $|f(x)-f(x_{0})|\geq \frac{1}{2}\diam f(Q)$. Then
\begin{align*}
|f(x)-A(x)|& \geq |f(x)-f(x_{0})|-|f(x_{0})-A(x_{0})|-|A(x_{0})-A(x)|\\
& \geq \frac{1}{2}\diam f(Q)-\ve_{1}|A'|\diam Q-|A'|\diam Q\\
& \geq \ps{\frac{1}{2}-\frac{(1+\ve_{1})}{4}}\diam f(Q)\geq \frac{1}{8}\diam f(Q)
\end{align*}
if $\ve_{1}<\frac{1}{2}$. However,
\begin{align*} 
|f(x)-A(x)| 
& \leq 2^{\alpha}C\ve_{1}^{\alpha} \diam f(Q)+ (w_{d}2^{d}+2)\ve_{1}^{1-\frac{1}{d}} |A'|\diam Q\\
& < (2^{\alpha}C\ve_{1}^{\alpha} + \frac{1}{4}(w_{d}2^{d}+2)\ve_{1}^{1-\frac{1}{d}})\diam f(Q) 
<\frac{1}{8}\diam f(Q)
\end{align*}
if $\ve>0$ is small enough, which is a contradiction. Thus, $|A'| \diam Q \geq \frac{1}{4}\diam f(Q)$, so that 
\[|f(x)-A(x)|\leq (2^{\alpha+2}C\ve^{\alpha} + (w_{d}2^{d}+2)\ve^{1-\frac{1}{d}}) |A'|\diam Q <\delta|A'|\diam Q\]
if $\ve>0$ is picked small enough. 

For the last part of the lemma, let $x,y\in Q$ be such that $\diam f(Q)=|f(x)-f(y)|$. Then
\begin{align*}
\diam f(Q)
& =|f(x)-f(y)|
\leq |A(x)-A(y)|+2\delta |A'|\diam Q\\
& \leq |A'||x-y|+2\delta\diam Q
  \leq (1+2\delta\sqrt{d})|A'|\diam Q.
\end{align*}
 For the opposite inequality, we may assume without loss of generality that $x_{Q}=0$. Pick $x\in \d B(x_{Q},\frac{\ell(Q)}{2})$ so that $|A(x)-A(-x)|=|A'|\diam B_{Q}$. Then
 \begin{align*}
 \diam f(B_{Q})
&  \geq |f(x)-f(y)|\geq |A(x)-A(y)|-2\delta \diam Q\\
&   =|A'|(1-2\delta\sqrt{d})\diam B_{Q}=|A'|(1-2\delta\sqrt{d})\ell(Q).
 \end{align*}
\end{proof}

\subsection{Dorronsoro's Theorem}
\label{s:dorronsoro}

Here we prove the following special case of Dorronsoro's theorem. We prove a more general version than what is stated in the introduction by showing we can replace $\Omega_{f}$ with a general $L^{p}$-type integral for $p\in [1,2]$ and obtain the same result. Throughout the paper, however, we only use the $p=2$ case and write $\Omega_{f}=\Omega_{2,f}$ for short.

\newtheorem*{dorronsoro}{\Theorem{dorronsoro}}

\begin{dorronsoro}[\cite{Dorronsoro}]
Let $f\in L^{2}(\bR^{d})$. For $x\in \bR^{d}$, $r>0$, and $p\in[1,2]$, define
\[\Omega_{p,f}(x,r)=\inf_{A} \ps{\avint_{B(x,r)} \ps{\frac{|f-A|}{r}}^{p}}^{\frac{1}{p}}\]
where the infimum is over all affine maps $A:\bR^{d}\rightarrow \bR$. Then $f\in W^{1,2}(\bR^{d})$ if and only if 
\[\Omega_{p}(f):=\int_{\bR^{d}}\int_{0}^{\infty}\Omega_{p,f}(x,r)^{2}\frac{dr}{r}dx<\infty,\]
in which case, 
\begin{equation}
||\grad f||_{2}^{2}\lec_{d} \Omega_{1}(f)\leq \Omega_{q}(f) \leq \Omega_{2}(f)\lec_{d}||\grad f||_{2}
\label{e:omegachain}
\end{equation}
 for all $q\in [1,2]$, so in particular, $||\grad f||_{2}^{2}\sim_{d} \Omega_{q}(f)$ for $q\in [1,2]$.
\label{t:dorronsoro-p}
\end{dorronsoro}

We should mention, of course, that the original result is far more general; in particular, Dorronsoro gives a characterization of the fractional Sobolev spaces $W^{\alpha,p}$ for all $\alpha>0$ and $p\in (1,\infty)$. We provide a proof of this special case for the interested reader, since the proof we supply avoids the interpolation theory and reference chasing in \cite{Dorronsoro}; only the basic properties of Sobolev spaces and the Fourier transform are needed. This proof is well known, but not completely written down anywhere to the author's knowledge (although hints at the proof are alluded to in \cite{David-park-city}); part of it is also explained in \cite{Christ}.

\begin{proof}

{\it Step 1:} We first show $||\grad f||_{2}^{2}\lec_{d} \Omega_{1}(f)$ supposing that $f\in W^{1,2}(\bR^{d})$ (we will show later that $\Omega_{1}(f)<\infty$ implies an $L^{2}$ function $f$ is actually in $W^{1,2}$, but we'll start with this case). Let $\phi$ be a radially symmetric $C^{\infty}$ function supported in $B(0,1)$ such that $\int\phi=1$. Set $\psi(x)=\phi(x)-2^{d}\phi(2x)$, so that it is also supported in $B(0,1)$. Then $\int \psi A=0$ for any affine function $A$. For $r>0$, set $\psi_{r}(x)=r^{-d}\psi(r^{-1}x)$. Then,
\begin{align*}
|\grad f*\psi_{r}(x)|
& =|\grad(f-A)*\psi_{r}(x)| 
= |(f-A)* \grad\psi_{r}(x)|\\
& \leq \av{\int_{B(x,r)}|f(y)-A(y)|  r^{-d-1}\grad\psi(r^{-1}(x-y))dy}\\
& \leq w_{d}||\grad\psi||_{\infty} \avint_{B(x,r)}\frac{|f-A|}{r}\\
\end{align*}
and infimizing over all affine maps $A$ gives
\begin{equation}
|\grad f*\psi_{r}(x)|\leq w_{d}||\grad\psi||_{\infty}\Omega_{1,f}(x,r).
\label{e:f*psi<Omega}
\end{equation}
Observe that by Fubini's theorem and Plancharel's theorem,
\begin{align*}
\int_{0}^{\infty}\int_{\bR^{d}} |\grad f* \psi_{r}(x)|^{2}dx\frac{dr}{r}
& =\int_{0}^{\infty}\int_{\bR^{d}}|\widehat{\grad f}(\xi)|^{2}|\hat{\psi}(r\xi)|^{2}d\xi \frac{dr}{r} \\
& =\int_{\bR^{d}}|\widehat{\grad f}(\xi)|^{2} \ps{\int_{0}^{\infty} |\hat{\psi}(r\xi)|^{2}\frac{dr}{r} } d\xi.
\end{align*}
Since $\psi$ is radially symmetric, so is $\hat{\psi}$, thus, if $e_{1}\in\bR^{d}$ denotes the first standard basis vector,
\[
\int_{0}^{\infty} |\hat{\psi}(r\xi)|^{2}\frac{dr}{r}
=\int_{0}^{\infty} |\hat{\psi}(r|\xi|e_{1})|^{2}\frac{dr}{r}
=\int_{0}^{\infty} |\hat{\psi}(re_{1})|^{2}\frac{dr}{r}=:c_{\psi}<\infty.\]
The reason this is finite is because  $\hat{\psi}$ is a Schwartz function, $\hat{\psi}(0)=\int\psi=0$, and $\hat{\psi}$ is differentiable at zero, so $|\psi(\xi)|\lec \frac{|\xi|}{|1+|\xi|^{3}}$. Thus,
\begin{align*}
\omega_{d}||\grad \psi||_{\infty}\int_{\bR^{d}}\int_{0}^{\infty}\Omega_{1,f}(x,r)^{2}\frac{dr}{r} 
& \geq \int_{0}^{\infty}\int_{\bR^{d}} |\grad f* \psi_{r}(x)|^{2}dx\frac{dr}{r}\\
&  =c_{\psi}\int_{\bR^{d}}|\widehat{\grad f}(\xi)|^{2}d\xi = c_{\psi}\int_{\bR^{d}}|\grad f(x)|^{2}dx.
\end{align*}
This proves the first inequality in \eqn{omegachain}.

{\it Step 2:} Now just suppose $f\in L^{2}(\bR^{d})$, we'll show that $\Omega_{1}(f)<\infty$ implies $f$ has a weak gradient $\grad f$ that is in $L^{2}$. Let $\phi$ be a nonnegative $C^{\infty}$ bump function supported in $B(0,1)$ with $\int \phi=1$. Observe that since $||\hat{\phi}||_{\infty}\leq \int\phi=1$, we have
\[
||f*\phi_{t}||_{2}=||\hat{f}\hat{\phi_{t}}||_{2}\leq ||\hat{f}||_{2}=||f||_{2}\]
and $||\hat{\grad\phi}||_{\infty}\leq ||\grad \phi||_{1}<\infty$, so that 
\begin{align*}
||\grad (f*\phi_{t})||_{2}^{2} & =||f*\grad\phi_{t}||_{2}^{2}
 =\int_{\bR^{d}}|\hat{f}(\xi)|^{2}|\widehat{\grad\phi_{t}}(\xi)|^{2}d\xi\\
& =\int_{\bR^{d}}|\hat{f}(\xi)|^{2}t^{-2}|\widehat{\grad \phi}(t\xi)|^{2}d\xi
  \leq t^{-2}||\grad\phi||_{1}\int_{\bR^{d}}|\hat{f}(\xi)|^{2}=\frac{||f||_{2}^{2}}{t^{2}}.
\end{align*}
Thus, $f*\phi_{t}\in W^{1,2}(\bR^{d})$, and so we know that
\[|| \grad f*\phi_{t}||_{2}^{2}\lec_{d} \int_{\bR^{d}}\int_{0}^{\infty}\Omega_{1,f*\phi_{t}}(x,r)^{2}\frac{dr}{r}dx.\]
Suppose $r\geq t$. Since $\int\phi=1$, and since affine functions are harmonic, we have $\phi_{t}*A=A$ for any affine function, thus
\begin{align*}
\avint_{B(x,r)} & \frac{|f*\phi_{t}-A|}{r} 
 = \avint_{B(x,r)}\frac{|(f-A)*\phi_{t}|}{r}
  \leq \avint_{B(x,r)}\int_{\bR^{d}} \frac{|f(z)-A(z)|}{r}\phi_{t}(y-z)dzdy\\
& =\frac{|B(x,2r)|}{|B(x,r)|}\int_{\bR^{d}} \avint_{B(x,2r)}\frac{|f(z)-A(z)|}{r}\phi_{t}(y-z)dydz
  =2^{d}\avint_{B(x,2r)}  \frac{|f(z)-A(z)|}{r}dz\\
\end{align*}
and infimizing over affine maps $A$ gives
\begin{equation}
\Omega_{1,f*\phi_{t}}(x,r) \leq 2^{d}\Omega_{1,f}(x,2r) \mbox{ for }r\geq t.
\end{equation}
Now suppose $r<t$. Then by Taylor's theorem, and since $\d^{\alpha}\phi_{t}*A=0$ for any affine map $A$ and $|\alpha|\geq 1$,
\begin{align*}
\frac{\Omega_{1,f*\phi_{t}}(x,r)}{r}
& \leq \frac{r^{2} \max_{|\alpha|=2} ||\d^{\alpha} f*\phi_{t}||_{L^{\infty}(B(x,t))}}{r^{2}}\\
& =\max_{|\alpha|=2} ||f*\d ^{\alpha}\phi_{t}||_{L^{\infty}(B(x,t))}\\
&  =\max_{|\alpha|=2}||(f-A)*\d ^{\alpha}\phi_{t}||_{L^{\infty}(B(x,t))}\\
& \leq \max_{|\alpha|=2}\sup_{z\in B(x,t)}\int |f(y)-A(y)||\d^{\alpha}\phi_{t}(z-y)|dy\\
& = t^{-2}\max_{|\alpha|=2}\sup_{z\in B(x,t)}\int_{B(x,r+t)} |f(y)-A(y)||(\d^{\alpha}\phi)_{t}(z-y)|dy\\
& \leq t^{-2}\max_{|\alpha|=2}\sup_{z\in B(x,t)}\int_{B(x,2t)} |f(y)-A(y)| \frac{||\d^{\alpha}\phi||_{\infty}}{t^{d}}dy\\
& =t^{-1}w_{d} \max_{|\alpha|=2}||\d^{\alpha}\phi||_{\infty} \avint_{B(x,2t)}\frac{|f(y)-A(y)}{t}dy
\end{align*}
and infimizing over affine maps $A$ gives
\[\Omega_{1,f*\phi_{t}}(x,r)\leq \frac{r}{t}w_{d} \max_{|\alpha|=2}||\d^{\alpha}\phi||_{\infty} \Omega_{1,f}(x,2t).\]
Thus,
\begin{align*}
||\grad f*\phi_{t}||_{2}^{2}
& \lec_{d}  \int_{\bR^{d}} \ps{\int_{0}^{t} \Omega_{1,f*\phi_{t}}(x,r)\frac{dr}{r} + \int_{t}^{\infty}\Omega_{1,f*\phi_{t}}(x,r)^{2}\frac{dr}{r}}dx\\
& \lec_{d} \int_{\bR^{d}}\ps{\int_{0}^{t}\frac{r}{t^{2}}\Omega_{1,f}(x,2t)^{2}dr+\int_{t}^{\infty}\Omega_{1,f}(x,2r)^{2}\frac{dr}{r}}dx\\
& \leq \frac{1}{2}\int_{\bR^{d}}\Omega_{1,f}(x,2t)^{2}dx+\Omega_{1}(f)^{2}.\end{align*}
Since $\Omega_{1,f}(x,t)\leq 2^{d}\Omega_{1,f}(x,t+s)$ for all $s\in [0,t]$, we have
\begin{align*}
\int_{\bR^{d}}\Omega_{1,f}(x,2t)^{2}dx
& \lec \int_{\bR^{d}}\int_{2t}^{4t} \Omega_{1,f}(x,2t)^{2}\frac{dr}{r}dx
 \lec_{d} \int_{\bR^{d}}\int_{2t}^{4t}\Omega_{1,f}(x,r)^{2}\frac{dr}{r}dx\\
& \leq  \int_{\bR^{d}}\int_{0}^{\infty}\Omega_{1,f}(x,r)^{2}\frac{dr}{r}dx=\Omega_{1}(f)^{2}.
\end{align*}
Hence, $||\grad f*\phi_{t}||_{2}^{2}\lec_{d} \Omega_{1}(f)^{2}$, and since $\hat{\phi}(t\xi)\rightarrow 1$ uniformly on compact subsets of $\bR^{d}$ as $t\rightarrow 0$, we have, for any $R>0$
\begin{align*}
\int_{B(0,R)}| \hat{f}(\xi)|^{2}|\xi|^{2}d\xi
& =\lim_{t\rightarrow 0}\int_{B(0,R)} | \hat{f}(\xi)|^{2}|\xi|^{2}|\hat{\phi}(t\xi)|^{2}d\xi
 =\lim_{t\rightarrow 0}\int_{\bR^{d}} |\hat{f}(\xi)|^{2}|\xi|^{2}|\hat{\phi}(t\xi)|^{2}d\xi\\
&  =\lim_{t\rightarrow 0}\int_{\bR^{d}} |\grad f*\phi_{t}|^{2}
 \lec_{d}\int_{\bR^{d}}\int_{0}^{\infty}\Omega_{1,f}(x,r)^{2}\frac{dr}{r}=\Omega_{1}(f)^{2}.
\end{align*}
Letting $R\rightarrow\infty$, we get $\int_{\bR^{d}}|\hat{f}(\xi)|^{2}|\xi|^{2}d\xi \lec_{d}\Omega_{1}(f)^{2}$, which implies $f\in W^{1,2}(\bR^{d})$. 

{\it Step 3:} Note that the second and third inequalities in \eqn{omegachain} follow from Jensen's inequality since $\Omega_{p,f}(x,r)\leq \Omega_{q,f}(x,r)$ if $p\leq q$, and hence $\Omega_{p}(f)\leq \Omega_{q}(f)$.

{\it Step 4:}  It remains to prove the last inequality $\Omega_{2}(f)\lec_{d}||\grad f||_{2}^{2}$. To do so, we follow the hint in \cite{David-park-city}. 

Assume $f\in W^{1,2}(\bR^{d})$ and let  $\phi$ be a radially symmetric nonnegative function supported in $B(0,1)$ such that $\int \phi=1$. Define an affine map
\[A_{x,r}(y)=\phi_{r}*\grad f (x)\cdot (y-x) + f*\phi_{r}(x).\]
Then by Tonelli's theorem, change of variables, Plancharel's theorem, and the fact that for $p\in [1,2]$
\[\Omega_{2,f}(x,r)^{2}\leq \avint_{B(x,r)}\frac{|f(y)-A_{x,r}|^{2}}{r^{2}}dy,\]
we have
\begin{align}
\omega_{d}\int_{\bR^{d}} \int_{0}^{\infty} & \Omega_{2,f}(x,r)^{2}\frac{dt}{t}dx
\leq \int_{\bR^{d}}\int_{0}^{\infty}\omega_{d}\avint_{B(x,r)}\frac{|f(y)-A_{x,r}|^{2}}{r^{2}}dy \frac{dr}{r}dx \notag \\
& = \int_{\bR^{d}}\int_{0}^{\infty}\int_{B(x,r)} |f(y)-\phi_{r}*\grad f (x)\cdot (y-x) -f*\phi_{r}(x)|^{2}dy\frac{dr}{r^{d+3}}dx \notag \\
& =\int_{B(0,r)}\int_{0}^{\infty}\int_{\bR^{d}}  |f(y+x)-\phi_{r}*\grad f (x)\cdot y -f*\phi_{r}(x)|^{2}dx\frac{dr}{r^{d+3}}dy \notag \\
& =\int_{B(0,r)}\int_{0}^{\infty}\int_{\bR^{d}}  |\hat{f}(\xi)e^{-2\pi i y\cdot \xi}-\hat{\phi}(r\xi) \hat{f} (\xi)(-2\pi i y\cdot \xi) \notag \\
& \qquad -\hat{f}(\xi) \hat{\phi}(r\xi)|^{2}d\xi\frac{dr}{r^{d+3}}dy \notag \\
& =\int_{\bR^{d}} |\hat{f}(\xi)|^{2} \int_{B(0,r)} \int_{0}^{\infty} |e^{-2\pi i y\cdot \xi}-\hat{\phi}(r\xi)(-2\pi i y\cdot \xi) - \hat{\phi}(r\xi)|^{2}\frac{dr}{r^{d+3}}dyd\xi.
\label{e:int_omega<}
\end{align}
If we show
\begin{equation}
 \int_{B(0,r)} \int_{0}^{\infty} |e^{-2\pi i y\cdot \xi}-\hat{\phi}(r\xi)(-2\pi i y\cdot \xi) - \hat{\phi}(r\xi)|^{2}\frac{dr}{r^{d+3}}dyd \lec |\xi|^{2}
 \label{e:xi^2}
 \end{equation}
then the theorem will follow since
\[
 \eqn{int_omega<}<
\int_{\bR^{d}}|\hat{f}(\xi)|^{2}|\xi|^{2}d\xi = \int_{\bR^{d}}|\grad f(\xi)|^{2}\d\xi = \int_{\bR^{d}}|\grad f|^{2}.
\]
We begin proving \eqn{xi^2}.  Again, since $\phi$ is radially symmetric, so is $\hat{\phi}$, and hence $\hat{\phi}(\xi)=\hat{\Phi}(|\xi|)$ for some function $\hat{\Phi}:[0,\infty)\rightarrow [0,\infty)$. We will abuse notation and write $\hat{\Phi}(r)=\hat{\phi}(r)$. By two changes of variables (once in $r$, then in $y$),
\begin{align}
\int_{B(0,r)}&  \int_{0}^{\infty} |e^{-2\pi i y\cdot \xi}-\hat{\phi}(r|\xi|)(-2\pi i y\cdot \xi) -\hat{\phi}(r|\xi|)|^{2}\frac{dr}{r^{d+3}}dy\notag \\
& =|\xi|^{d+2}\int_{B(0,t/|\xi|)} \int_{0}^{\infty} |e^{-2\pi i y\cdot \xi}-\hat{\phi}(t)(-2\pi i y\cdot \xi) -\hat{\phi}(t)|^{2}\frac{dt}{t^{d+3}}dy \notag \\
& =|\xi|^{2}\int_{B(0,t)} \int_{0}^{\infty} |e^{-2\pi i y\cdot  \frac{\xi}{|\xi|}}-\hat{\phi}(t)(-2\pi i y\cdot  \frac{\xi}{|\xi|}) -\hat{\phi}(t)|^{2}\frac{dt}{t^{d+3}}dy \notag \\
& =|\xi|^{2} \int_{0}^{\infty}\int_{B(0,t)} |e^{-2\pi i y\cdot  \frac{\xi}{|\xi|}}(1-\hat{\phi}(t))+\hat{\phi}(t)(e^{-2\pi i y \cdot \frac{\xi}{|\xi|}}  -2\pi i y\cdot  \frac{\xi}{|\xi|}-1)|^{2}dy\frac{dt}{t^{d+3}}
\label{e:int_omega<2}
\end{align}
Since $\hat{\phi}$ is a Schwartz function and $\hat{\phi}(0)=\int\phi=1$, and 
\[\frac{d}{dt}\hat{\phi}(0)=\widehat{(-2\pi i t \phi(t))} |_{t=0}=-\int 2\pi i t\phi(t)dt=0\]
since $\phi$ is radially symmetric, we thus have by Taylor's theorem
\[|e^{-2\pi i y\cdot \frac{\xi}{|\xi|}}(1-\hat{\phi}(t))| \lec \min\ck{t^2,\frac{1}{1+|t|^{2}}}.\]
Again by Taylor's theorem, since $1+a$ is the first two terms of the Taylor series for $e^{a}$, and since we always have $|y|\leq t$ in the domain of the integral, we get  
\[\av{\hat{\phi}(y)\ps{e^{-2\pi i y \cdot \frac{\xi}{|\xi|}}-2\pi i y\cdot  \frac{\xi}{|\xi|}-1}} \lec  \frac{1}{1+t^{4}}\av{-2\pi i y\cdot \frac{\xi}{|\xi|}}^{2}\leq \frac{t^{2}}{1+t^{4}}  \]
so that 
\begin{multline}
\eqn{int_omega<2}
\lec |\xi|^{2} \int_{0}^{\infty} \int_{B(0,t)} \ps{\min\ck{t,\frac{1}{1+|t|^{2}}}+\frac{t^{2}}{1+t^{4}}}^{2}dy\frac{dt}{t^{d+3}} \\
=|\xi|^{2}\int_{0}^{\infty} \ps{\min\ck{t^{2},\frac{1}{1+|t|^{2}}}+\frac{t^{2}}{1+t^{4}}}^{2} \frac{dt}{t^{3}}\lec |\xi|^{2}
\end{multline}
which proves \eqn{xi^2}.
\end{proof}

\def\cprime{$'$}

\end{document}